\documentclass[reqno, 11pt]{amsart}
\usepackage{amssymb,mathrsfs}
\usepackage[usenames, dvipsnames]{color}
\usepackage{esint}
\usepackage[colorlinks=true,linkcolor=blue,citecolor=red,urlcolor=black]{hyperref}
\usepackage{verbatim}
\usepackage{cancel}
\usepackage[margin=1in]{geometry}
\usepackage{tikz}

\usepackage{cite}
\usepackage[nobysame,non-sorted-cites,alphabetic,initials]{amsrefs}
\def\MR#1{}

\theoremstyle{plain}
\newtheorem{theorem}{Theorem}[section]
\newtheorem{lemma}[theorem]{Lemma}
\newtheorem{corollary}[theorem]{Corollary}
\newtheorem{proposition}[theorem]{Proposition}

\theoremstyle{definition}

\theoremstyle{remark}

\newcommand{\I}{\operatorname{I}}

\numberwithin{equation}{section}

\newcommand{\bR}{\mathbb{R}}

\newcommand{\bT}{\mathbb{T}}

\newcommand\cB{\mathcal{B}}

\newcommand\cS{\mathcal{S}}

\newcommand\cX{\mathcal{X}}

\makeatletter
\def\dashint{\operatorname%
{\,\,\text{\bf--}\kern-.98em\DOTSI\intop\ilimits@\!\!}}
\makeatother

\begin{document}
\title[Stability of symmetric flows]{Asymptotic stability of symmetric flows with viscous inflow boundary condition}
\author[Y. Guo]{Yan Guo}

\author[Z. Yang]{Zhuolun Yang}

\address[Y. Guo]{Division of Applied Mathematics, Brown University, 182 George Street, Providence, RI 02912, USA}
\email{yan\_guo@brown.edu}

\address[Z. Yang]{Department of Mathematics, The Ohio State University, 231 West 18th Avenue, Columbus, OH 43210, USA}
\email{yang.8242@osu.edu}

\thanks{Y. Guo was partially supported by NSF Grant DMS-2405051}

\thanks{Z. Yang was partially supported by NSF Grant DMS-2550221}



\date{\today}
\begin{abstract}
We study the two-dimensional incompressible Navier-Stokes equations in a channel $\Omega=(0,L)\times(0,H)$ with small viscosity $\varepsilon\ll1$, an $\varepsilon$-Navier slip condition on the horizontal walls, and a viscous inflow condition for the perturbation stream function. For a broad class of symmetric base profiles $u_0(y)$ vanishing on the walls, we construct an exact steady solution $(u_s,v_s)$ that is $O(\varepsilon^{1/3})$-close to the shear $(u_0,0)$. We then develop a new weighted vorticity energy method to prove uniform linear stability and exponential decay: perturbations decay exponentially in a weighted $L^2$ norm on the time scale $O(\varepsilon^{-1/3})$. In the short-channel regime $L\ll1$, the method yields nonlinear asymptotic stability with threshold $O(\varepsilon^{2/3})$. In the long-channel regime, assuming concavity together with a spectral condition, we introduce a quantity \textit{Rayleigh vorticity} to control the non-favorable terms and obtain nonlinear stability with threshold
$O(\varepsilon^{5/6+})$.
\end{abstract}

\maketitle
\tableofcontents


\section{Introduction and main results}


\subsection{The problem setup}

Let $\Omega := (0,L) \times (0,H)$ denote a two-dimensional channel. We consider the incompressible Navier-Stokes equations
\begin{equation} \label{NS}
\left\{
\begin{aligned}
u_t + u u_x + v u_y + p_x - \varepsilon \Delta u &= 0,\\
v_t + u v_x + v v_y + p_y - \varepsilon \Delta v &= 0,\\
u_x+ v_y &= 0,
\end{aligned}
\right. \quad \mbox{in}~\Omega \times (0,\infty),
\end{equation}
with an $\varepsilon$-Navier boundary condition:
\begin{equation}\label{Navier_boundary_condition_new}
\left\{\begin{aligned}
u(x,0) - A\varepsilon^{1/3} u_y(x,0) &= 0,\\
u(x,H) + A\varepsilon^{1/3} u_y(x,H) &= 0,\\
v(x,0) = v(x,H) &= 0
\end{aligned}
\right.
\end{equation}
for some $A > 0$.

The objective of this paper is to develop a new energy method to investigate the stability of steady symmetric channel flows, possibly depending on $x$, for small viscosity $\varepsilon \ll 1$.

Let the base flow be $\{u_0(y),0\}$, whose assumptions are specified below. We construct a steady-state solution $\{u_s(x,y), v_s(x,y)\}$ that satisfies the steady Navier-Stokes equation and the boundary condition \eqref{Navier_boundary_condition_new}, and that is close to base flow $\{u_0(y),0\}$. We then analyze the linear and nonlinear stability of $\{ u_s, v_s\}$ by studying the linearized Navier-Stokes equation around $(u_s,v_s)$ in vorticity form:
\begin{equation}\label{Linearized_NS}
\left\{\begin{aligned}
\Delta \phi_t + u_s \Delta \phi_x - \Delta u_s \phi_x + v_s \Delta \phi_y - \Delta v_s \phi_y - \varepsilon \Delta^2 \phi &= 0, \quad \mbox{in}~\Omega \times (0,\infty),\\
\phi|_{t=0} &= \phi_0,
\end{aligned}\right.
\end{equation}
and the nonlinear Navier-Stokes equation:
\begin{equation}\label{Nonlinear_NS}
\left\{\begin{aligned}
\Delta \phi_t + u_s \Delta \phi_x - \Delta u_s \phi_x + v_s \Delta \phi_y - \Delta v_s \phi_y - \varepsilon \Delta^2 \phi &= -\phi_y \Delta \phi_x + \phi_x \Delta \phi_y, \quad \mbox{in}~\Omega \times (0,\infty),\\
\phi|_{t=0} &= \phi_0,
\end{aligned}\right.
\end{equation}
where $\phi$ denotes the stream function of the perturbation. We assume that perturbation $\phi$ satisfies the viscous inflow boundary condition:
\begin{equation}\label{viscous_bc}
\phi|_{x=0} = \phi_x|_{x=L} = \phi_{xx}|_{x=0} = \phi_{xxx}|_{x = L} =  0.
\end{equation}
The $\varepsilon$-Navier boundary condition \eqref{Navier_boundary_condition_new} reads as
\begin{equation}\label{Navier_bc2}
(\phi_y - A\varepsilon^{1/3} \Delta \phi)|_{y=0} = (\phi_y + A\varepsilon^{1/3} \Delta \phi)|_{y=H}= \phi|_{y = 0} = \phi|_{y = H} =0.
\end{equation}\\

\subsection{Main results}

\subsubsection{Main results in the short channel}

We first state our main results in the short channel with $L \ll 1$. In this case, the height $H$ is irrelevant, so we assume $H = 1$ for simplicity. 

Let $u_0(y) \in C^\infty(0,1)$ be even symmetric with respect to $y = \frac{1}{2}$, and satisfy
\begin{equation} \label{assumption_L1}
\left\{
\begin{aligned}
&u_0(y) > 0, \quad y \in (0, 1),\\
&u_0'(0) = - u_0'(1) > 0,\\
&u_0(0) = u_0(1) = 0,\\
&\left\|\frac{u_0'''}{u_0}\right\|_{L^\infty} \le C.
\end{aligned}\right.
\end{equation}
Given such a base profile, we can construct a steady-state solution that is close to the base profile.

\begin{theorem}[Existence of steady state]\label{Thm_Existence_L}
Let $u_0(y)$ be the base flow that satisfies the assumption \eqref{assumption_L1} . For any $A > 0$, $0 < \varepsilon \ll L \ll 1$, there exists a steady state $\{ u_s(x,y), v_s(x,y) \}$ to the Navier-Stokes equation \eqref{NS} with the $\varepsilon$-Navier boundary condition \eqref{Navier_boundary_condition_new} such that $u_s$ is even with respect to $y = 1/2$, and $v_s$ is odd with respect to $y = 1/2$. Furthermore, it satisfies
\begin{equation}\label{ss_estimates}
\begin{aligned}
&0 < u_s(x,y) - u_0(y) \lesssim \varepsilon^{\frac{1}{3}}, \quad |v_{s}|  \lesssim \varepsilon, \quad |u_{sy} - u_{0y}| \lesssim \varepsilon^{\frac{2}{3}}, \quad
|u_{sx}| + |v_{sx}|  \lesssim \varepsilon^{1-},\\
& |\Delta u_a| \lesssim 1, \qquad | \Delta u_{ax}| + |\Delta u_a - u_{0yy}| \lesssim \varepsilon^{\frac{1}{3}},  \qquad |\Delta v_a| + |\Delta v_{ax}| \lesssim \varepsilon^{\frac{2}{3}},\\
& \| \sqrt{u_0} (\Delta u_s - \Delta u_a) \|_{L^\infty_x L^2_y} + \| \sqrt{u_0} (\Delta v_s - \Delta v_a) \|_{L^\infty_x L^2_y}    \lesssim \varepsilon^{\frac{5}{6}-},\\
& \| \sqrt{u_0} (\Delta u_{sx} - \Delta u_{ax})  \|_{L^2} + \sqrt{u_0} (\Delta v_{sx} - \Delta v_{ax}) \|_{L^2} \lesssim \varepsilon^{\frac{5}{6}-},
\end{aligned}
\end{equation}
where $\{u_a, v_a\}$ is the approximated solution given in \eqref{approximated}, and the constants of these estimates may depend on $A$.
\end{theorem}

Then, we have the linear and nonlinear stability results for the steady state $\{ u_s(x,y), v_s(x,y) \}$.

\begin{theorem}[Linear stability]\label{Thm_L_linear}
Under the assumption of Theorem \ref{Thm_Existence_L}, let $u_0, u_s, v_s$ be as in Theorem \ref{Thm_Existence_L}, and let $\phi$ satisfy the linearized Navier-Stokes equation \eqref{Linearized_NS} with boundary conditions \eqref{viscous_bc} and \eqref{Navier_bc2}, we have the estimate
\begin{equation}\label{est:L_linear}
\begin{aligned}
\|\sqrt u_0 \Delta \phi\|_{L^2}^2(T) + \frac{1}{L} \int_0^T \| u_0 \Delta \phi \|_{L^2}^2 + \varepsilon \int_0^T  \| \sqrt{u_0} \nabla \Delta \phi \|_{L^2}^2 \lesssim  \|\sqrt u_0 \Delta \phi_0\|_{L^2}^2 \quad \mbox{for all}~T > 0.
\end{aligned}
\end{equation}
Moreover, we have the exponential decay
\begin{equation}\label{est:L_linear_enhance}
\| \nabla \phi\|_{L^2}(T) + \| \sqrt u_0 \Delta \phi\|_{L^2}(T) \lesssim e^{- c_1 \varepsilon^{\frac{1}{3}} T}\|\sqrt u_0 \Delta \phi_0\|_{L^2}
\end{equation}
for all $T > 0$ and some $c_1 > 0$.
\end{theorem}

\begin{theorem}[Nonlinear stability]\label{Thm_L_nonlinear}
Under the assumption of Theorem \ref{Thm_Existence_L}, let $u_0, u_s, v_s$ be as in Theorem \ref{Thm_Existence_L}. There exists a positive constant $c_0$ depending only on $u_0$, such that if
\begin{equation}\label{L_initial}
\|\sqrt u_0 \Delta \phi_{0x}\|_{L^2} \le c_0 \varepsilon^{\frac{2}{3}},
\end{equation}
then \eqref{Nonlinear_NS} admits a unique solution $\phi$ satisfying the boundary conditions \eqref{viscous_bc} and \eqref{Navier_bc2}, and the estimate
\begin{equation}\label{est:L_nonlinear}
\begin{aligned}
\|\sqrt u_0 \Delta \phi_x\|_{L^2}^2(T) + \frac{1}{L} \int_0^T \| u_0 \Delta \phi_x \|_{L^2}^2 + \varepsilon \int_0^T  \| \sqrt{u_0} \nabla \Delta \phi_x \|_{L^2}^2 \lesssim  \|\sqrt u_0 \Delta \phi_{0x}\|_{L^2}^2 \quad \mbox{for all}~T > 0.
\end{aligned}
\end{equation}
Moreover, we have the exponential decay
\begin{equation}\label{est:L_nonlinear_enhance}
\| \nabla \phi_x\|_{L^2}(T) + \| \sqrt u_0 \Delta \phi_x\|_{L^2}(T) \lesssim e^{- c_1 \varepsilon^{\frac{1}{3}} T} \|\sqrt u_0 \Delta \phi_{0x}\|_{L^2}
\end{equation}
for all $T > 0$ and some $c_1 > 0$.\\
\end{theorem}

\subsubsection{Main results in the long channel}

We now state our results in the long channel $\Omega = (0,L) \times (0,H)$. In this case, we consider a similar base flow, with an additional concavity assumption. Let $\mu(y) \in C^\infty(0,1)$ be even symmetric with respect to $y = \frac{1}{2}$, and satisfy
\begin{equation} \label{assumption_1}
\left\{
\begin{aligned}
&\mu(y) > 0, \quad y \in (0, 1),\\
&\mu'(0) = - \mu'(1) > 0,\\
&\mu(0) = \mu(1) = 0,\\
&-\mu''(y) \ge C_1 > 0,\\
&\left\|\frac{\mu'''}{\mu}\right\|_{L^\infty} \le C.
\end{aligned}\right.
\end{equation}
We additionally assume that for any function $f \in H^1(0,1)$ satisfying $f(1/2) = 0$, we have the following Hardy-type inequality
\begin{equation}\label{assumption_2}
\|f\|_{L^2_y} \le C_2 \| \mu f' \|_{L^2_y}
\end{equation}
for some positive $C_2$ satisfying
\begin{equation}\label{assumption_3}
C_1 C_2 < 2\pi.
\end{equation}
Let $u_0(y):= \mu(y/H)$ be the base flow.
We define the quantity ``Rayleigh vorticity" by
\begin{equation}\label{def:R}
R(t,x,y):= u_0(y) \Delta \phi(t,x,y) - u_{0yy}(y) \phi(t,x,y),
\end{equation}
and the quantity ``quotient" by
\begin{equation}\label{quotient}
q(t,x,y):= \frac{\phi(t,x,y)}{u_0(y)}.
\end{equation}
Then we have similar results on the existence of a steady state and its stability.

\begin{theorem}[Existence of steady state]\label{Thm_H_Existence}
Let $\mu(y)$ satisfy the assumption \eqref{assumption_1}, \eqref{assumption_2}, and \eqref{assumption_3}, and let $u_0(y) := \mu(y/H)$ be the base flow. For positive $H$, $L$, and $A$ satisfying
\begin{equation}\label{H_assumption}
 HL/A^3 \ll 1,
\end{equation}
 and $0 < \varepsilon \ll 1$, there exists a steady state $\{ u_s(x,y), v_s(x,y) \}$ to the Navier-Stokes equation \eqref{NS} with the $\varepsilon$-Navier boundary condition \eqref{Navier_boundary_condition_new} such that $u_s$ is even with respect to $y = H/2$, and $v_s$ is odd with respect to $y = H/2$. Furthermore, it satisfies the estimate
\begin{equation}\label{ss_estimates_H}
\begin{aligned}
&0 < u_s(x,y) - u_0(y) \lesssim A \varepsilon^{\frac{1}{3}}, \quad |v_{s}|  \lesssim \varepsilon^{1-}, \quad \|u_{sy} - u_{0y}\|_{L^2} \lesssim \varepsilon^{\frac{2}{3}}, \quad
\|u_{sx}\|_{L^2}  \lesssim \varepsilon^{1-},\\
& |\Delta u_a| \lesssim 1, \qquad | \Delta u_{ax}| + |\Delta u_a - u_{0yy}| \lesssim \varepsilon^{\frac{1}{3}},  \qquad |\Delta v_a| + |\Delta v_{ax}| \lesssim \varepsilon^{\frac{2}{3}},\\
& \| \sqrt{u_0} (\Delta u_s - \Delta u_a) \|_{L^2} + \| \sqrt{u_0} (\Delta v_s - \Delta v_a) \|_{L^2}    \lesssim \varepsilon^{\frac{5}{6}-},
\end{aligned}
\end{equation}
where $\{u_a, v_a\}$ is the approximated solution given in \eqref{approximated}, and the constants of these estimates, expect for the first one, may depend on $A, H,$ and $L$.
\end{theorem}

\begin{theorem}[Linear stability]\label{Thm_H_linear}
Under the assumption of Theorem \ref{Thm_H_Existence}, let $u_0, u_s, v_s$ be as in Theorem \ref{Thm_H_Existence}, and let $\phi$ be odd with respect to $y = H/2$, and satisfy the linearized Navier-Stokes equation \eqref{Linearized_NS} with boundary conditions \eqref{viscous_bc} and \eqref{Navier_bc2}, we have the estimate
\begin{equation}\label{est:H_linear}
\begin{aligned}
&\| \sqrt{\frac{u_0}{-u_{0yy}}}  \Delta \phi \|_{L^2}^2(T) + \| \nabla \phi\|_{L^2}^2(T)  + \frac{1}{L} \int_0^T \|\frac{R}{\sqrt{-u_{0yy}}}\|_{L^2}^2 + \varepsilon \int_0^T  \| \sqrt{\frac{u_0}{-u_{0yy}}} \nabla \Delta \phi \|_{L^2}^2\\
&\lesssim \| \sqrt{\frac{u_0}{-u_{0yy}}}  \Delta \phi_0 \|_{L^2}^2 \quad \mbox{for all}~T > 0.
\end{aligned}
\end{equation}
Moreover, we have the exponential decay
\begin{equation}\label{est:H_linear_enhance}
\| \nabla \phi\|_{L^2}(T) + \| \sqrt{\frac{u_0}{-u_{0yy}}}  \Delta \phi \|_{L^2}(T) \lesssim e^{- c_1 \varepsilon^{\frac{1}{3}} T}\| \sqrt{\frac{u_0}{-u_{0yy}}}  \Delta \phi_0 \|_{L^2}
\end{equation}
for all $T > 0$ and some $c_1 > 0$.
\end{theorem}

\begin{theorem}[Nonlinear stability]\label{Thm_H_nonlinear}
Under the assumption of Theorem \ref{Thm_H_Existence}, let $u_0, u_s, v_s$ be as in Theorem \ref{Thm_H_Existence}. For any $\beta > 0$, there exists a positive constant $c_0$ depending on $\beta$, $u_0$, $L$, and $H$, such that if $\phi_0$ is odd with respect to $y = H/2$, and satisfies
\begin{equation}\label{H_initial}
\| \sqrt{\frac{u_0}{-u_{0yy}}}  \Delta \phi_{0} \|_{L^2} \le c_0 \varepsilon^{\frac{5}{6}+ \beta},
\end{equation}
then \eqref{Nonlinear_NS} admits a unique solution $\phi$ satisfying the boundary conditions \eqref{viscous_bc} and \eqref{Navier_bc2}, and the estimates \eqref{est:H_linear} and \eqref{est:H_linear_enhance}.
\end{theorem}

Let us give some remarks on the theorems above:
\begin{itemize}
\item Formally, as the viscosity $\varepsilon \to 0$, the $\varepsilon$-Navier boundary condition \eqref{Navier_boundary_condition_new} becomes the no-slip boundary condition, and the inflow boundary condition of the base flow $u_0(y)$ is retained due to the viscous inflow boundary condition \eqref{viscous_bc}. Similar to \cite{GuoIyer23}, one may impose a more general viscous boundary condition
$$
\phi|_{x=0} = a_1^\varepsilon(y), \quad \phi_x|_{x=L} = a_2^\varepsilon(y), \quad \phi_{xx}|_{x=0} = a_3^\varepsilon(y), \quad \phi_{xxx}|_{x=L} = a_4^\varepsilon(y).
$$
Thanks to the exponential decay estimates \eqref{est:L_linear_enhance}, \eqref{est:L_nonlinear_enhance}, and \eqref{est:H_linear_enhance}, we can still derive global-in-time stability provided that the boundary data $a^\varepsilon_i$, $i=1,2,3,4$, are of order $O(\varepsilon^{1/6})$.
\\
\item The stability of shear flow $u_0(y)$ has been vastly studied in many literature. However, when $u_0$ is neither the Couette nor Poiseuille flow, it satisfies the Navier-Stokes equation with an extra forcing term $(-\varepsilon u''_0, 0)$. In contrast, we first construct a non-shear steady state $\{u_s,v_s\}$ that is close to $\{u_0, 0\}$ and satisfies the unforced Navier–Stokes equation exactly. We then analyze the stability of this steady state.\\
\item Assumption \eqref{assumption_L1} is mild. In particular, profiles satisfying \eqref{assumption_L1} may have a very fluctuating part. See Figure~\ref{fig:example} for example. As a consequence, Theorems~\ref{Thm_L_linear} and \ref{Thm_L_nonlinear} imply that, when the channel is sufficiently short, there is a wide class of stable symmetric flows. Thanks to the viscous inflow boundary condition \eqref{viscous_bc}, any profile satisfying \eqref{assumption_L1} serves as a local attractor within its neighborhood of $O(\varepsilon^{\frac{2}{3}})$, demonstrating a rich and chaotic behavior in the inviscid limit of $\varepsilon \ll 1$.

\begin{figure}[h]
\begin{tikzpicture}[>=stealth]

\begin{scope}[x=6cm,y=2cm]

  \draw[->] (-0.02,0) -- (1.05,0) node[right] {$y$};
  \draw[->] (0,-0.10) -- (0,1.05) node[above] {$u_0(y)$};

  \draw (0,0) node[below left] {$0$};
  \draw (1,0) node[below] {$1$};
  \draw (0.5,0) -- (0.5,0.03) node[below=4pt] {$\tfrac12$};

  \draw[densely dashed] (0.5,0) -- (0.5,0.98);

  \draw[line width=1.1pt,domain=0:1,samples=700,smooth,variable=\y]
    plot ({\y},
          { 1.55*(\y*(1-\y))^0.60 *
            (1.25
             + 0.22*cos(deg(10*pi*(\y-0.5)))
             + 0.18*cos(deg(18*pi*(\y-0.5)))
             + 0.12*cos(deg(26*pi*(\y-0.5)))
            )
          });

\end{scope}
\begin{scope}[shift={(7.5,0)},x=6cm,y=2cm]

  \draw[->] (-0.02,0) -- (1.05,0) node[right] {$y$};
  \draw[->] (0,-0.10) -- (0,1.05) node[above] {$u_0(y)$};

  \draw (0,0) node[below left] {$0$};
  \draw (1,0) node[below] {$1$};
  \draw (0.5,0) -- (0.5,0.03) node[below=4pt] {$\tfrac12$};

  \draw[densely dashed] (0.5,0) -- (0.5,0.98);

  \draw[line width=1.1pt,domain=0:1,samples=450,smooth,variable=\y]
    plot ({\y},
          { 3.0*(\y)*(1-\y) *
            (0.90
             + 0.22*cos(deg(4*pi*(\y-0.5)))
             + 0.12*cos(deg(8*pi*(\y-0.5)))
            )
          });

\end{scope}
\end{tikzpicture}
\caption{Two examples of $u_0$ satisfying the assumption \eqref{assumption_L1}}\label{fig:example}
\end{figure}
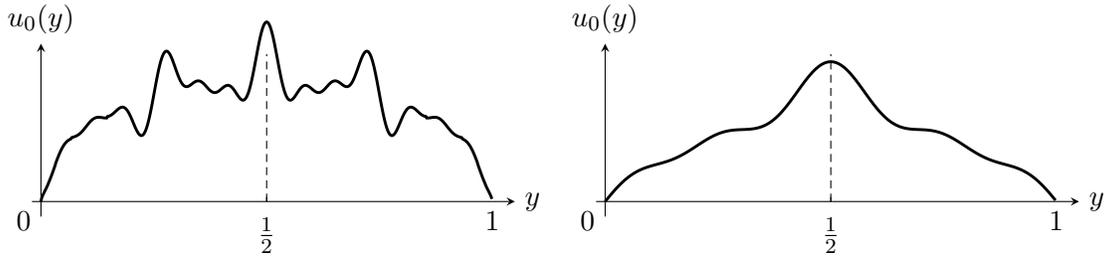

\item Under assumption \eqref{assumption_1}, the Hardy-type inequality \eqref{assumption_2} holds automatically for some constant $C_2$ depending on the profile $\mu(y)$ (see Lemma \ref{lem_hardy2_narrow}). Thus, the essential assumption for the profile $\mu$ is the spectral condition \eqref{assumption_3}. For the Poiseuille flow $\mu(y) = y(1-y)$, we have $C_1 = 2$, and we provide a proof of \eqref{assumption_2} with $C_2 = 2$ (see Appendix \ref{Sec:appendix})\footnote{In an earlier draft, we proved \eqref{assumption_2} with $C_2 \approx 2.58$. We thank Hongjie Dong for pointing out an alternative argument that improves the constant to $C_2 = 2$.}. So \eqref{assumption_3} is satisfied for the Poiseuille flow with a substantial margin. Therefore, one can verify that the assumption \eqref{assumption_3} holds for a broader class of concave symmetric profiles close to the Poiseuille flow. It is very likely that the spectral condition \eqref{assumption_3} can be relaxed.\\

\item The asymptotic stability threshold in the short channel is $\frac{2}{3}$. This agrees with \cites{Chen-Li-Shen-Zhang25, Ding-Lin25}, which study, respectively, symmetric shear flows with no-slip boundary conditions and Poiseuille flow with Navier-slip boundary conditions (in which $\Delta\phi=0$). However, the space we impose is a weighted $H^1$ space, which is weaker than $H^4$ and $H^{\frac{7}{2}+}$ used in those work. Moreover, the weight $\sqrt{u_0}$ degenerates near the boundary and behaves like $\varepsilon^{\frac{1}{6}}$ there (see, for instance, Lemma~\ref{lem_hardy1}). Therefore, our result indicates that the stability threshold for an unweighted Sobolev space should be $\frac{1}{2}$ near the boundary. The weighted $H^1$ space can be further relaxed to a weighted $L^2$ space. The trade-off is a higher stability threshold, namely $\frac{5}{6}+$, which is analogous to the long-channel case (Theorem \ref{Thm_H_nonlinear}).\\

\item The exponential decay estimates \eqref{est:L_linear_enhance}, \eqref{est:L_nonlinear_enhance}, and \eqref{est:H_linear_enhance} correspond to a dissipation time scale $O(\varepsilon^{-1/3})$. This is faster than the typical $O(\varepsilon^{-1/2})$ enhanced dissipation time scale proved for symmetric shear flows in periodic channels; see \cites{CotiZelati-Elgindi-Widmayer20, Zotto23, Ding-Lin25, Chen-Li-Shen-Zhang25}. This stronger stabilizing effect is due to our viscous inflow boundary condition \eqref{viscous_bc}.\\

\end{itemize}

\subsection{Background and motivations}

Hydrodynamic stability in the high Reynolds number regime has attracted significant attention since the early work of Reynolds. Experimentally, steady laminar flows generally lose stability and transition to turbulence as the Reynolds number increases. A natural question is how small a perturbation must be, in terms of the Reynolds number, to prevent transition away from a given steady state. Trefethen et al. \cite{Trefethen} first purposed this crucial question, which was later formulated by Bedrossian, Germain, and Masmoudi \cites{Bedrossian-Germain-Masmoudi17, Bedrossian-Germain-Masmoudi19} as the following stability threshold problem: Given a norm $\| \cdot\|_X$ on the initial data, determine a $\gamma = \gamma(X)$ such that $\| \phi_0\|_X \ll \varepsilon^\gamma$ implies stability, whereas $\| \phi_0\|_X \gg \varepsilon^\gamma$ implies instability. The exponent $\gamma$ is referred as the stability threshold.

Since then, there has been a number of literature in studying the stability of various flows. For Couette flow and more general monotone shear flows, we refer to \cites{Bedrossian-Vicol-Wang18, Chen-Li-Wei-Zhang20, Bedrossian-Germain-Masmoudi20, Bedrossian-Germain-Masmoudi22, Masmoudi-Zhao22, Chen-Wei-Zhang23a, Deng-Masmoudi23, Bian-Grenier-Masmoudi-Zhao25, Bedrossian-He-Iyer-Wang25, Bedrossian-He-Iyer-Li-Wang25} and the references therein. For non-monotone shear flows, including Poiseuille and Kolmogorov flows, see \cites{Wei-Zhang-Zhao19, Wei-Zhang-Zhao20, Li-Wei-Zhang20, Chen-Wei-Zhang23b, Chen-Ding-Lin-Zhang23, Chen-Jia-Wei-Zhang25, Beekie-Chen-Jia26} and the references therein. These works highlight three important stabilizing mechanisms. The first is inviscid damping, namely the decay of velocity or stream function perturbations driven by mixing in the underlying Rayleigh term. The second is enhanced dissipation, where mixing couples with viscosity to produce decay on time scales much faster than the diffusive scale. The third is vorticity depletion, which refers to the asymptotic vanishing of the vorticity near the critical points of a non-monotone shear flow.

For planar Poiseuille flow, enhanced dissipation on the time scale $O(\varepsilon^{-\frac{1}{2}}|\log \varepsilon|)$ in $\bT \times \bR$ was first proved by Coti Zelati, Elgindi, and Widmayer \cite{CotiZelati-Elgindi-Widmayer20} based on the hypocoercivity framework. Del Zotto \cite{Zotto23} later sharpened the time scale to $O(\varepsilon^{-1/2})$ and also obtained an $L^2$-based stability threshold $\gamma \le \frac{2}{3}+$. On the bounded channel $\bT \times [-1,1]$, under the boundary conditions 
$$\phi(t,x,\pm 1) = 0, \quad \Delta \phi (t,x,\pm 1) = 0,$$ Ding and Lin \cite{Ding-Lin25} proved the same enhanced dissipation rate and established a stability threshold $\gamma \le \frac{2}{3}$ for $H^{\frac{7}{2}+}$. More recently, Chen, Li, Shen, and Zhang \cite{Chen-Li-Shen-Zhang25} considered general concave symmetric shear flows in $\bT \times [-1,1]$ with no-slip boundary conditions, and again obtained the $O(\varepsilon^{-1/2})$ enhanced dissipation time scale together with a stability threshold $\gamma \le \frac{2}{3}$ in $H^{4}$. A common feature of the settings above is that the domain is periodic in $x$ and the base flow is a shear. This structure allows one to decompose the solution in Fourier modes in $x$ and analyze the mode-by-mode dynamics. We also point out that on an unbounded channel $\bR \times [-1,1]$, symmetric flows are linearly unstable \cite{Grenier-Guo-Toan16}.

Another important physical phenomenon in the high Reynolds number regime is the formation of boundary layers. Because the no-slip boundary condition for the Navier–Stokes equations is incompatible with the boundary condition satisfied by the inviscid limit, Prandtl proposed that the Navier–Stokes solution admits an asymptotic expansion of the form
$$\left\{
\begin{aligned}
u_{NS} &= u_e + u_p + O(\sqrt{\varepsilon}),\\
v_{NS} &= v_e + \sqrt{\varepsilon} v_p + O(\sqrt{\varepsilon}),
\end{aligned}\right.
$$
where $\{u_e,v_e\}$ is the limiting Euler solution, and $\{u_p,v_p\}$ is the Prandtl boundary layer corrector. In general, this expansion fails for the unsteady Navier–Stokes equations in $L^\infty$ topology; see \cite{Grenier-Toan19}. In contrast, in the steady setting the expansion can be justified when $\{u_e,v_e\}$ is shear and $\{u_p,v_p\}$ is close to Blasius self-similar profile; see \cites{GuoIyer23,IyerMasmoudi21a}. This yields an important class of steady solutions in channels, where the Prandtl corrector introduces an $x$-dependence that is not compatible with Fourier-mode analysis.

Motivated by the discussion above, we develop a new energy framework to study stability for symmetric flows in non-periodic channels. Rather than working with steady states that include a full Prandtl boundary layer, we construct a steady state with a weaker boundary layer, built from a symmetric shear profile that vanishes at the boundaries, and we prove its stability within our framework. We view this as a first step toward analyzing stability for more general non-periodic steady states, including the physically relevant examples discussed above. To the best of our knowledge, this is the first stability-thresholds result in a non-periodic channel.

\subsection{Notations}

Since our analysis relies heavily on the $L^2$-norm, we use $\| \cdot \|_2$ (and occasionally $\| \cdot \|$ or $\| \cdot \|_{L^2}$) to denote the $L^2$-norm for the spacial variables. Norms restricted to a one-dimensional boundary (for example, $\{x = L\}$) are denoted by  $\| \cdot \|_{x=L}$. Similarly, we often use $(\cdot, \cdot)$ to denote the $L^2$ inner product, with surface-restricted inner products distinguished analogously. If a different norm is intended, we will indicate it with a subscript, e.g., $\| \cdot \|_{L^\infty_x L^2_y}$.

We often use the notation $\lesssim$ in the estimates, to indicate an inequality that holds up to a uniform constant. Unless otherwise stated, the uniform constant $C$ may depend on the base flow $u_0$, but is independent of $\varepsilon, A, L$, and $H$ (we write $\lesssim_L$ when dependence on $L$ is allowed). Similarly, we use the notation $O(B)$ to denote a quantity that can be controlled by $CB$, where $C$ is a uniform constant. The only exception is Section \ref{sec:construction}, where the uniform constants may depend on $u_0, A, L$, and $H$, but remain independent of $\varepsilon$.\\

\subsection{Overview of the proof}

One of the major difficulties in studying non-monotone flow is to estimate the term $-u_{0yy} \phi_x$. We overcome this difficulty by two different arguments in the short- and long-channel settings.

When the channel is sufficiently short, we test the linearized Navier-Stokes equation \eqref{Linearized_NS} by $u_0 \Delta \phi W$, where $W = 2-x/L$. With this choice, the Rayleigh term generates a large positive contribution $\|u_0 \Delta \phi\|^2/L$, which can be arranged to dominate the unfavorable term $-u_{0yy} \phi_x$. The viscous dissipation yields an additional positive term $\varepsilon\|\sqrt{u_0} \nabla \Delta \phi\|^2$, but it also produces boundary contributions $\varepsilon \| \sqrt{|u_{0y}|} \Delta \phi \|^2_{y=0,H}$. These boundary terms are critical: they sit at the borderline where a direct Hardy inequality does not close. To bypass this obstacle, we use the $\varepsilon$-Navier boundary condition \eqref{Navier_bc2} to reduce the effective order of these boundary contributions.

When the channel is not short, the term $u_0 \Delta \phi_x$ no longer dominates the unfavorable term $-u_{0yy} \phi_x$. In this case, we group the two terms together as $R_x$, where $R$ is the ``Rayleigh vorticity" defined in \eqref{def:R}. The key observation is that the quantity $-R/u_{0yy}$ couples well with the remaining terms in \eqref{Linearized_NS}. In particular, we obtain a positivity estimate from its interaction with the temporal term; see Lemma \ref{lem:positive}. Another important ingredient is the estimate of a residual term $(\phi_t, \phi_x)$, where the spectral condition \eqref{assumption_3} is involved; see Proposition \ref{prop:phi_t}.

The rest of the article is organized as follows. We postpone the proof of Theorems \ref{Thm_Existence_L} and \ref{Thm_H_Existence} to Section \ref{sec:construction}, since the steady-state construction can be carried out uniformly. Assuming the existence of these steady states, we prove stability results in the short channel (Theorems \ref{Thm_L_linear} and \ref{Thm_L_nonlinear}) in Section \ref{sec:stability_short}, and in the long channel (Theorems \ref{Thm_H_linear} and \ref{Thm_H_nonlinear}) in Section \ref{sec:stability_long}. Finally, we provide a proof of \eqref{assumption_2} for the Poiseuille flow $\mu(y) = y(1-y)$ with $C_2 =  \frac{4\sqrt{5}}{2 \sqrt{5} - 1}$ in the Appendix \ref{Sec:appendix}.\\


\section{Stability in the short channel}\label{sec:stability_short}


In this section, we prove Theorems \ref{Thm_L_linear} and \ref{Thm_L_nonlinear}. As mentioned before, we postpone the construction of steady states to Section \ref{sec:construction}. Here, we recall \eqref{ss_estimates} the key estimates of the steady state.

\subsection{Preliminary estimates}\label{sec:L_prelim}
In this subsection, we establish some preliminary estimates. First, we have a weighted Hardy inequality.
\begin{lemma}\label{lem_hardy1}
For $f \in H^1(\Omega)$, $\sigma > 0$, we have
\begin{equation}\label{Hardy1}
\int u_0^i f^2 \lesssim (\sigma \varepsilon^{1/3})^{i+1}  \int u_0 f_{y}^2 + (\sigma \varepsilon^{1/3})^{i-2}  \int u_0^2 f^2, \quad i=0,1.
\end{equation}
\end{lemma}
\begin{proof}
Let $\chi(y)$ be a smooth cut-off function supported in $[0,1]$, and $\chi \equiv 1$ in $[0, 1/2]$. We write
$$
\int u_0^i  f^2 \lesssim \int u_0^i f^2 \chi^2 \Big( \frac{y}{\sigma \varepsilon^{1/3}} \Big) + \int u_0^i f^2 \chi^2 \Big( \frac{1-y}{\sigma \varepsilon^{1/3}} \Big)+ \int u_0^i f^2 [1-\chi\Big( \frac{y}{\sigma \varepsilon^{1/3}} \Big) - \chi \Big( \frac{1-y}{\sigma \varepsilon^{1/3}}  \Big)]^2.
$$

For the first term, we have
\begin{align*}
\int u_0^i f^2 \chi^2 \Big( \frac{y}{\sigma \varepsilon^{1/3}} \Big) \lesssim & (\sigma \varepsilon^{1/3})^i \int  f^2 \chi^2 \Big( \frac{y}{\sigma \varepsilon^{1/3}} \Big)\\
 =& (\sigma \varepsilon^{1/3})^i  \int \partial_y(y) f^2 \chi^2 \Big( \frac{y}{\sigma \varepsilon^{1/3}} \Big)\\
=&-2 (\sigma \varepsilon^{1/3})^i  \int y f f_y \chi^2 - 2 (\sigma \varepsilon^{1/3})^i  \int \frac{y}{\sigma \varepsilon^{1/3}} f^2 \chi' \chi\\
\lesssim & (\sigma \varepsilon^{1/3})^i  \int y^2 f_{y}^2 \chi^2 + (\sigma \varepsilon^{1/3})^i  \int \Big( \frac{y}{\sigma \varepsilon^{1/3}} \Big)^2 f^2 |\chi'|^2\\
\lesssim & (\sigma \varepsilon^{1/3})^{i+1}  \int u_0 f_y^2 + (\sigma \varepsilon^{1/3})^{i-2} \int u_0^2 f^2,
\end{align*}
where we used $y \lesssim u_0 \lesssim y$ and $y \lesssim \sigma \varepsilon^{1/3}$ in the support of $\chi$. The second term can be estimated similarly.

For the last term, using $u_0 \gtrsim \sigma \varepsilon^{1/3}$ in the support of $1-\chi\Big( \frac{y}{\sigma \varepsilon^{1/3}} \Big) - \chi \Big( \frac{1-y}{\sigma \varepsilon^{1/3}}  \Big)$, we have
$$
\int u_0^i f^2 [1-\chi\Big( \frac{y}{\sigma \varepsilon^{1/3}} \Big) - \chi \Big( \frac{1-y}{\sigma \varepsilon^{1/3}}  \Big)]^2 \lesssim (\sigma \varepsilon^{1/3})^{i-2} \int u_0^2 f^2.
$$
This proves the lemma.
\end{proof}

Very often, we use \eqref{Hardy1} with $f= \Delta \phi$ or $f = \Delta \phi_x$.

\begin{corollary}
For $\sigma > 0$, we have
\begin{equation}\label{est:L_weighted_Hardy}
\begin{aligned}
\varepsilon^{\frac{1}{3}}\| \Delta \phi\|_2 & \lesssim \sqrt\sigma \sqrt\varepsilon \| \sqrt{u_0} \Delta \phi_y\|_2 + \sigma^{-1} \|u_0 \Delta \phi\|_2,\\
\varepsilon^{\frac{1}{3}}\| \Delta \phi_x\|_2 & \lesssim \sqrt\sigma \sqrt\varepsilon \| \sqrt{u_0} \Delta \phi_{xy}\|_2 + \sigma^{-1} \|u_0 \Delta \phi_x\|_2,\\
\varepsilon^{\frac{1}{6}}\| \sqrt{u_0} \Delta \phi\|_2 & \lesssim \sigma \sqrt\varepsilon \| \sqrt{u_0} \Delta \phi_y\|_2 + \sigma^{-1/2} \|u_0 \Delta \phi\|_2,\\
\varepsilon^{\frac{1}{6}}\| \sqrt{u_0} \Delta \phi_x\|_2 & \lesssim \sigma \sqrt\varepsilon \| \sqrt{u_0} \Delta \phi_{xy}\|_2 + \sigma^{-1/2} \|u_0 \Delta \phi_x\|_2.
\end{aligned}
\end{equation}
\end{corollary}

\begin{lemma}\label{lem_hardy2}
For $f \in H^1(\Omega)$ with $f = 0$ on $x = 0$ or $x = L$, $0 < \sigma < 1/2$, we have
\begin{equation}\label{Hardy2}
\int f^2 \lesssim \int u_0^2 f_{y}^2 + \frac{L^2}{\sigma^2}  \int u_0^2 f_x^2.
\end{equation}
\end{lemma}
\begin{proof}
Similar as the above, we let $\chi(y)$ be a smooth cut-off function supported in $[0,1]$, and $\chi \equiv 1$ in $[0, 1/2]$. We write
$$
\int f^2 \lesssim \int f^2 \chi^2 \Big( \frac{y}{\sigma} \Big) + \int f^2 \chi^2 \Big( \frac{1-y}{\sigma} \Big) + \int f^2 [1-\chi\Big( \frac{y}{\sigma} \Big) - \chi \Big( \frac{1-y}{\sigma} \Big)]^2.
$$
For the first term, we proceed as in the proof of Lemma \ref{lem_hardy1} and obtain
\begin{align*}
\int f^2 \chi^2 \Big( \frac{y}{\sigma} \Big)\lesssim & \int y^2 f_{y}^2 \chi^2 + \int \Big( \frac{y}{\sigma} \Big)^2 f^2 |\chi'|^2\\
\lesssim &  \int u_0^2 f_{y}^2 + \frac{1}{\sigma^2} \int u_0^2 f^2.
\end{align*}
Similarly, the second term can be estimated in the same way, and the last term can be controlled by
$$
\int f^2  [1-\chi\Big( \frac{y}{\sigma} \Big) - \chi \Big( \frac{1-y}{\sigma} \Big)]^2 \lesssim \frac{1}{\sigma^2} \int u_0^2 f^2.
$$
Finally, by Poincare inequality in $x$, we have
$$
\int u_0^2 f^2 \lesssim  L^2 \int u_0^2 f_x^2.
$$
This concludes the proof.
\end{proof}

\begin{lemma}
For $L \ll 1$, we have
\begin{equation}\label{phi_H1_estimate}
\|\phi_x\|_2 + \| \phi_y\|_2  \lesssim \|u_0 \Delta \phi\|_2,
\end{equation}
and
\begin{equation}\label{phi_x_H1_estimate}
\|\phi_{xx}\|_2 + \| \phi_{xy}\|_2  \lesssim \|u_0 \Delta \phi_x\|_2.
\end{equation}
\end{lemma}
\begin{proof}
Recall \eqref{quotient}, the definition of $q$, we have
\begin{align*}
(q, u_0 \Delta \phi) =& (q, u_0(u_0 q_{yy} + 2u_{0y}q_y + u_{0yy}q + u_0 q_{xx}))\\
=& - (u_0^2 q_y, q_y) - 2(u_0 u_{0y} q, q_y) + 2(u_0 u_{0y} q, q_y) + (q, u_0 u_{0yy} q) - (u_0^2 q_x, q_x)\\
=& -(u_0^2 \nabla q, \nabla q)+ (q, u_0 u_{0yy} q).
\end{align*}
Using Poincare inequality in $x$, and Lemma \ref{lem_hardy2} with $f = q$, we can estimate
\begin{align*}
|(q, u_0 u_{0yy} q)| \lesssim& \|q\|_2 \|u_0 q\|_2 \lesssim L \|q\|_2 \|u_0 q_x\|_2\\
\lesssim& L (\|u_0 q_y\|_2 + \frac{L}{\sigma}\|u_0 q_x\|_2) \|u_0 q_x\|_2 \le \delta \|u_0 \nabla q\|_2^2
\end{align*}
for some small $\delta > 0$. Therefore by Lemma \ref{lem_hardy2} with $f = q$ again,
\begin{align*}
(1-\delta) \|u_0 \nabla q\|_2^2 \lesssim |(q, u_0 \Delta \phi)| \lesssim \|q\|_2\|u_0 \Delta \phi\|_2 \lesssim \|u_0 \nabla q\|_2\|u_0 \Delta \phi\|_2,
\end{align*}
which implies
$$
\|u_0 \nabla q\|_2 \lesssim \|u_0 \Delta \phi\|_2.
$$
Therefore, by Lemma \ref{lem_hardy2} with $f = q$ again, we have
\begin{align*}
\|\phi_x\|_2 + \| \phi_y\|_2 =& \|u_0 q_x\|_2 + \|u_0 q_y + u_{0y}q \|_2\\
\lesssim& \|u_0 q_x \|_2 + \|u_0 q_y \|_2 + \| q \|_2\\
\lesssim& \|u_0 \nabla q\|_2 \lesssim \|u_0 \Delta \phi\|_2.
\end{align*}
This proves the estimate \eqref{phi_H1_estimate}. Replacing $\phi$ by $\phi_x$ and $q$ by $q_x$, and repeating the argument above yields the estimate \eqref{phi_x_H1_estimate}.
\end{proof}

Finally, we collect some mixed norm estimates.
\begin{lemma}\label{lem:mix_norm}
\begin{equation}\label{est:mix_norm}
\begin{aligned}
\| \nabla \phi\|_{L^2_x L^\infty_y} &\lesssim \varepsilon^{-\frac{1}{6}} (\| u_0 \Delta \phi\|_2 + \sqrt{\varepsilon} \| \sqrt{u_0} \Delta \phi_y\|_2),\\
\| \nabla \phi_x\|_{L^2_x L^\infty_y} &\lesssim \varepsilon^{-\frac{1}{6}} (\| u_0 \Delta \phi_x\|_2 + \sqrt{\varepsilon} \| \sqrt{u_0} \Delta \phi_{xy}\|_2),\\
\| \sqrt{u_0} \Delta \phi\|_{L^\infty_x L^2_y} &\lesssim \varepsilon^{-\frac{1}{3}} ( \| u_0 \Delta \phi\|_2 + \sqrt{\varepsilon} \| \sqrt{u_0} \nabla \Delta \phi\|_2),\\
\| \sqrt{u_0} \Delta \phi_x\|_{L^\infty_x L^2_y} &\lesssim \varepsilon^{-\frac{1}{3}} ( \| u_0 \Delta \phi_x\|_2 + \sqrt{\varepsilon} \| \sqrt{u_0} \nabla\Delta \phi_{x}\|_2)
\end{aligned}
\end{equation}
\end{lemma}

\begin{proof}
We only prove the first and the third inequalities, as the second and the fourth ones follow similarly.
\begin{align*}
\| \nabla \phi\|_{L^2_x L^\infty_y}^2 &= \int_0^L \sup_{y} |\nabla \phi|^2 \, dx\\
&\lesssim \int_0^L \Big( \int_0^1 |\nabla \phi|^2 \, dy + \int_0^1 |\nabla \phi||\nabla \phi_y| \, dy \Big) \, dx\\
&\lesssim \| \nabla \phi \|_2^2 + \| \nabla \phi\|_2 \| \Delta \phi\|_2,
\end{align*}
where we used in second line the Sobolev inequality in $y$. Then the first inequality follows from \eqref{est:L_weighted_Hardy} and \eqref{phi_H1_estimate}. For the third inequality,
\begin{align*}
\int_0^1 u_0| \Delta \phi|^2 \, dy &= 2 \int_0^x \int_0^1 u_0 \Delta \phi_x \Delta \phi \, dy dx\\
&\lesssim \|\sqrt{u_0} \Delta \phi\|_2 \| \sqrt{u_0} \Delta \phi_{x}\|_2\\
&\lesssim \varepsilon^{-\frac{2}{3}} ( \|u_0 \Delta \phi\|_2^2 + \varepsilon \| \sqrt{u_0} \nabla\Delta \phi\|_2^2)
\end{align*}
by \eqref{est:L_weighted_Hardy}. Therefore, the third inequality follows.
\end{proof}

\subsection{Linear stability estimates}\label{sec:L_linear} In the subsection, we derive estimates for linear stability for $\varepsilon \ll L \ll 1$. We define the weight
$$W(x) := 2 - x/L$$ so that $$W' = - \frac{1}{L}$$ gives a negative large constant. We multiply the linearized Navier-Stokes equation \eqref{Linearized_NS} by $W u_0 \Delta \phi$, and estimate each term.

\subsubsection{Rayleigh terms}

We estimate the four Rayleigh terms separately. First, we have
\begin{equation}\label{Ray_1}
\begin{aligned}
(u_s \Delta \phi_x , Wu_0 \Delta \phi) =& (u_0 \Delta \phi_x , Wu_0 \Delta \phi) + ((u_{s} - u_0) \Delta \phi_x , Wu_0 \Delta \phi)\\
=& \frac{1}{2} (u_0 \Delta \phi, W u_0 \Delta \phi)_{x=L} - \frac{1}{2}(W' u_0 \Delta \phi, u_0 \Delta \phi)\\
&+\frac{1}{2}( (u_{s} - u_0)\Delta \phi, Wu_0 \Delta \phi)_{x=L} - \frac{1}{2} ((u_{s} - u_0)W' \Delta \phi , u_0 \Delta \phi) \\
&- \frac{1}{2} (u_{sx}W \Delta \phi , u_0 \Delta \phi).
\end{aligned}
\end{equation}
The term 
\begin{equation}\label{Ray_crucial}
(\ref{Ray_1}.2)= \frac{1}{2L} \|u_0 \Delta \phi\|_2^2
\end{equation}
is the crucial large positive term. The terms (\ref{Ray_1}.1), (\ref{Ray_1}.3) and (\ref{Ray_1}.4) are positive since $u_s - u_0 > 0$. Using the fact that $|u_{sx}| \lesssim \varepsilon^{1-}$ from \eqref{ss_estimates} and \eqref{est:L_weighted_Hardy}, we can estimate the last term
\begin{equation}\label{Ray_1-1}
\begin{aligned}
|(\ref{Ray_1}.5)| \lesssim_A  \varepsilon^{1-} \| \sqrt{u_0} \Delta \phi\|_2^2  
\lesssim_A   \varepsilon^{\frac{2}{3}-} ( \varepsilon \| \sqrt{u_0} \Delta \phi_y \|_2^2 +  \|u_0 \Delta \phi\|_2^2).
\end{aligned}
\end{equation}

Second, we write
\begin{equation}\label{Ray_2}
(-\Delta u_s \phi_x, u_0 W \Delta \phi) = (-\Delta u_a \phi_x, u_0 W \Delta \phi) - ((\Delta u_s - \Delta u_a) \phi_x, u_0 W \Delta \phi).
\end{equation}
Using \eqref{ss_estimates} and \eqref{phi_H1_estimate}, we have
\begin{equation}\label{Ray_2_1}
|(\ref{Ray_2}.1)| \lesssim_A \|\phi_x\|_2\| u_0 \Delta \phi\|_2  \lesssim_A \| u_0 \Delta \phi\|_2^2,
\end{equation}
which can be absorbed by \eqref{Ray_crucial}. By \eqref{ss_estimates}, \eqref{est:L_weighted_Hardy} and \eqref{est:mix_norm}, we can estimate
\begin{equation}\label{Ray_2_2}
\begin{aligned}
|(\ref{Ray_2}.2)| &\lesssim  \| \sqrt{u_0} (\Delta u_s - \Delta u_a) \|_{L^\infty_x L^2_y} \| \phi_x\|_{L^2_x L^\infty_y} \| \sqrt{u_0} \Delta \phi\|_2 \\
&\lesssim_A \varepsilon^{\frac{1}{2}-}(\| u_0 \Delta \phi\|_2^2 + \varepsilon \| \sqrt{u_0} \Delta \phi_{y}\|_2^2).
\end{aligned}
\end{equation}

Next, using  $|v_s| \lesssim \varepsilon$ from \eqref{ss_estimates} and \eqref{est:L_weighted_Hardy}, we have
\begin{equation}\label{Ray_3}
\begin{aligned}
|(v_s \Delta \phi_y ,  u_0 W \Delta \phi)| &\lesssim_A \varepsilon \| \sqrt{u_0} \Delta \phi_y \|_2 \| \sqrt{u_0} \Delta \phi\|_2\\
&\lesssim_A \varepsilon^{\frac{1}{3}}(\| u_0 \Delta \phi\|_2^2 + \varepsilon \| \sqrt{u_0} \Delta \phi_{y}\|_2^2).
\end{aligned}
\end{equation}

For the last term, we split $\Delta v_s = \Delta v_a + (\Delta v_s - \Delta v_a)$ and use \eqref{ss_estimates}, \eqref{phi_H1_estimate}, and \eqref{est:mix_norm} to estimate
\begin{equation}\label{Ray_4}
\begin{aligned}
|(\Delta v_s \phi_y ,  u_0 W \Delta \phi)| \lesssim & |(\Delta v_a \phi_y ,  u_0 W \Delta \phi)| + |((\Delta v_s - \Delta v_a) \phi_y ,  u_0 W \Delta \phi)|\\
\lesssim_A & \varepsilon^{\frac{2}{3}} \|\phi_y\|_2 \|u_0 \Delta \phi\|_2 + \| \sqrt{u_0} (\Delta v_s - \Delta v_a) \|_{L^\infty_x L^2_y} \|\phi_y\|_{L^2_x L^\infty_y}  \| u_0 \Delta \phi\|_2\\
\lesssim_A &  \varepsilon^{\frac{2}{3}} \| u_0 \Delta \phi\|_2^2 +  \varepsilon^{\frac{2}{3}-} (\| u_0 \Delta \phi\|_2^2 + \varepsilon \| \sqrt{u_0} \Delta \phi_{y}\|_2^2).
\end{aligned}
\end{equation}

\subsubsection{Dissipation term}
\begin{equation}\label{Dis_1}
\begin{aligned}
-\varepsilon (\Delta^2 \phi, W u_0 \Delta \phi) =& \varepsilon (Wu_0 \Delta \phi_x, \Delta \phi_x) + \varepsilon (Wu_0 \Delta \phi_y, \Delta \phi_y)\\
&+  \varepsilon (W' u_0 \Delta \phi_x, \Delta \phi)  +  \varepsilon (W u_{0y} \Delta \phi_y, \Delta \phi).
\end{aligned}
\end{equation}
The terms (\ref{Dis_1}.1) and (\ref{Dis_1}.2) are favorable positive terms. Using \eqref{est:L_weighted_Hardy}, we can estimate
\begin{equation}\label{Dis_1-1}
\begin{aligned}
|(\ref{Dis_1}.3)| \le & \frac{\varepsilon}{100} \| \sqrt{u_0} \Delta \phi_x\|_2^2 + \frac{O(\varepsilon)}{L^2} \|\Delta \phi \|_2^2\\
 \le & \frac{\varepsilon}{100} \| \sqrt{u_0} \Delta \phi_x\|_2^2 +\frac{\varepsilon}{100} \| \sqrt{u_0} \Delta \phi_y\|_2^2 + \frac{O(\varepsilon^{1/3})}{L^2} \|u_0 \Delta \phi\|_2^2.
\end{aligned}
\end{equation}
The last term (\ref{Dis_1}.4) is a crucial term, we integrate by parts in $y$ and use the $\varepsilon$-Navier boundary condition \eqref{Navier_bc2}:
\begin{equation}\label{Dis_1-2}
\begin{aligned}
(\ref{Dis_1}.4) =& \frac{\varepsilon}{2} (W u_{0y} \Delta \phi , \Delta \phi)_{y=H} - \frac{\varepsilon}{2} (W u_{0y} \Delta \phi , \Delta \phi)_{y=0} - \frac{\varepsilon}{2} (W u_{0yy} \Delta \phi, \Delta \phi) \\
=& \frac{\varepsilon^{1/3}}{2A^2} (W u_{0y} \phi_y ,\phi_y)_{y=H} - \frac{\varepsilon^{1/3}}{2A^2} (W u_{0y} \phi_y ,\phi_y)_{y=0} - \frac{\varepsilon}{2} (W u_{0yy} \Delta \phi, \Delta \phi)\\
=& \frac{\varepsilon^{1/3}}{A^2} (W u_{0y} \phi_y ,\phi_{yy}) + \frac{\varepsilon^{1/3}}{2A^2} (W u_{0yy} \phi_y ,\phi_y) - \frac{\varepsilon}{2} (W u_{0yy} \Delta \phi, \Delta \phi).
\end{aligned}
\end{equation}
Therefore, by \eqref{est:L_weighted_Hardy} and \eqref{phi_H1_estimate}, we have
\begin{equation}\label{Dis_1-3}
\begin{aligned}
|(\ref{Dis_1-2}.1)| \lesssim & \frac{1}{A^4} \| \phi_y\|_2^2 + \varepsilon^{2/3} \|\Delta \phi\|_2^2\\
\le& \frac{\varepsilon}{50} \|\sqrt{u_0} \Delta \phi_y\|_2^2 +  \frac{O(1)}{A^4} \|u_0 \Delta \phi\|_2^2.
\end{aligned}
\end{equation}
And the last two terms in \eqref{Dis_1-2} can be estimated similarly:
\begin{equation}\label{Dis_1-4}
\begin{aligned}
|(\ref{Dis_1-2}.2)| + |(\ref{Dis_1-2}.3)| \lesssim & \frac{\varepsilon^{1/3}}{A^2}\|\phi_y\|_2^2 + \varepsilon \| \Delta \phi \|_2^2\\
\lesssim_A &  \varepsilon^{4/3} \|\sqrt{u_0} \Delta \phi_y\|_2^2 + \varepsilon^{1/3} \|u_0 \Delta \phi\|_2^2.
\end{aligned}
\end{equation}

\subsubsection{Proof of Theorem \ref{Thm_L_linear}} The temporal term can be straightforwardly written as
$$
(\Delta \phi_t, W u_0 \Delta \phi) = \frac{1}{2} \frac{d}{dt} \| \sqrt{W} \sqrt{u_0} \Delta \phi \|_2^2.
$$
Now we collect all the terms from \eqref{Ray_1} - \eqref{Dis_1-4}. By choosing  $L > 0$ small enough, we can see that the favorable positive terms \eqref{Ray_crucial}, (\ref{Dis_1}.1), and (\ref{Dis_1}.2) dominate all the non-favorable terms from \eqref{Ray_1} - \eqref{Dis_1-4} when $\varepsilon \ll L \ll 1$. Therefore, we end up with
\begin{equation}\label{est:L_linear_final}
\frac{d}{dt} \|\sqrt{W} \sqrt{u_0} \Delta \phi \|_2^2 + \frac{1}{L} \|u_0 \Delta \phi\|_2^2 + \varepsilon \| \sqrt{u_0} \nabla \Delta \phi \|_2^2 \le 0.
\end{equation}
This concludes the proof of \eqref{est:L_linear}. Furthermore, by \eqref{est:L_weighted_Hardy}, we know that
$$
\varepsilon^{1/3} \|\sqrt{u_0} \Delta \phi\|_2^2 \lesssim \|u_0 \Delta \phi\|_2^2 + \varepsilon \| \sqrt{u_0} \nabla \Delta \phi \|_2^2.
$$
Therefore, \eqref{est:L_linear_final} and \eqref{phi_H1_estimate} imply the exponential decay estimate \eqref{est:L_linear_enhance}. This concludes the proof of Theorem \ref{Thm_L_linear}.\\

\subsection{Nonlinear stability estimates}\label{sec:L_nonlinear} In this subsection, we prove Theorem \ref{Thm_L_nonlinear}. We define the space $\cX$ to be the closure of smooth function satisfying the boundary conditions \eqref{viscous_bc} and \eqref{Navier_bc2} under the norm
$$
\| \phi \|_{\cX}^2:= \sup_{t \in (0,\infty)} \|\sqrt u_0 \Delta \phi_x\|_{L^2}^2 + \frac{1}{L} \int_0^\infty \| u_0 \Delta \phi_x \|_{L^2}^2 \, dt + \varepsilon \int_0^\infty  \| \sqrt{u_0} \nabla \Delta \phi_x \|_{L^2}^2 \, dt.
$$
We define the solution map $\cS: \psi \mapsto \phi$, where $\phi$ is the solution to
\begin{equation}\label{Nonlinear_new}
\left\{\begin{aligned}
\Delta \phi_t + u_s \Delta \phi_x - \Delta u_s \phi_x + v_s \Delta \phi_y - \Delta v_s \phi_y - \varepsilon \Delta^2 \phi &= -\psi_y \Delta \psi_x + \psi_x \Delta \psi_y, \quad \mbox{in}~\Omega \times (0,\infty),\\
\phi|_{t=0} &= \phi_0,
\end{aligned}\right.
\end{equation}
and satisfies the boundary conditions \eqref{viscous_bc} and \eqref{Navier_bc2}. Let 
$$
\cB := \{ \phi \in \cX : \| \phi \|_{\cX} \le 2 c_0 \varepsilon^{\frac{2}{3}} \},
$$
where $c_0$ is the small constant in Theorem \ref{Thm_L_nonlinear}. Our goal is to show that $\cS$ is a contraction mapping in $\cB$ when $\varepsilon \ll L \ll 1$.

We multiply the equation \eqref{Nonlinear_new} by $-(u_0 \Delta\phi_x W)_x$ and estimate term by term. The linear part is similar to what have been done in the previous section.

\subsubsection{Temporal term}
\begin{equation}\label{N_Tem}
-(\Delta \phi_t, (u_0 \Delta\phi_x W)_x) = (\Delta \phi_{xt}, u_0 \Delta\phi_x W) = \frac{1}{2} \frac{d}{dt} \|\sqrt{u_0} \Delta \phi_x \sqrt{W} \|_2^2.
\end{equation}

\subsubsection{Rayleigh terms}
Again, we estimate the four Rayleigh terms separately. First, we have
\begin{equation}\label{N_Ray_1}
\begin{aligned}
-(u_s \Delta \phi_x , (u_0 \Delta\phi_x W)_x) =& -(u_0 \Delta \phi_x , (u_0 \Delta\phi_x W)_x) - ((u_{s} - u_0)\Delta \phi_x, (u_0 \Delta\phi_x W)_x)\\
=& - \frac{1}{2}( u_0 \Delta \phi_x, u_0 \Delta \phi_x W') + \frac{1}{2} (u_0 \Delta \phi_x,  u_0 \Delta \phi_x W)_{x=0}\\
&- \frac{1}{2} ((u_{s} - u_0) \Delta \phi_x , u_0 \Delta \phi_x W') +\frac{1}{2}( (u_{s} - u_0)\Delta \phi_x, Wu_0 \Delta \phi_x)_{x=0}  \\
&+ \frac{1}{2} (u_{sx}W \Delta \phi_x , u_0 \Delta \phi_x).
\end{aligned}
\end{equation}
The term 
\begin{equation}\label{N_Ray_crucial}
(\ref{N_Ray_1}.1)= \frac{1}{2L} \|u_0 \Delta \phi_x\|_2^2
\end{equation}
is the crucial large positive term. The terms (\ref{N_Ray_1}.2), (\ref{N_Ray_1}.3) and (\ref{N_Ray_1}.4) are positive since $u_s - u_0 > 0$. As \eqref{Ray_1-1}, the last term can be estimated by

\begin{equation}\label{N_Ray_1-1}
\begin{aligned}
|(\ref{N_Ray_1}.5)| \lesssim_A  \varepsilon^{1-} \| \sqrt{u_0} \Delta \phi_x\|_2^2  
\lesssim_A   \varepsilon^{\frac{2}{3}-} ( \varepsilon \| \sqrt{u_0} \Delta \phi_{xy} \|_2^2 +  \|u_0 \Delta \phi_x\|_2^2).
\end{aligned}
\end{equation}

Second, we write
\begin{equation}\label{N_Ray_2}
\begin{aligned}
(\Delta u_s \phi_x, (u_0 \Delta\phi_x W)_x) =& (\Delta u_a \phi_x,(u_0 \Delta\phi_x W)_x) + ((\Delta u_s - \Delta u_a) \phi_x, (u_0 \Delta\phi_x W)_x)\\
=& -(\Delta u_{ax} \phi_x, u_0 \Delta\phi_x W) - (\Delta u_a \phi_{xx},u_0 \Delta\phi_x W)\\
&- (\Delta u_a \phi_{x},u_0 \Delta\phi_x W)_{x = 0} + ((\Delta u_s - \Delta u_a) \phi_x, u_0 \Delta\phi_{xx} W)\\
&+ ((\Delta u_s - \Delta u_a) \phi_x, u_0 \Delta\phi_x W').
\end{aligned}
\end{equation}
By \eqref{ss_estimates}, Poincare inequality in $x$, and \eqref{phi_x_H1_estimate}, we can estimate
\begin{equation}\label{N_Ray_2_1}
\begin{aligned}
|(\ref{N_Ray_2}.1)| + |(\ref{N_Ray_2}.2)| \lesssim_A \| \phi_{xx}\|_2 \| u_0 \Delta\phi_x \|_2 \lesssim_A \| u_0 \Delta\phi_x \|_2^2,
\end{aligned}
\end{equation}
and
\begin{equation}\label{N_Ray_2_2}
\begin{aligned}
|(\ref{N_Ray_2}.3)| &\le \frac{1}{100} \| u_0 \Delta\phi_x \|_{x=0}^2 + O(1) \| \phi_x\|_{x=0}^2\\
&\le \frac{1}{100} \| u_0 \Delta\phi_x \|_{x=0}^2 + O(1)\|\phi_{xx}\|_2^2\\
&\le \frac{1}{100} \| u_0 \Delta\phi_x \|_{x=0}^2 + O(1) \| u_0 \Delta\phi_x \|_2^2,
\end{aligned}
\end{equation}
where the first term can be absorbed by (\ref{N_Ray_1}.2). Then we use Sobolev inequality in $y$, \eqref{ss_estimates}, and \eqref{phi_x_H1_estimate} to estimate
\begin{equation}\label{N_Ray_2_3}
\begin{aligned}
|(\ref{N_Ray_2}.4)| + |(\ref{N_Ray_2}.5)| &\lesssim \frac{1}{L}  \| \sqrt{u_0} (\Delta u_s - \Delta u_a) \|_{L^\infty_x L^2_y}  \|\phi_x\|_{L^2_x L^\infty_y} \| \sqrt{u_0} \Delta\phi_{xx} \|_2\\
& \lesssim_A \frac{\varepsilon^{\frac{5}{6}-}}{L}\| \phi_{xy}\|_2 \| \sqrt{u_0} \Delta\phi_{xx} \|_2\\
& \lesssim_A \frac{\varepsilon^{\frac{1}{3}-}}{L} (\| u_0 \Delta\phi_x \|_2^2 + \varepsilon  \| \sqrt{u_0} \Delta\phi_{xx} \|_2^2).
\end{aligned}
\end{equation}

Next, we use \eqref{ss_estimates}, Poincare inequality in $x$ and \eqref{est:L_weighted_Hardy} to estimate
\begin{equation}\label{N_Ray_3}
\begin{aligned}
-(v_s \Delta \phi_y, (u_0 \Delta\phi_x W)_x) =&~ (v_{sx}\Delta \phi_y, u_0 \Delta\phi_x W) + (v_{s}\Delta \phi_{xy}, u_0 \Delta\phi_x W)\\
\lesssim_A&~ \varepsilon^{1-}\| \sqrt{u_0} \Delta \phi_y\|_2 \| \sqrt{u_0} \Delta \phi_x\|_2 + \varepsilon \| \sqrt{u_0} \Delta \phi_{xy}\|_2 \| \sqrt{u_0} \Delta \phi_x\|_2\\
\lesssim_A&~ \varepsilon^{\frac{1}{3}-} (\sqrt{\varepsilon}  \| \sqrt{u_0} \Delta \phi_{xy}\|_2) (\|u_0 \Delta \phi_x\|_2  +  \sqrt{\varepsilon} \| \sqrt{u_0} \Delta \phi_{xy} \|_2)\\
&+ \varepsilon^{\frac{1}{3}} (\sqrt{\varepsilon}  \| \sqrt{u_0} \Delta \phi_{xy}\|_2) (\|u_0 \Delta \phi_x\|_2  +  \sqrt{\varepsilon} \| \sqrt{u_0} \Delta \phi_{xy} \|_2)\\
\lesssim_A&~ \varepsilon^{\frac{1}{3}-}(\|u_0 \Delta \phi_x\|_2^2  +  \varepsilon \| \sqrt{u_0} \Delta \phi_{xy} \|_2^2).
\end{aligned}
\end{equation}

Finally, we split $\Delta v_s$ and obtain 
\begin{equation}\label{N_Ray_4}
\begin{aligned}
(\Delta v_s  \phi_y, (u_0 \Delta\phi_x W)_x) = & (\Delta v_a  \phi_y, (u_0 \Delta\phi_x W)_x) + (\Delta (v_s-v_a)  \phi_y, (u_0 \Delta\phi_x W)_x)\\
=& (\Delta v_a  \phi_y, u_0 \Delta\phi_{xx} W) + (\Delta v_a  \phi_y, u_0 \Delta\phi_x W')\\
& - (\Delta (v_{sx}-v_{ax})  \phi_y, u_0 \Delta\phi_x W) - (\Delta (v_s-v_a)  \phi_{xy}, u_0 \Delta\phi_x W).
\end{aligned}
\end{equation}
Using \eqref{ss_estimates}, Poincare inequality in $x$ and \eqref{phi_x_H1_estimate}, we can directly estimate
\begin{equation}\label{N_Ray_4_1}
\begin{aligned}
|(\ref{N_Ray_4}.1)|+ |(\ref{N_Ray_4}.2)| \lesssim_A &~ \varepsilon^{\frac{2}{3}} \| \phi_{xy}\|_2 \|\sqrt{u_0} \Delta\phi_{xx}\|_2 + \frac{\varepsilon^{\frac{2}{3}}}{L} \| \phi_{xy}\|_2  \|u_0 \Delta\phi_{x}\|_2\\
\lesssim_A &~ \varepsilon^{\frac{1}{6}} \|u_0 \Delta\phi_{x}\|_2 ( \sqrt{\varepsilon} \|\sqrt{u_0} \Delta\phi_{xx}\|_2) + \frac{\varepsilon^{\frac{2}{3}}}{L} \|u_0 \Delta\phi_{x}\|_2^2\\
\lesssim_A &~ \varepsilon^{\frac{1}{6}}(\|u_0 \Delta \phi_x\|_2^2  +  \varepsilon \| \sqrt{u_0} \Delta \phi_{xy} \|_2^2).
\end{aligned}
\end{equation}
For the last two terms in \eqref{N_Ray_4}, we use Sobolev inequality in $x$, \eqref{ss_estimates}, \eqref{est:mix_norm}, and \eqref{est:L_weighted_Hardy} to estimate
\begin{equation}\label{N_Ray_4_2}
\begin{aligned}
|(\ref{N_Ray_4}.3)|+ |(\ref{N_Ray_4}.4)| \lesssim &~ \| \sqrt{u_0} \Delta (v_{sx}-v_{ax})\|_2 \|\phi_y\|_{L^\infty} \| \sqrt{u_0} \Delta \phi_x\|_2\\
+&~  \| \sqrt{u_0} \Delta (v_{s}-v_{a})\|_{L^\infty_x L^2_y} \| \phi_{xy}\|_{L^2_x L^\infty_y}  \| \sqrt{u_0} \Delta \phi_x\|_2\\
\lesssim&_A ~ \varepsilon^{\frac{5}{6}-}\| \phi_{xy}\|_{L^2_x L^\infty_y}  \| \sqrt{u_0} \Delta \phi_x\|_2\\
\lesssim&_A ~ \varepsilon^{\frac{1}{2}-}(\|u_0 \Delta \phi_x\|_2^2  +  \varepsilon \| \sqrt{u_0} \Delta \phi_{xy} \|_2^2).
\end{aligned}
\end{equation}

\subsubsection{Dissipation term}
\begin{equation}\label{N_Dis_1}
\begin{aligned}
\varepsilon (\Delta^2 \phi, (u_0 \Delta\phi_x W)_x) =& \varepsilon (\Delta \phi_{xx}, (u_0 \Delta\phi_x W)_x) + \varepsilon (\Delta \phi_{yy}, (u_0 \Delta\phi_x W)_x)\\
=&  \varepsilon (\Delta \phi_{xx}, u_0 \Delta\phi_{xx} W)  +  \varepsilon (\Delta \phi_{xx}, u_0 \Delta\phi_{x} W') - \varepsilon (\Delta \phi_{xyy}, u_0 \Delta\phi_x W)\\
=&  \varepsilon (\Delta \phi_{xx}, u_0 \Delta\phi_{xx} W)  +  \varepsilon (\Delta \phi_{xx}, u_0 \Delta\phi_{x} W')\\
+&  \varepsilon (\Delta \phi_{xy}, u_0 \Delta\phi_{xy} W) +  \varepsilon (\Delta \phi_{xy}, u_{0y} \Delta\phi_x W).
\end{aligned}
\end{equation}
The term (\ref{N_Dis_1}.1) and (\ref{N_Dis_1}.3) are favorable positive terms. The second term can be estimated by \eqref{est:L_weighted_Hardy},
\begin{equation}\label{N_Dis_1_1}
\begin{aligned}
|(\ref{N_Dis_1}.2)| \lesssim & \frac{\varepsilon}{L} \| \sqrt{u_0}\Delta \phi_{xx} \|_2 \| \sqrt{u_0}\Delta \phi_{x} \|_2\\
\lesssim & \frac{\varepsilon^{\frac{1}{3}}}{L} (\|u_0 \Delta \phi_x\|_2^2  +  \varepsilon \| \sqrt{u_0} \nabla \Delta \phi_{x} \|_2^2).
\end{aligned}
\end{equation}
The last term (\ref{N_Dis_1}.4) is a crucial term, we use the $\varepsilon$-Navier boundary condition \eqref{Navier_bc2}:
\begin{equation}\label{N_Dis_1_2}
\begin{aligned}
(\ref{N_Dis_1}.4) =& \frac{\varepsilon}{2} (W u_{0y} \Delta \phi_x , \Delta \phi_x)_{y=H} - \frac{\varepsilon}{2} (W u_{0y} \Delta \phi_x , \Delta \phi_x)_{y=0} - \frac{\varepsilon}{2} (W u_{0yy} \Delta \phi_x, \Delta \phi_x) \\
=& \frac{\varepsilon^{1/3}}{2A^2} (W u_{0y} \phi_{xy} ,\phi_{xy})_{y=H} - \frac{\varepsilon^{1/3}}{2A^2} (W u_{0y} \phi_{xy} ,\phi_{xy})_{y=0} - \frac{\varepsilon}{2} (W u_{0yy} \Delta \phi_x, \Delta \phi_x)\\
=& \frac{\varepsilon^{1/3}}{A^2} (W u_{0y} \phi_{xy} ,\phi_{xyy}) + \frac{\varepsilon^{1/3}}{2A^2} (W u_{0yy} \phi_{xy} ,\phi_{xy}) - \frac{\varepsilon}{2} (W u_{0yy} \Delta \phi_x, \Delta \phi_x).
\end{aligned}
\end{equation}
Therefore, by \eqref{est:L_weighted_Hardy} and \eqref{phi_x_H1_estimate}, we have
\begin{equation}\label{N_Dis_1_3}
\begin{aligned}
|(\ref{N_Dis_1_2}.1)| \lesssim &~ \varepsilon^{\frac{2}{3}} \|\Delta \phi_x\|_2^2 + \frac{1}{A^4} \| \phi_{xy}\|_2^2\\
\le&~ \frac{\varepsilon}{100} \|\sqrt{u_0} \Delta \phi_y\|_2^2 +  \frac{O(1)}{A^4} \|u_0 \Delta \phi\|_2^2.
\end{aligned}
\end{equation}
And the last two terms in \eqref{N_Dis_1_2} can be estimated similarly:
\begin{equation}\label{N_Dis_1_4}
\begin{aligned}
|(\ref{N_Dis_1_2}.2)| + |(\ref{N_Dis_1_2}.3)| \lesssim &~ \frac{\varepsilon^{\frac{1}{3}}}{A^2}\|\phi_{xy}\|_2^2 + \varepsilon \| \Delta \phi_x \|_2^2\\
\lesssim&_A ~  \varepsilon^{\frac{1}{3}} (\|u_0 \Delta \phi_x\|_2^2  +  \varepsilon \| \sqrt{u_0} \Delta \phi_{xy} \|_2^2).
\end{aligned}
\end{equation}

\subsubsection{Nonlinear terms}
First, we have
\begin{equation}\label{N_nonlinear_1}
(\psi_y \Delta \psi_x, (u_0 \Delta\phi_x W)_x) = - (\psi_{xy} \Delta \psi_x, u_0 \Delta\phi_x W) - (\psi_{y} \Delta \psi_{xx}, u_0 \Delta\phi_x W).
\end{equation}
By Sobolev inequality in $x$ and \eqref{est:mix_norm}, we can estimate
\begin{equation}\label{N_nonlinear_1_1}
\begin{aligned}
|(\ref{N_nonlinear_1}.1)| &\lesssim \| \psi_{xy}\|_{L^2_x L^\infty_y} \| \sqrt{u_0} \Delta \psi_x \|_{L^\infty_x L^2_y} \| \sqrt{u_0} \Delta \phi_x\|_2\\
&\lesssim \| \psi_{xy}\|_{L^2_x L^\infty_y} \| \sqrt{u_0} \Delta \psi_{xx} \|_{2} \| \sqrt{u_0} \Delta \phi_x\|_2\\
&\lesssim \varepsilon^{-\frac{2}{3}} (\|u_0 \Delta \psi_x\|_2^2  +  \varepsilon \| \sqrt{u_0} \nabla \Delta \psi_{x} \|_2^2) \| \sqrt{u_0} \Delta \phi_x\|_2,
\end{aligned}
\end{equation}
and
\begin{equation}\label{N_nonlinear_1_2}
\begin{aligned}
|(\ref{N_nonlinear_1}.2)| &\lesssim \| \psi_{y}\|_{L^\infty} \| \sqrt{u_0} \Delta \psi_{xx} \|_{2} \| \sqrt{u_0} \Delta \phi_x\|_2\\
&\lesssim \| \psi_{xy}\|_{L^2_x L^\infty_y} \| \sqrt{u_0} \Delta \psi_{xx} \|_{2} \| \sqrt{u_0} \Delta \phi_x\|_2\\
&\lesssim \varepsilon^{-\frac{2}{3}} (\|u_0 \Delta \psi_x\|_2^2  +  \varepsilon \| \sqrt{u_0} \nabla \Delta \psi_{x} \|_2^2) \| \sqrt{u_0} \Delta \phi_x\|_2.
\end{aligned}
\end{equation}

Finally, we can write the second term
\begin{equation}\label{N_nonlinear_2}
-(\psi_x \Delta \psi_y, (u_0 \Delta\phi_x W)_x) =  (\psi_{xx} \Delta \psi_y, u_0 \Delta\phi_x W) + (\psi_{x} \Delta \psi_{xy}, u_0 \Delta\phi_x W).
\end{equation}
Similar as above, we can estimate
\begin{equation}\label{N_nonlinear_2_2}
\begin{aligned}
|(\ref{N_nonlinear_2}.1)| + |(\ref{N_nonlinear_2}.2)| &\lesssim \| \psi_{xx}\|_{L^2_x L^\infty_y} \| \sqrt{u_0} \Delta \psi_{xy} \|_{2} \| \sqrt{u_0} \Delta \phi_x\|_2\\
&\lesssim \varepsilon^{-\frac{2}{3}} (\|u_0 \Delta \psi_x\|_2^2  +  \varepsilon \| \sqrt{u_0} \nabla \Delta \psi_{x} \|_2^2) \| \sqrt{u_0} \Delta \phi_x\|_2.
\end{aligned}
\end{equation}

\subsubsection{Proof of Theorem \ref{Thm_L_nonlinear}}
Now we are ready to prove Theorem \ref{Thm_L_nonlinear}. We collect all the terms from \eqref{N_Tem} - \eqref{N_nonlinear_2_2}. By choosing $L > 0$ small enough, we can see that the favorable positive terms \eqref{N_Ray_crucial}, (\ref{N_Dis_1}.1), and (\ref{N_Dis_1}.3) dominate all the other terms from the linear part when $\varepsilon \ll L \ll 1$. Therefore, we end up with
\begin{equation}\label{est:L_nonlinear_final}
\frac{d}{dt} \|\sqrt{W} \sqrt{u_0} \Delta \phi_x \|_2^2 + \frac{1}{L} \|u_0 \Delta \phi_x\|_2^2 + \varepsilon \| \sqrt{u_0} \nabla \Delta \phi_x \|_2^2 \lesssim \varepsilon^{-\frac{2}{3}} (\|u_0 \Delta \psi_x\|_2^2  +  \varepsilon \| \sqrt{u_0} \nabla \Delta \psi_{x} \|_2^2) \| \sqrt{u_0} \Delta \phi_x\|_2.
\end{equation}
This implies
\begin{align*}
&\sup_{t \in (0,\infty)} \|\sqrt u_0 \Delta \phi_x\|_{2}^2 + \frac{1}{L} \int_0^\infty \| u_0 \Delta \phi_x \|_{2}^2 \, dt + \varepsilon \int_0^\infty  \| \sqrt{u_0} \nabla \Delta \phi_x \|_{2}^2 \, dt\\
\lesssim & \varepsilon^{-\frac{2}{3}} \sup_{t \in (0,\infty)} \|\sqrt u_0 \Delta \phi_x\|_{2} \int_0^\infty \Big( \|u_0 \Delta \psi_x\|_2^2  +  \varepsilon \| \sqrt{u_0} \nabla \Delta \psi_{x} \|_2^2 \Big) \, dt + \|\sqrt u_0 \Delta \phi_{0x}\|_{2}^2.
\end{align*}
Since $\psi \in \cB$, and by the initial data \eqref{L_initial}, we can conclude that $\phi \in \cB$ when $c_0$ is small enough. Hence the solution map $\cS$ maps $\cB$ into $\cB$. 

To show $\cS$ is a contraction mapping, we take $\psi_1, \psi_2 \in \cB$, and denote $\phi_1 = \cS(\psi_1), \phi_2 = \cS(\psi_2)$. We further denote $\widetilde{\phi} = \phi_1 - \phi_2$ and $\widetilde{\psi} = \psi_1 - \psi_2$. Repeating the estimates in this section for the equation for $\widetilde{\phi}$, we will end up with
\begin{align*}
&\sup_{t \in (0,\infty)} \|\sqrt u_0 \Delta \widetilde\phi_x\|_{2}^2 + \frac{1}{L} \int_0^\infty \| u_0 \Delta \widetilde\phi_x \|_{2}^2 \, dt + \varepsilon \int_0^\infty  \| \sqrt{u_0} \nabla \Delta \widetilde\phi_x \|_{2}^2 \, dt\\
\lesssim & \varepsilon^{-\frac{2}{3}} \sup_{t \in (0,\infty)} \|\sqrt u_0 \Delta \widetilde\phi_x\|_{2} \left( \int_0^\infty \Big( \|u_0 \Delta \widetilde\psi_x\|_2^2  +  \varepsilon \| \sqrt{u_0} \nabla \Delta \widetilde\psi_{x} \|_2^2 \Big) \, dt \right)^{\frac{1}{2}} \times\\
& \times \left( \int_0^\infty \sum_{j=1}^2 \Big( \|u_0 \Delta \psi_{jx}\|_2^2  +  \varepsilon \| \sqrt{u_0} \nabla \Delta \psi_{jx} \|_2^2 \Big) \, dt \right)^{\frac{1}{2}}\\
\lesssim & c_0 \sup_{t \in (0,\infty)} \|\sqrt u_0 \Delta \widetilde\phi_x\|_{2} \left( \int_0^\infty \Big( \|u_0 \Delta \widetilde\psi_x\|_2^2  +  \varepsilon \| \sqrt{u_0} \nabla \Delta \widetilde\psi_{x} \|_2^2 \Big) \, dt \right)^{\frac{1}{2}}.
\end{align*}
This implies 
$$
\| \widetilde \phi\|_{\cX} \lesssim c_0 \| \widetilde \psi\|_{\cX},
$$
and hence $\cS$ is a contraction mapping when $c_0$ is small. By the contraction mapping theorem, there exists a unique solution $\phi$ satisfying the estimate \eqref{est:L_nonlinear}. Finally, the exponential decay estimate \eqref{est:L_nonlinear_enhance} follows from \eqref{est:L_nonlinear_final}, \eqref{phi_x_H1_estimate}, and the weighted Hardy inequality
$$
\varepsilon^{1/3} \|\sqrt{u_0} \Delta \phi_x\|_2^2 \lesssim \|u_0 \Delta \phi_x\|_2^2 + \varepsilon \| \sqrt{u_0} \nabla \Delta \phi_x \|_2^2
$$
from \eqref{est:L_weighted_Hardy}. This concludes the proof of Theorem \ref{Thm_L_nonlinear}.


\section{Stability in the long channel}\label{sec:stability_long}

In this section, we prove Theorems \ref{Thm_H_linear} and \ref{Thm_H_nonlinear}, assuming Theorem \ref{Thm_H_Existence}, which will be proved in Section \ref{sec:construction} too.

\subsection{Preliminary estimates} In this subsection, we establish some preliminary estimates. Many of the estimates are similar to what we have obtained in Section \ref{sec:L_prelim}.

\begin{lemma}\label{lem_hardy1_narrow}
For $f \in H^1(\Omega)$, $\sigma > 0$, we have
\begin{equation}\label{Hardy1_narrow}
\int u_0^i f^2 \lesssim (\sigma \varepsilon^{1/3}H)^{i+1}  \int u_0 f_{y}^2 + (\sigma \varepsilon^{1/3})^{i-2}H^{2-i}  \int u_0^2 f^2, \quad i=0,1.
\end{equation}
\end{lemma}
\begin{proof}
The proof is similar to that of Lemma \ref{lem_hardy1}.
\end{proof}

\begin{lemma}\label{lem_hardy2_narrow}
For $f \in H^1(\Omega)$ with $f = 0$ on $y = H/2$, we have
\begin{equation}\label{Hardy2_narrow}
\int f^2 \lesssim H^2 \int u_0^2 f_{y}^2.
\end{equation}
\end{lemma}
\begin{proof}
Similar as the above, we let $\chi(y)$ be a smooth cut-off function supported in $[0,1]$, and $\chi \equiv 1$ in $[0, 1/2]$. For some small $\sigma > 0$, we write
$$
\int f^2 \lesssim \int f^2 \chi^2 \Big( \frac{y}{\sigma H} \Big) + \int f^2 \chi^2 \Big( \frac{H-y}{\sigma H} \Big) + \int f^2 [1-\chi\Big( \frac{y}{\sigma H} \Big) - \chi \Big( \frac{H-y}{\sigma H} \Big)]^2.
$$
For the first term, we proceed as in the proof of Lemma \ref{lem_hardy1_narrow} and obtain
\begin{align*}
\int f^2 \chi^2 \Big( \frac{y}{\sigma H} \Big)\lesssim & \int y^2 f_{y}^2 \chi^2 + \int \Big( \frac{y}{\sigma H} \Big)^2 f^2 |\chi'|^2\\
\lesssim & H^2 \int u_0^2 f_{y}^2 + \int f^2 |\chi'|^2.
\end{align*}
The second term can be handled similarly. For $y \in (\frac{\sigma}{2} H , (1-\frac{\sigma}{2})H)$, we write
$$
f(x,y) = \int_{y}^{\frac{H}{2}}f_y(x,s)\, ds.
$$
Therefore,
$$
\begin{aligned}
 &\int f^2 |\chi' \Big( \frac{y}{\sigma H} \Big)|^2 +  \int f^2 |\chi' \Big( \frac{H-y}{\sigma H} \Big)|^2 + \int f^2  [1-\chi\Big( \frac{y}{\sigma H} \Big) - \chi \Big( \frac{H-y}{\sigma H} \Big)]^2\\
\lesssim & \int_{\frac{\sigma}{2} H}^{(1-\frac{\sigma}{2})H} \left| \int_{y}^{\frac{H}{2}}f_y(x,s)\, ds \right|^2 \, dy\\
\lesssim & \int_{\frac{\sigma}{2} H}^{(1-\frac{\sigma}{2})H} H  \int_{\frac{\sigma}{2} H}^{(1-\frac{\sigma}{2})H} |f_y(x,s)|^2\, ds \, dy\\
\lesssim & H^2 \int_{\frac{\sigma}{2} H}^{(1-\frac{\sigma}{2})H} u_0^2(s)|f_y(x,s)|^2\, ds,
\end{aligned}
$$
where we used the fact $u_0(s) \gtrsim \sigma$ when $s \in (\frac{\sigma}{2} H , (1-\frac{\sigma}{2})H)$. This concludes the proof.
\end{proof}

\begin{lemma}\label{lem_H1} We have the following estimates
\begin{equation}\label{H1_control_by_R}
\|\phi_x\|_2 + \|\phi_y\|_2 \lesssim \|u_0 \nabla q\|_2 \lesssim H \|R\|_2,
\end{equation}
\begin{equation}\label{Laplace_control_by_R}
\|u_0 \Delta \phi \|_2 \lesssim \|R\|_2,
\end{equation}
where $R$ is given in \eqref{def:R}.
\end{lemma}
\begin{proof}
Recall \eqref{quotient}, the definition of $q$, we have
\begin{align*}
(q, R) =& (q, u_0(u_0 q_{yy} + 2u_{0y}q_y + u_{0yy}q + u_s q_{xx})) - (q, u_{0yy}\phi)\\
=& - (u_0^2 q_y, q_y) - 2(u_0 u_{0y} q, q_y) + 2(u_0 u_{0y} q, q_y) + (q, u_0 u_{0yy} q) - (u_0^2 q_x, q_x) - (q, u_{0yy}u_0 q)\\
=& -(u_0^2 \nabla q, \nabla q).
\end{align*}
Therefore, by \eqref{Hardy2_narrow} with $f = q$,
\begin{align*}
 \|u_0 \nabla q\|_2^2 = |(q, R)| \le \|q\|_2\|R\|_2 \lesssim H\|u_0 \nabla q\|_2\|R\|_2,
\end{align*}
which implies
$$
\|u_0 \nabla q\|_2 \lesssim  H\|R\|_2.
$$
Therefore, by \eqref{Hardy2_narrow} with $f = q$ again, we have
\begin{align*}
\|\phi_x\|_2 + \| \phi_y\|_2 =& \|u_0 q_x \|_2 + \|u_0 q_y + u_{0y}q \|_2\\
\lesssim& \|u_0 q_x \|_2 + \|u_0 q_y \|_2 + \frac{1}{H} \| q \|_2\\
\lesssim& \|u_0 \nabla q\|_2 \lesssim H \|R\|_2.
\end{align*}
This concludes the proof of \eqref{H1_control_by_R}. To prove \eqref{Laplace_control_by_R}, we use Poincare's inequality in $y$ and \eqref{H1_control_by_R} to obtain
\begin{align*}
\|u_0 \Delta \phi \|_2 \lesssim \|R\|_2 + \| u_{0yy} \phi\|_2 \lesssim \|R\|_2 + \frac{1}{H}\| \phi_y\|_2 \lesssim \|R\|_2.
\end{align*}
\end{proof}

Very often, we use \eqref{Hardy1_narrow} with $f= \Delta \phi$ or $f = \Delta \phi_x$. Together with Lemma \ref{lem_H1}, we have the following.

\begin{corollary}
For $\sigma > 0$, we have
\begin{equation}\label{est:H_weighted_Hardy}
\begin{aligned}
\varepsilon^{\frac{1}{3}}\| \Delta \phi\|_2 & \lesssim \frac{\sqrt\sigma}{\sqrt{H}} \sqrt\varepsilon \| \sqrt{\frac{u_0}{-u_{0yy}}} \Delta \phi_y\|_2 + \sigma^{-1} \|\frac{R}{\sqrt{-u_{0yy}}}\|_2,\\
\varepsilon^{\frac{1}{6}}\| \sqrt{\frac{u_0}{-u_{0yy}}}  \Delta \phi\|_2 & \lesssim \sigma \sqrt\varepsilon \| \sqrt{\frac{u_0}{-u_{0yy}}} \Delta \phi_y\|_2 + \sigma^{-1/2} \sqrt{H}  \|\frac{R}{\sqrt{-u_{0yy}}}\|_2.\\
\end{aligned}
\end{equation}
\end{corollary}

By the Lemmas above, we have the following mixed norm estimates.

\begin{lemma}\label{lem:mix_norm_H}
\begin{equation}\label{est:mix_norm_H}
\begin{aligned}
\| \nabla \phi\|_{L^2_x L^\infty_y} &\lesssim_H \varepsilon^{-\frac{1}{6}} \Big( \|\frac{R}{\sqrt{-u_{0yy}}}\|_2 + \sqrt{\varepsilon} \| \sqrt{\frac{u_0}{-u_{0yy}}} \Delta \phi_y\|_2\Big),\\
\| \nabla \phi\|_{L^\infty_x L^2_y} &\lesssim_H \varepsilon^{-\frac{1}{6}}  \Big( \|\frac{R}{\sqrt{-u_{0yy}}}\|_2 + \sqrt{\varepsilon} \| \sqrt{\frac{u_0}{-u_{0yy}}} \Delta \phi_{y}\|_2\Big).
\end{aligned}
\end{equation}
\end{lemma}

\begin{proof}
The proofs of these estimates are almost identical to that of Lemma \ref{lem:mix_norm}.
\end{proof}

\begin{lemma}\label{lem:mix_norm_Hardy}
\begin{equation}\label{est:mix_norm_Hardy}
\begin{aligned}
\| \frac{\phi}{\sqrt{u_0}}\|_{L^2_x L^\infty_y} &\lesssim_H \|\phi_y\|_2 \lesssim_H \|\frac{R}{\sqrt{-u_{0yy}}}\|_2,\\
\| \frac{\phi_x}{\sqrt{u_0}}\|_{L^2_x L^\infty_y} &\lesssim_H \|\phi_{xy}\|_2 \lesssim_H \varepsilon^{-\frac{1}{3}} \Big( \|\frac{R}{\sqrt{-u_{0yy}}}\|_2 + \sqrt{\varepsilon} \| \sqrt{\frac{u_0}{-u_{0yy}}} \Delta \phi_{y}\|_2\Big).
\end{aligned}
\end{equation}
\end{lemma}

\begin{proof}
For the first estimate, we first consider when $0 < y < \frac{H}{10}$. Then
$$
\left|\frac{\phi}{\sqrt{u_0}} \right| \lesssim_H \frac{|\phi(x,y)|}{\sqrt{y}} = \frac{\left| \int_0^y \phi_{y}(x,y') \, dy'\right|}{\sqrt{y}} \lesssim \|\phi_{y}\|_{L^2_y}.
$$
Same estimate holds when $\frac{9H}{10} < y < H$. When $\frac{H}{10} \le y \le \frac{9H}{10}$, we simply have
$$
\left|\frac{\phi}{\sqrt{u_0}} \right| \lesssim_H |\phi(x,y)| = \left| \int_0^y \phi_{y}(x,y') \, dy'\right| \lesssim_H \|\phi_{y}\|_{L^2_y}.
$$
Therefore, combining with \eqref{H1_control_by_R}, we have
$$
\| \frac{\phi}{\sqrt{u_0}}\|_{L^2_x L^\infty_y} \lesssim_H \| \phi_y \|_2 \lesssim_H   \|\frac{R_x}{\sqrt{-u_{0yy}}}\|_2.
$$
For the second estimate, we follow the same computations to obtain
$$
\| \frac{\phi_x}{\sqrt{u_0}}\|_{L^2_x L^\infty_y} \lesssim_H \|\phi_{xy}\|_2.
$$
Then we use the estimate \eqref{est:H_weighted_Hardy}. This concludes the proof.
\end{proof}

Next, we derive some mixed norm estimate for $\Delta \phi$.

\begin{lemma}
For any $2 < q < \infty$, we have
\begin{equation}\label{est:Delta_phi_mix}
\begin{aligned}
\| \Delta \phi \|_{L^2_x L^q_y} &\lesssim_{q,H} \| \Delta \phi \|_{2} + \| \sqrt{u_0} \Delta \phi_y \|_{2},\\
\| \Delta \phi \|_{L^2_x L^\infty_y}&\lesssim_{H} \| \Delta \phi \|_{2} + \varepsilon^{0-}\| \sqrt{u_0} \Delta \phi_y \|_{2}.
\end{aligned}
\end{equation}
\end{lemma}
\begin{proof}
To prove the first estimate, we separate three cases. When $H/10 < y < 9H/10$, we have that
$$
|\Delta \phi (x,y)| \lesssim |\Delta \phi (x,y_1)| + \Big| \int_y^{y_1} |\sqrt{u_0} \Delta \phi_y (x,s)| \, ds  \Big|
$$
for some $y_1 \in (H/10, 9H/10)$. Therefore, by direct integration and H\"older's inequality, we have
$$
\Bigg( \int_0^L \Big( \int_{H/10}^{9H/10} |\Delta \phi (x,y)|^q \, dy \Big)^{2/q} \Bigg)^{1/2} \lesssim_{q,H} \| \Delta \phi \|_{L^2} + \| \sqrt{u_0} \Delta \phi_y \|_{L^2}.
$$
When $0< y < H/10$, we have that for some $y_2 \in (0, H/10)$,
\begin{align*}
|\Delta \phi (x,y)| &\lesssim |\Delta \phi (x,y_2)| + \Big| \int_y^{y_2} |\frac{1}{\sqrt{s}}\sqrt{u_0} \Delta \phi_y (x,s)| \, ds  \Big|\\
&\lesssim |\Delta \phi (x,y_2)| + |\log y|^{1/2} \Big( \int_0^H u_0 |\Delta \phi_y| \, ds \Big)^{1/2}.
\end{align*}
Since $|\log y| \in L^q(0,H)$ for any $q < \infty$, we have 
$$
\Bigg( \int_0^L \Big( \int_{0}^{H/10} |\Delta \phi (x,y)|^q \, dy \Big)^{2/q} \Bigg)^{1/2} \lesssim_q \| \Delta \phi \|_{L^2} + \| \sqrt{u_0} \Delta \phi_y \|_{L^2}.
$$
The same argument applies when $9H/10< y < H$, which concludes the proof of the first estimate. For the second one, we separate the cases when $\varepsilon^{0+} < y < H - \varepsilon^{0^+}$, when $0 < y < \varepsilon^{0+}$, and when $H - \varepsilon^{0^+} < y < H$.

When $\varepsilon^{0+} < y < H - \varepsilon^{0^+}$, we have
$$
|\Delta \phi (x,y)| \lesssim_H |\Delta \phi (x,y_1)| + \varepsilon^{0-} \Big| \int_y^{y_1} |\sqrt{u_0} \Delta \phi_y (x,s)| \, ds  \Big|
$$
for some $y_1 \in (\varepsilon^{0+}, H - \varepsilon^{0+})$. For the other two cases, we have
$$
|\Delta \phi (x,y)| \lesssim_H |\Delta \phi (x,y_2)| + |\log \varepsilon|^{1/2} \Big( \int_0^H u_0 |\Delta \phi_y| \, ds \Big)^{1/2}
$$
for some $y_2 \in (0, \varepsilon^{0+})$ or $(H - \varepsilon^{0+}, H)$. This concludes the proof.
\end{proof}

We can control the $L^\infty$ norm of $\nabla \phi$ by some mixed norm of $\Delta \phi$ as follows.

\begin{lemma}
\begin{equation}\label{est:nabla_phi1}
\| \nabla \phi\|_{L^\infty} \lesssim_{p,H} \| \Delta \phi\|_{L^p_yL^2_x}\quad \mbox{for any}~ p > 2.
\end{equation}
\end{lemma}
\begin{proof}
For the simplicity of the presentation, we prove this Lemma with $H = 1$. Due to the viscous inflow boundary condition \eqref{viscous_bc}, we take the Fourier basis in $x$ to be
$$
\varphi_m(x)=\sin(\lambda_m x),\quad \lambda_m=\frac{(2m+1)\pi}{2L},\quad m=0,1,2,\dots.
$$
We write
\begin{align*}
\phi(x,y)&=\sum_{m\ge0} a_m(y)\varphi_m(x),\\
\Delta \phi(x,y)&=\sum_{m\ge0} b_m(y)\varphi_m(x).
\end{align*}
Then we have
\begin{equation}\label{eq:mode-ode}
a_m''(y)-\lambda_m^2 a_m(y)=b_m(y),\qquad y\in(0,1),
\end{equation}
with Dirichlet boundary condition \(a_m(0)=a_m(1)=0\). It is classical that one can solve \eqref{eq:mode-ode} through the Green's function
$$
G_m(y,s) = \frac{1}{\lambda_m \sinh \lambda_m} \left\{
\begin{aligned}
&\sinh(\lambda_m y) \sinh (\lambda_m (1-s)), \quad && 0 \le y \le s \le 1,\\
&\sinh(\lambda_m s) \sinh (\lambda_m (1-y)), \quad && 0 \le s < y \le 1.
\end{aligned}
\right.
$$
Then we can write
$$
a_m(y) = \int_0^1 G(y,s) b_m(s) \, ds.
$$
It is straightforward to show that
\begin{equation}\label{est:Green}
|G_m(y,s)| \lesssim \frac{1}{\lambda_m}e^{- \lambda_m |s-y|}.
\end{equation}
We only prove \eqref{est:Green} when $y \le s$, as the other case follows similarly. When $\lambda_m > 1$, we have $\sinh \lambda_m \gtrsim e^{\lambda_m}$, and hence
$$
|G_m(y,s)| \lesssim \frac{e^{\lambda_m y}e^{\lambda_m (1-s)}}{\lambda_m e^{\lambda_m}} \le \frac{1}{\lambda_m}e^{ \lambda_m (y-s)}.
$$
When $0 \le \lambda_m \le 1$, we have $e^{ \lambda_m (y-s)} \ge e^{-1} \gtrsim 1$ and $\sinh \lambda_m \le \sinh 1 \lesssim 1$. Therefore,
$$
|G_m(y,s)| \lesssim \frac{(\sinh \lambda_m)^2}{\lambda_m \sinh \lambda_m} \lesssim  \frac{1}{\lambda_m}e^{ \lambda_m (y-s)}.
$$
Similarly, one can prove that
\begin{equation}\label{est:Green_y}
|\partial_y G_m(y,s)| \lesssim e^{- \lambda_m |s-y|}.
\end{equation}
By \eqref{est:Green_y}, Fubini Theorem, and H\"older's inequality, we can estimate 
\begin{align*}
|\phi_y(x,y)| &= \left| \sum_m a_m'(y) \varphi(x) \right|\\
& \lesssim \sum_m \int_0^1 e^{- \lambda_m |s-y|} |b_m(s)| \, ds\\
&= \int_0^1 \sum_m e^{- \lambda_m |s-y|} |b_m(s)| \, ds\\
&\le \int_0^1 \Big( \sum_m e^{- 2\lambda_m |s-y|} \Big)^{1/2} \Big( \sum_m |b_m(s)|^2 \Big)^{1/2}\, ds\\
&\le \Bigg( \underbrace{\int_0^1 \Big( \sum_m e^{- 2\lambda_m |s-y|}\Big)^{p'/2} \, ds}_{\I}  \Bigg)^{1/p'} \| \Delta \phi\|_{L^p_y L^2_x},
\end{align*}
where $1/p' = 1 - 1/p$. It remains to show the integral $\I$ converges, and it suffices to show that $( \sum_m e^{- 2\lambda_m |s-y|})^{p'/2}$ is integrable near $s = y$. For $s \neq y$, we can directly compute that
\begin{align*}
 \sum_m e^{- 2\lambda_m |s-y|} = \frac{e^{- \pi|s-y|/L}}{1 - e^{-2\pi|s-y|/L}} \lesssim \frac{1}{|s-y|}, \qquad \mbox{when}~s \sim y.
\end{align*}
Since $p > 2$, $p'/2 < 1$. Therefore, the integral $\I$ converges. We can conclude that
$$
\|\phi_y\|_{L^\infty} \lesssim_p  \| \Delta \phi\|_{L^p_y L^2_x}.
$$
Similarly, by \eqref{est:Green}, we can estimate
\begin{align*}
|\phi_x(x,y)| = \left| \sum_m a_m(y) \varphi'(x) \right| \lesssim \sum_m \int_0^1 e^{- \lambda_m |s-y|} |b_m(s)| \, ds.
\end{align*}
Following the same argument gives
$$
\|\phi_x\|_{L^\infty} \lesssim_p  \| \Delta \phi\|_{L^p_y L^2_x}.
$$
This concludes the proof of this lemma.
\end{proof}

Combining the two lemmas above, we have the following $L^\infty$ estimate for $\nabla \phi$.

\begin{corollary}
\begin{equation}\label{est:nabla_phi_L_infty}
\| \nabla \phi\|_{L^\infty} \lesssim_H \varepsilon^{-(\frac{1}{3}+)} \Big( \|\frac{R}{\sqrt{-u_{0yy}}}\|_2 + \sqrt{\varepsilon} \| \sqrt{\frac{u_0}{-u_{0yy}}} \Delta \phi_y\|_2\Big).
\end{equation}
\end{corollary}
\begin{proof}
By \eqref{est:nabla_phi1}, Minkowski's inequality, and H\"older's inequality, we have
$$
\| \nabla \phi\|_{L^\infty} \lesssim \| \Delta \phi\|_{L^p_yL^2_x} \lesssim \| \Delta \phi\|_{L^2_xL^p_y} \lesssim \|\Delta \phi\|_{L^2_xL^2_y}^\theta \|\Delta \phi\|_{L^2_xL^q_y}^{1-\theta},
$$
where
$$
\frac{1}{p} = \frac{\theta}{2} + \frac{1-\theta}{q}.
$$
Therefore, for any $\delta > 0$, we can choose $p>2$ and close to $2$, and $q$ sufficiently large, such that
$\theta = 1-6\delta$. Therefore, by \eqref{est:H_weighted_Hardy} and \eqref{est:Delta_phi_mix}, we have
\begin{align*}
\|\nabla \phi\|_{L^\infty} &\lesssim_\delta \|\Delta \phi\|_{L^2_xL^2_y}^{1-6\delta} \|\Delta \phi\|_{L^2_xL^q_y}^{6\delta}\\
&\lesssim_\delta \varepsilon^{-1/3 - \delta} \Big( \|\frac{R}{\sqrt{-u_{0yy}}}\|_2 + \sqrt{\varepsilon} \| \sqrt{\frac{u_0}{-u_{0yy}}} \Delta \phi_y\|_2\Big).
\end{align*}
This concludes the proof.
\end{proof}

Finally, we collect some other useful mixed norm estimates.

\begin{lemma}
\begin{equation}\label{est:mix_norm_misc}
\begin{aligned}
\| \sqrt{u_0} \Delta \phi\|_{L^\infty_x L^2_y} &\lesssim_{H} \varepsilon^{-\frac{1}{3}} \Big( \|\frac{R}{\sqrt{-u_{0yy}}}\|_2 + \sqrt{\varepsilon} \| \sqrt{\frac{u_0}{-u_{0yy}}} \nabla \Delta \phi\|_2\Big),\\
\|\phi\|_{L^\infty_xL^2_y} &\lesssim_{L}  \|\frac{R}{\sqrt{-u_{0yy}}}\|_2.
\end{aligned}
\end{equation}
\end{lemma}

\begin{proof}For the first estimate
\begin{align*}
\int_0^H u_0| \Delta \phi|^2 \, dy &= 2 \int_0^x \int_0^H u_0 \Delta \phi_x \Delta \phi \, dy dx\\
&\lesssim \|\sqrt{u_0} \Delta \phi\|_2 \| \sqrt{u_0} \Delta \phi_{x}\|_2\\
&\lesssim_{H} \varepsilon^{-\frac{2}{3}} \Big( \|\frac{R}{\sqrt{-u_{0yy}}}\|_2^2 + \varepsilon \| \sqrt{\frac{u_0}{-u_{0yy}}} \nabla \Delta \phi\|_2^2\Big)
\end{align*}
By \eqref{est:H_weighted_Hardy}. For the second one, we use Sobolev inequality in $x$ and \eqref{H1_control_by_R} to obtain
$$
\|\phi\|_{L^\infty_xL^2_y} \lesssim_{L} \|\phi_x\|_2 \lesssim_{L}  \|\frac{R}{\sqrt{-u_{0yy}}}\|_2.
$$
\end{proof}

\subsection{Linear stability estimates}\label{sec:H_linear}

In this section, we prove Theorem \ref{Thm_H_linear}. 

Define
\begin{equation}\label{weight}
W(x) = NL - x,
\end{equation}
where $N$ is a large number to be determined later. Then $W' = -1$. We test the linearized Navier-Stokes equation \eqref{Linearized_NS} by $\frac{RW}{-u_{0yy}}$, and estimate each term.

\subsubsection{Rayleigh term}
Note that the main part of the Rayleigh terms can be written as
$$u_s \Delta \phi_x - \Delta u_s \phi_x = R_x + (u_s-u_0)\Delta \phi_x - (\Delta u_s - u_{0yy})\phi_x.$$
We then have
\begin{equation}\label{Ray_nar_1}
\begin{aligned}
(u_s \Delta \phi_x - \Delta u_s \phi_x , W \frac{R}{-u_{0yy}})=& (R_x,  W \frac{R}{-u_{0yy}}) + ((u_s-u_0)W\Delta\phi_{x},u_0\frac{\Delta\phi}{ -u_{0yy}})\\
&+((u_s-u_0)W\Delta\phi_{x},\phi) - ((\Delta u_s - u_{0yy})\phi_x,W \frac{R}{-u_{0yy}} ).
\end{aligned}
\end{equation}
The first term
\begin{equation}\label{Ray_nar_1-1}
(\ref{Ray_nar_1}.1) = \frac{1}{2}(R, \frac{R}{-u_{0yy}}) + \frac{1}{2}(R,W\frac{R}{-u_{0yy}})_{x=L}
\end{equation}
gives the crucial positive contribution.
\begin{equation}\label{Ray_nar_1-2}
(\ref{Ray_nar_1}.2) = \frac{1}{2}((u_s-u_0)W\Delta\phi,\frac{u_0\Delta\phi}{-u_{0yy}})_{x=L}+\frac{1}{2}((u_s-u_0)\Delta
\phi,\frac{u_0\Delta\phi}{-u_{0yy}}) - \frac{1}{2} (u_{sx} W\Delta
\phi,\frac{u_0\Delta\phi}{-u_{0yy}}).
\end{equation}
The first two terms of \eqref{Ray_nar_1-2} are positive. By \eqref{ss_estimates_H}, \eqref{est:Delta_phi_mix}, and \eqref{est:mix_norm_misc}, we can estimate
\begin{equation}\label{Ray_nar_1-2-1}
\begin{aligned}
(\ref{Ray_nar_1-2}.3) &\lesssim_{H,L} \| u_{sx}\|_2 \| \Delta \phi\|_{L^2_x L^\infty_y} \| u_0 \Delta \phi\|_{L^\infty_x L^2_y}\\
&\lesssim_{A,L,H} \varepsilon^{\frac{1}{6}-} \Big(  \|\frac{R}{\sqrt{-u_{0yy}}}\|_2^2 + \varepsilon \| \sqrt{\frac{u_0}{-u_{0yy}}} \Delta \phi_y\|_2^2 \Big).
\end{aligned}
\end{equation}
Next, we write
\begin{equation}\label{Ray_nar_1-3}
(\ref{Ray_nar_1}.3) = ((u_s-u_0)W\Delta\phi,\phi)_{x=L}-((u_s-u_0)W\Delta
\phi,\phi_{x})-(u_{sx}W\Delta
\phi,\phi)+((u_s-u_0)\Delta\phi,\phi).
\end{equation}
We note
\begin{equation}\label{Ray_nar_1-3-1}
\begin{aligned}
|(\ref{Ray_nar_1-3}.1)|  \leq &\frac{1}{4}((u_s-u_0)W\Delta\phi,\frac{u_0\Delta\phi}{-u_{0yy}})_{x=L}
+O(\varepsilon^{1/3})A(W\{\frac{-u_{0yy}}{u_0}\}\phi, \phi)_{x=L}\\
\leq &\frac{1}{4}((u_s-u_0)W\Delta\phi,\frac{u_0\Delta\phi}{-u_{0yy}})_{x=L}
+O(\varepsilon^{1/3})AH^{-2}(\phi_x, \frac{\phi}{u_0})\\
\leq &\frac{1}{4}((u_s-u_0)W\Delta\phi,\frac{u_0\Delta\phi}{-u_{0yy}})_{x=L}
+O(\varepsilon^{1/3})AH^{-1}\|\phi_x\|_2\|\phi_y\|_2\\
\leq &\frac{1}{4}((u_s-u_0)W\Delta\phi,\frac{u_0\Delta\phi}{-u_{0yy}})_{x=L}
+O(\varepsilon^{1/3})AH^{-1} \|\frac{R}{\sqrt{-u_{0yy}}}\|_2^2,
\end{aligned}
\end{equation}
where we used Hardy's inequality in $y$ and \eqref{H1_control_by_R}. Then (\ref{Ray_nar_1-3-1}.1) and (\ref{Ray_nar_1-3-1}.2) can be absorbed by (\ref{Ray_nar_1-2}.1) and (\ref{Ray_nar_1-1}.1), respectively. Next, we can estimate by \eqref{ss_estimates_H} and \eqref{H1_control_by_R},
\begin{equation}\label{Ray_nar_1-3-2}
\begin{aligned}
(\ref{Ray_nar_1-3}.2)  =& -((u_s-u_0)W\phi_{yy},\phi_{x})-((u_s-u_0)W\phi_{xx}
,\phi_{x})\\
 =&((u_s-u_0)W\phi_{y},\phi_{xy}) + ((u_s-u_0)_y W\phi_{y},\phi_{x}) -((u_s-u_0)W\phi_{xx} ,\phi_{x})\\
=&\frac{1}{2}((u_s-u_0)W\phi_{y},\phi_{y})_{x=L} - \frac{1}{2}(u_{sx}W\phi_{y},\phi_{y}) + \frac{1}{2}((u_s-u_0)\phi_{y},\phi_{y})+ ((u_{sy}-u_{0y}) W\phi_{y},\phi_{x})\\
&+ \frac{1}{2}((u_s-u_0)W\phi_{x},\phi_{x})_{x=0} + \frac{1}{2}(u_{sx}W\phi_{x},\phi_{x}) - \frac{1}{2}((u_s-u_0)\phi_{x},\phi_{x})\\
=&\frac{1}{2}((u_s-u_0)W\phi_{y},\phi_{y})_{x=L}+ \|u_{sx}\|_2\|\nabla\phi\|_{L^2_xL^\infty_y}\|\nabla\phi\|_{L^\infty_xL^2_y} + O(\varepsilon^{\frac{1}{3}})\|\nabla\phi\|_2^2\\
& +\|u_{sy}-u_{0y}\|_2 \|  \phi_x\|_{L^2_xL^\infty_y} \|\phi_y\|_{L^\infty_xL^2_y}+ \frac{1}{2}((u_s-u_0)W\phi_{x},\phi_{x})_{x=0} \\
=& \frac{1}{2}((u_s-u_0)W\phi_{y},\phi_{y})_{x=L} + \frac
{1}{2}((u_s-u_0)W\phi_{x},\phi_{x})_{x=0}\\
&+ O(\varepsilon^{\frac{1}{3}})\Big(  \|\frac{R}{\sqrt{-u_{0yy}}}\|_2^2 + \varepsilon \| \sqrt{\frac{u_0}{-u_{0yy}}} \Delta \phi_y\|_2^2 \Big).
\end{aligned}
\end{equation}
The first two terms are positive, and the last term can be absorbed by (\ref{Ray_nar_1-1}.1). Similarly, by \eqref{ss_estimates_H}, \eqref{H1_control_by_R}, and \eqref{est:mix_norm_H}, we have
\begin{equation}\label{Ray_nar_1-3-3}
\begin{aligned}
|(\ref{Ray_nar_1-3}.3+4)| =& |-(u_{sx}W\Delta \phi,\phi)+((u_s-u_0)(\phi_{xx}+ \phi_{yy}),\phi)|\\
=& |(u_{sx}W\Delta \phi,\phi)
+ ((u_{s}-u_0) \phi_x, \phi_x) + (u_{sx} \phi_x, \phi)\\
& + ((u_{s}-u_0) \phi_y, \phi_y) + ((u_{sy}-u_{0y}) \phi_y, \phi)|\\
\lesssim&_{A,L,H}~ \|u_{sx}\|_2 \| \Delta \phi\|_{L^2_xL^\infty_y} \|\phi\|_{L^\infty_xL^2_y} + \varepsilon^{\frac{1}{3}} \|\nabla \phi\|_2^2\\
& + \|u_{sx}\|_2 \|  \phi_x\|_{L^2_xL^\infty_y} \|\phi\|_{L^\infty_xL^2_y} + \|u_{sy}-u_{0y}\|_2 \|  \phi_x\|_{L^2_xL^\infty_y} \|\phi\|_{L^\infty_xL^2_y}\\
\lesssim&_{A,L,H}~ \varepsilon^{\frac{1}{3}} \Big(  \|\frac{R}{\sqrt{-u_{0yy}}}\|_2^2 + \varepsilon \| \sqrt{\frac{u_0}{-u_{0yy}}} \Delta \phi_y\|_2^2 \Big).
\end{aligned}
\end{equation}
Finally, we write
\begin{equation}\label{Ray_nar_1-4}
(\ref{Ray_nar_1}.4) = - ((\Delta u_s - \Delta u_a) \phi_x, W \frac{R}{-u_{0yy}}) -((\Delta u_a - u_{0yy}) \phi_x, W \frac{R}{-u_{0yy}}).
\end{equation}
For the first term, we use \eqref{ss_estimates_H}, \eqref{est:nabla_phi_L_infty}, \eqref{est:H_weighted_Hardy}, \eqref{est:mix_norm_H}, and \eqref{est:mix_norm_Hardy} to estimate
\begin{equation}\label{Ray_nar_1-4-1}
\begin{aligned}
|(\ref{Ray_nar_1-4}.1)| \le& |((\Delta u_s - \Delta u_a) \phi_x, W \frac{u_0 \Delta\phi }{-u_{0yy}})| + |((\Delta u_s - \Delta u_a) \phi_x, W \phi)| \\
 \lesssim&_H \| \sqrt{u_0}(\Delta u_s - \Delta u_a)\|_{2} \| \phi_x\|_{L^2_xL^\infty_y} \|\sqrt{u_0} \Delta\phi \|_{L^\infty_xL^2_y}\\
  &+ \| \sqrt{u_0}(\Delta u_s - \Delta u_a)\|_{2}\| \frac{\phi}{\sqrt{u_0}}\|_{L^2_x L^\infty_y} \| \phi_x\|_{L^\infty_xL^2_y}\\
\lesssim&_{A,L,H} \varepsilon^{\frac{1}{3}-}  \Big( \|\frac{R}{\sqrt{-u_{0yy}}}\|_2^2 + \varepsilon \| \sqrt{\frac{u_0}{-u_{0yy}}} \nabla\Delta \phi\|_2^2\Big),
\end{aligned}
\end{equation}
and the second term can be simply estimated by \eqref{ss_estimates_H} and \eqref{H1_control_by_R},
\begin{equation}\label{Ray_nar_1-4-2}
\begin{aligned}
|(\ref{Ray_nar_1-4}.2)| \lesssim_{A,L,H}\varepsilon^{\frac{1}{3}} \| \phi_x\|_{2} \|\frac{R}{\sqrt{-u_{0yy}}}\|_2 \lesssim_{A,L,H}\varepsilon^{\frac{1}{3}}\|\frac{R}{\sqrt{-u_{0yy}}}\|_2^2.
\end{aligned}
\end{equation}

For rest of the Rayleigh terms, we have
\begin{equation}\label{Ray_nar_2}
\begin{aligned}
(v_s \Delta \phi_y - \Delta v_s \phi_y,  W \frac{R}{-u_{0yy}}) =& (v_s \Delta \phi_y, W \frac{u_0 \Delta \phi}{-u_{0yy}} ) - (\Delta v_s \phi_y, W \frac{u_0 \Delta \phi}{-u_{0yy}} )\\
& + (v_s \Delta \phi_y, W \phi) - (\Delta v_s \phi_y, W \phi).
\end{aligned}
\end{equation}
By \eqref{ss_estimates_H} and \eqref{est:H_weighted_Hardy}, we can estimate
\begin{equation}\label{Ray_nar_2-1}
\begin{aligned}
|(\ref{Ray_nar_2}.1)| &\lesssim_{A,L,H} \varepsilon^{1-} \| \sqrt{\frac{u_0}{-u_{0yy}}} \Delta \phi_{y}\|_2  \| \sqrt{\frac{u_0}{-u_{0yy}}} \Delta \phi\|_2  \\
&\lesssim_{A,L,H}  \varepsilon^{\frac{1}{3}-} \Big( \|\frac{R}{\sqrt{-u_{0yy}}}\|_2^2 + \varepsilon \| \sqrt{\frac{u_0}{-u_{0yy}}} \Delta \phi_{y}\|_2^2\Big).
\end{aligned}
\end{equation}
For the third term, we use Hardy inequality in $y$, \eqref{ss_estimates_H}, and \eqref{est:H_weighted_Hardy} to estimate
\begin{equation}\label{Ray_nar_2-2}
\begin{aligned}
|(\ref{Ray_nar_2}.3)| &\lesssim_{A,L,H} \varepsilon^{1-} \|u_0 \Delta \phi_y\|_2 \| \frac{\phi}{u_0} \|_2\\
&\lesssim_{A,L,H} \varepsilon^{1-} \| \sqrt{\frac{u_0}{-u_{0yy}}} \Delta \phi_{y}\|_2 \|\phi_y\|_2\\
&\lesssim_{A,L,H} \varepsilon^{\frac{1}{2}-} \Big( \|\frac{R}{\sqrt{-u_{0yy}}}\|_2^2 + \varepsilon \| \sqrt{\frac{u_0}{-u_{0yy}}} \Delta \phi_{y}\|_2^2\Big).
\end{aligned}
\end{equation}
For the second and the last term, similar to \eqref{Ray_nar_1-4}, we split $v_s$, and use \eqref{ss_estimates_H}, \eqref{H1_control_by_R}, \eqref{est:H_weighted_Hardy}, \eqref{est:mix_norm_H}, and Hardy inequality in $y$ to estimate
\begin{equation}\label{Ray_nar_2-3}
\begin{aligned}
|(\ref{Ray_nar_2}.2 + 4)| =& |(\Delta v_a \phi_y, W \frac{u_0 \Delta \phi}{-u_{0yy}} ) + ((\Delta v_s - \Delta v_a) \phi_y, W \frac{u_0 \Delta \phi}{-u_{0yy}} )\\
&+ (\Delta v_a \phi_y, W \phi) + ((\Delta v_s - \Delta v_a) \phi_y, W \phi)|\\
\lesssim&_{A,L,H}~ \varepsilon^{\frac{2}{3}} \|\phi_y\|_2 \|u_0 \Delta \phi\|_2 + \|\sqrt{u_0} (\Delta v_s - \Delta v_a)\|_{2} \| \phi_y\|_{L^2_x L^\infty_y} \| \sqrt{u_0} \Delta \phi\|_{L^\infty_x L^2_y}\\
&+ \varepsilon^{\frac{2}{3}} \| \phi_y \|_2 \| \phi \|_y + \|\sqrt{u_0} (\Delta v_s - \Delta v_a)\|_{2} \| \phi_y\|_{L^\infty_x L^2_y} \| \frac{\phi}{\sqrt{u_0}} \|_{L^2_x L^\infty_y} \\
\lesssim&_{A,L,H}~ \varepsilon^{\frac{1}{3}-} \Big( \|\frac{R}{\sqrt{-u_{0yy}}}\|_2^2 + \varepsilon \| \sqrt{\frac{u_0}{-u_{0yy}}} \nabla\Delta \phi\|_2^2\Big).
\end{aligned}
\end{equation}
Estimates for the Rayleigh terms are concluded.

\subsubsection{Dissipation term}
\begin{equation}\label{Dis_nar_1}
-\varepsilon(\Delta^2 \phi,  W \frac{R}{-u_{0yy}}) = - \varepsilon (\Delta^2 \phi, W u_0 \frac{\Delta \phi}{-u_{0yy}}) + \varepsilon (\Delta^2 \phi, W \phi)
\end{equation}
For the first term, we have
\begin{equation}\label{Dis_nar_1-1}
\begin{aligned}
(\ref{Dis_nar_1}.1) =& ~\varepsilon (W \frac{u_0}{-u_{0yy}} \Delta \phi_x, \Delta \phi_x) + \varepsilon (W\frac{u_0}{-u_{0yy}} \Delta \phi_y, \Delta \phi_y)\\
&+  \varepsilon (W' \frac{u_0}{-u_{0yy}} \Delta \phi_x, \Delta \phi)  +  \varepsilon (W \Big( \frac{u_{0y}}{-u_{0yy}} + \frac{u_0u_{0yyy}}{u_{0yy}^2} \Big) \Delta \phi_y, \Delta \phi).
\end{aligned}
\end{equation}
The terms (\ref{Dis_nar_1-1}.1) and (\ref{Dis_nar_1-1}.2) are favorable positive terms. Using \eqref{est:H_weighted_Hardy}, we can estimate
\begin{equation}\label{Dis_nar_1-1-1}
\begin{aligned}
|(\ref{Dis_nar_1-1}.3)| \lesssim &~ \varepsilon \| \sqrt{\frac{u_0}{-u_{0yy}}} \Delta \phi_x\|_2 \|  \sqrt{\frac{u_0}{-u_{0yy}}}\Delta \phi \|_2\\
\lesssim &_H~ \varepsilon^{\frac{1}{3}} \Big( \|\frac{R}{\sqrt{-u_{0yy}}}\|_2^2 + \varepsilon \| \sqrt{\frac{u_0}{-u_{0yy}}} \nabla\Delta \phi\|_2^2\Big).
\end{aligned}
\end{equation}
(\ref{Dis_nar_1-1}.4) is the crucial term. For simplicity of the presentation, we denote $\Xi := \frac{u_{0y}}{-u_{0yy}} + \frac{u_0u_{0yyy}}{u_{0yy}^2}$. Note that $|\Xi| \lesssim H$, and $|\Xi_y| \lesssim 1$. We integrate by parts in $y$ and use the Navier boundary condition \eqref{Navier_bc2}:
\begin{equation}\label{Dis_nar_1-1-2}
\begin{aligned}
(\ref{Dis_nar_1-1}.4) =& ~\frac{\varepsilon}{2} (W \Xi \Delta \phi , \Delta \phi)_{y=H} - \frac{\varepsilon}{2} (W \Xi \Delta \phi , \Delta \phi)_{y=0} - \frac{\varepsilon}{2} (W \Xi_y \Delta \phi, \Delta \phi) \\
=& ~\frac{\varepsilon^{\frac{1}{3}}}{2A^2} (W \Xi \phi_y ,\phi_y)_{y=H} - \frac{\varepsilon^{\frac{1}{3}}}{2A^2} (W \Xi \phi_y ,\phi_y)_{y=0} - \frac{\varepsilon}{2} (W \Xi_y \Delta \phi, \Delta \phi)\\
=& ~\frac{\varepsilon^{\frac{1}{3}}}{A^2} (W \Xi \phi_y ,\phi_{yy}) + \frac{\varepsilon^{\frac{1}{3}}}{2A^2} (W \Xi_y \phi_y ,\phi_y) - \frac{\varepsilon}{2} (W \Xi_y \Delta \phi, \Delta \phi).
\end{aligned}
\end{equation}
The term (\ref{Dis_nar_1-1-2}.1) is the key term, we estimate it by  \eqref{H1_control_by_R} and \eqref{est:H_weighted_Hardy} with $\sigma = (HL)^{\frac{1}{3}}$,
\begin{equation}\label{Dis_nar_1-1-2-1}
\begin{aligned}
|(\ref{Dis_nar_1-1-2}.1)| \lesssim & \frac{\varepsilon^{\frac{1}{3}}HL}{A^2} \|\phi_y\|_2\| \Delta \phi\|_2\\
\lesssim & \frac{H^{\frac{2}{3}}L^{\frac{2}{3}}}{A^2}  \|\frac{R}{\sqrt{-u_{0yy}}}\|_2 \Big( \|\frac{R}{\sqrt{-u_{0yy}}}\|_2 + \sqrt\varepsilon \| \sqrt{\frac{u_0}{-u_{0yy}}} \Delta \phi_y \sqrt{W}\|_2\Big).
\end{aligned}
\end{equation}
The assumption \eqref{H_assumption} is applied here to assure that (\ref{Dis_nar_1-1-2}.1) can be absorbed by (\ref{Ray_nar_1-1}.1) and (\ref{Dis_nar_1-1}.2).
For the last two terms in \eqref{Dis_nar_1-1-2}, we use Lemma \eqref{est:H_weighted_Hardy} and \eqref{H1_control_by_R} to estimate
\begin{equation}\label{Dis_nar_1-1-2-2}
\begin{aligned}
|(\ref{Dis_nar_1-1-2}.2)| + |(\ref{Dis_nar_1-1-2}.3)| \lesssim_A &~ \varepsilon^{\frac{1}{3}}\| \phi_y\|_2^2 + \varepsilon \|\Delta \phi\|_2^2\\
\lesssim_{A,H} &~ \varepsilon^{\frac{1}{3}}  \Big( \|\frac{R}{\sqrt{-u_{0yy}}}\|_2^2 + \varepsilon \| \sqrt{\frac{u_0}{-u_{0yy}}} \Delta \phi_y\|_2^2\Big).
\end{aligned}
\end{equation}
Finally, for the second term in \eqref{Dis_nar_1}, we integrate by parts and obtain
\begin{equation}\label{Dis_nar_1-2}
(\ref{Dis_nar_1}.2) = - \varepsilon(\Delta \phi_y, W \phi_y) - \varepsilon(\Delta \phi_x, W \phi_x) + \varepsilon (\Delta \phi_x, \phi).
\end{equation}
For the first term, we further integrate by parts and use the $\varepsilon$-Navier boundary condition \eqref{Navier_bc2} to obtain 
\begin{equation}\label{Dis_nar_1-2-1}
\begin{aligned}
(\ref{Dis_nar_1-2}.1) &= \varepsilon(\Delta \phi, W \phi_{yy}) - \varepsilon (\Delta \phi, W \phi_y)_{y=H} + \varepsilon (\Delta \phi, W \phi_y)_{y = 0}\\
&= \varepsilon(\Delta \phi, W \phi_{yy}) + \frac{\varepsilon^{\frac{2}{3}}}{A} (\phi_y, W \phi_y)_{y = H} + \frac{\varepsilon^{\frac{2}{3}}}{A} (\phi_y, W \phi_y)_{y = 0}.
\end{aligned}
\end{equation}
Then we can estimate $(\ref{Dis_nar_1-2}.1)$ using \eqref{est:H_weighted_Hardy} and \eqref{est:mix_norm_H},
\begin{equation}\label{Dis_nar_1-2-2}
\begin{aligned}
|(\ref{Dis_nar_1-2}.1)| &\lesssim_{A,H,L} \varepsilon \|\Delta \phi\|_2^2 + \varepsilon^{\frac{2}{3}} \| \phi_y\|_{L^2_x L^\infty_y}^2\\
 &\lesssim_{A,H,L} \varepsilon^{\frac{1}{3}}  \Big( \|\frac{R}{\sqrt{-u_{0yy}}}\|_2^2 + \varepsilon \| \sqrt{\frac{u_0}{-u_{0yy}}} \Delta \phi_y\|_2^2\Big).
\end{aligned}
\end{equation}
For the last two terms in \eqref{Dis_nar_1-2}, we use \eqref{est:mix_norm_Hardy} to estimate
\begin{equation}\label{Dis_nar_1-2-3}
\begin{aligned}
|(\ref{Dis_nar_1-2}.2)| &\lesssim_L \varepsilon \| \sqrt{u_0} \Delta \phi_x \|_2 \| \frac{\phi_x}{\sqrt{u_0}}\|_2 + \varepsilon \| \sqrt{u_0} \Delta \phi_x \|_2 \| \frac{\phi}{\sqrt{u_0}}\|_2\\
&\lesssim_{H,L} \varepsilon^{\frac{1}{6}}  \Big( \|\frac{R}{\sqrt{-u_{0yy}}}\|_2^2 + \varepsilon \| \sqrt{\frac{u_0}{-u_{0yy}}} \nabla \Delta \phi\|_2^2\Big).
\end{aligned}
\end{equation}
Estimates for the dissipation term are concluded.

\subsubsection{Temporal term} For the temporal term, recalling that $W'(x) = -1$, we have
\begin{equation}\label{Tem}
\begin{aligned}
(\Delta \phi_t, \frac{R}{-u_{0yy}}W ) &= (\Delta \phi_t, \frac{u_0}{-u_{0yy} } \Delta \phi W) + (\Delta \phi_t, \phi W)\\
&= \frac{d}{dt} \frac{1}{2} (\Delta {\phi}, \frac{u_0}{-u_{0yy} } \Delta \phi W) - \frac{d}{dt} \frac{1}{2} (\nabla \phi, \nabla \phi W) + (\phi_{tx} , \phi) .
\end{aligned}
\end{equation}
For the last term, we integrate by parts to obtain
\begin{equation}\label{Tem-1}
\begin{aligned}
(\ref{Tem-1}.3) = (\phi_t, \phi)_{x=L} - (\phi_t, \phi_x) = \frac{1}{2} \frac{d}{dt} (\phi,\phi)_{x=L} - (\phi_t, \phi_x).
\end{aligned}
\end{equation}
The estimation of \eqref{Tem} consists of two parts. First, we prove the positivity of 
\begin{equation}\label{Tem_term}
(\Delta {\phi}, \frac{u_0}{-u_{0yy} } \Delta \phi W) - (\nabla \phi, \nabla \phi W) + (\phi,\phi)_{x=L}.
\end{equation}
The second part is the estimation of (\ref{Tem-1}.2). Part one is summarized in the following lemma.
\begin{lemma}\label{lem:positive}
There exists $c > 0$ such that
\begin{equation}\label{Tem_positive}
\begin{aligned}
&(\Delta {\phi}, \frac{u_0}{-u_{0yy} } \Delta \phi W) - (\nabla \phi, \nabla \phi W) + (\phi,\phi)_{x=L}\\
 \ge &c(\Delta {\phi}, \frac{u_0}{-u_{0yy} } \Delta \phi W) + c ( u_0 \nabla q, u_0 \nabla q W) + c(\phi,\phi)_{x=L}.
\end{aligned}
\end{equation}
\end{lemma}
\begin{proof}
First, we claim that
\begin{equation}\label{iden_quotient}
\int |\phi_y|^2 W + \int \frac{u_{0yy}}{u_0} \phi^2 W = \int u_0^2 q_y^2 W.
\end{equation}
Indeed, w e rewrite $\phi$ as $u_0 q$ and integrate by parts,
\begin{align*}
\int |\phi_y|^2 W + \int \frac{u_{0yy}}{u_0} \phi^2 W &= \int \Big[ (u_0 q_y + u_{0y}q)^2 + u_0 u_{0yy} q^2 \Big] W\\
&= \int \Big[ u_0^2 q_y^2 + u_{0y}^2q^2 + 2 u_0 u_{0y} q q_y + u_0 u_{0yy} q^2 \Big] W\\
&= \int \Big[ u_0^2 q_y^2 + u_{0y}^2q^2 - (u_0 u_{0y})_y q^2 + u_0 u_{0yy} q^2 \Big] W\\
&= \int u_0^2 q_y^2 W.
\end{align*}
Then we can estimate
\begin{equation}\label{Tem-2}
\begin{aligned}
 (\ref{Tem_term}.2) &= -  (\nabla \phi, \nabla \phi W) \\ &=  (\Delta \phi, \phi W)-  (\phi_x, \phi)\\
 &=  (\Delta \phi, \phi W) - \frac{1}{2} (\phi ,\phi)|_{x=L}\\
 & \ge - \frac{1}{2} \int \frac{u_0}{-u_{0yy}} |\Delta \phi|^2 W - \frac{1}{2} \int \frac{-u_{0yy}}{u_0} \phi^2 W - \frac{1}{2} (\phi ,\phi)|_{x=L}\\
 &= - \frac{1}{2} \int \frac{u_0}{-u_{0yy}} |\Delta \phi|^2 W + \frac{1}{2} \int u_0^2 q_y^2 W - \frac{1}{2} \int |\phi_y|^2 W - \frac{1}{2} (\phi ,\phi)|_{x=L},\\
\end{aligned}
\end{equation}
where we used \eqref{iden_quotient} in the last line. Moving the term $- \frac{1}{2} \int |\phi_y|^2 W$ to the left-hand side of \eqref{Tem-2}, and a piece $- \frac{1}{2} \int |\phi_x|^2 W$ from $-  (\nabla \phi, \nabla \phi W)$ to the right-hand side of \eqref{Tem-2}, we can conclude that
\begin{equation}\label{Tem-2_1}
\begin{aligned}
 (\ref{Tem_term}.2) &= -  (\nabla \phi, \nabla \phi W) \\
 & \ge - \int \frac{u_0}{-u_{0yy}} |\Delta \phi|^2 W + \int u_0^2 q_y^2 W +  \int |\phi_x|^2 W -  (\phi ,\phi)|_{x=L}.
\end{aligned}
\end{equation}
Plugging \eqref{Tem-2_1} into \eqref{Tem_term} gives that
\begin{equation}\label{Tem-2_2}
\begin{aligned}
\eqref{Tem_term} \ge \int u_0^2 q_y^2 W +  \int |\phi_x|^2 W = \|u_0 \nabla q \sqrt{W}\|_2^2.
\end{aligned}
\end{equation}
Finally, because of \eqref{H1_control_by_R}, we know that there exists a $c > 0$, such that
\begin{equation}\label{Tem-2_3}
c \| \nabla \phi \sqrt{W}\|_2^2 \le (1-2c) \|u_0 \nabla q \sqrt{W}\|_2^2.
\end{equation}
Combining \eqref{Tem-2_2} and \eqref{Tem-2_3}, we can conclude that
\begin{align*}
\eqref{Tem_term} \ge &c (\Delta {\phi}, \frac{u_0}{-u_{0yy} } \Delta \phi W) -  c (\nabla \phi, \nabla \phi W) + c (\phi, \phi)|_{x=L} +(1-c) ( u_0 \nabla q, u_0 \nabla q W)\\
 \ge &c (\Delta {\phi}, \frac{u_0}{-u_{0yy} } \Delta \phi W) + c (\phi, \phi)|_{x=L} +c ( u_0 \nabla q, u_0 \nabla q W).
\end{align*}
This concludes the proof of \eqref{Tem_positive}.
\end{proof}

By a similar argument, we can also prove the following estimate.
\begin{lemma}
\begin{equation}\label{H1_control_by_energy}
\| \nabla \phi \|_2 + \| u_0 \nabla q\|_2 \le \| \sqrt{\frac{u_0}{-u_{0yy}}} \Delta \phi\|_2.
\end{equation}
\end{lemma}
\begin{proof}
Similar to \eqref{Tem-2}, we have
\begin{align*}
 -  (\nabla \phi, \nabla \phi ) &=  (\Delta \phi, \phi )\\
 & \ge - \frac{1}{2} \int \frac{u_0}{-u_{0yy}} |\Delta \phi|^2 - \frac{1}{2} \int \frac{-u_{0yy}}{u_0} \phi^2 \\
 &= - \frac{1}{2} \int \frac{u_0}{-u_{0yy}} |\Delta \phi|^2 + \frac{1}{2} \int u_0^2 q_y^2 - \frac{1}{2} \int |\phi_y|^2,
\end{align*}
where we used the identity \eqref{iden_quotient} with $W = 1$. From the proof of \eqref{iden_quotient}, one can see that the identity holds for any $W$ that only depends on $x$. Then we use the same rearrangement as in the previous lemma to conclude that
$$
 - (\nabla \phi, \nabla \phi ) \ge - \int \frac{u_0}{-u_{0yy}} |\Delta \phi|^2  + \int u_0^2 |\nabla q|^2. 
$$
This concludes the proof of \eqref{H1_control_by_energy}.
\end{proof}

It remains to estimate the term $(\phi_t, \phi_x)$. In the following, we estimate $\|\phi_t\|_2$ by duality.

\subsubsection{Estimate of $\phi_t$ by duality}

Let $g$ be the auxiliary function satisfying
$$
\left\{
\begin{aligned}
-\Delta g = \phi_t, \quad \mbox{in}~(0,L)\times(0,H),\\
g|_{x=0} = g_x|_{x=L} = g|_{y=0,H} = 0.
\end{aligned}
\right.
$$
By the symmetry of $\phi_t$, we know that $g(x,y) = -g(x,H-y)$, hence $g(x , H/2) = 0$. It is straightforward to see that $g$ satisfies the estimates
\begin{equation}\label{est:g_basic}
\|g\|_2 + \| \nabla g\|_2 + \| \nabla^2 g\|_2 \lesssim_{H,L} \| \phi_t\|_2.
\end{equation}
In the following, we estimate $g_x$ with an explicit constant by separation of variables.

\begin{lemma}
\begin{equation}\label{est:g_x}
\|g_x\|_2 \le \frac{H}{4\pi} \|\phi_t\|_2.
\end{equation}
\end{lemma}

\begin{proof}
By the boundary conditions and the symmetry above, we can solve $g$ in $(0,L) \times (0, H/2)$ using separation of variables. The orthonormal basis is
$$
\varphi_{nm} = \sqrt{\frac{8}{HL}} \sin \Big[ \Big( n + \frac{1}{2} \Big) \frac{\pi}{L} x \Big] \sin \Big( \frac{2m\pi}{H} y \Big), \quad n \ge 0, m \ge 1.
$$
Define
$$
\phi_{nm}:= \int_0^{H/2} \int_0^L \phi_t \varphi_{nm} \, dx dy,
$$
and
$$
g_{nm} := \frac{\phi_{nm}}{\Big( n + \frac{1}{2} \Big)^2 \Big(\frac{\pi}{L}\Big)^2 + \Big( \frac{2m\pi}{H} \Big)^2 }.
$$
Then
$$
g = \sum_{n \ge 0} \sum_{m \ge 1} g_{nm} \varphi_{nm},
$$ 
and
$$
g_x = \sum_{n \ge 0} \sum_{m \ge 1} g_{nm} \Big( n + \frac{1}{2} \Big)  \frac{\pi}{L} \sqrt{\frac{8}{HL}} \cos \Big[ \Big( n + \frac{1}{2} \Big) \frac{\pi}{L} x \Big] \sin \Big( \frac{2m\pi}{H} y \Big).
$$
Therefore,
$$
\|g_x\|_2^2 = \sum_{n \ge 0} \sum_{m \ge 1} \frac{\Big( n + \frac{1}{2} \Big)^2 \Big(\frac{\pi}{L}\Big)^2}{\Big[ \Big( n + \frac{1}{2} \Big)^2 \Big(\frac{\pi}{L}\Big)^2 + \Big( \frac{2m\pi}{H} \Big)^2 \Big]^2} \phi_{nm}^2.
$$
Now we consider the function 
$$
f(s) = \frac{s}{(s + a)^2}, \quad a,s > 0.
$$
It is straightforward to see that $f(s)$ achieves maximum at $s = a$. Therefore,
$$
\frac{\Big( n + \frac{1}{2} \Big)^2 \Big(\frac{\pi}{L}\Big)^2}{\Big[ \Big( n + \frac{1}{2} \Big)^2 \Big(\frac{\pi}{L}\Big)^2 + \Big( \frac{2m\pi}{H} \Big)^2 \Big]^2} \le \frac{\Big( \frac{2m\pi}{H} \Big)^2}{4\Big( \frac{2m\pi}{H} \Big)^4} = \Big( \frac{H}{4m\pi}\Big)^2.
$$
Then we can estimate
$$
\|g_x\|_2^2 \le \sum_{n \ge 0} \sum_{m \ge 1} \Big( \frac{H}{4m\pi}\Big)^2 \phi_{nm}^2 \le \Big( \frac{H}{4\pi}\Big)^2 \sum_{n \ge 0} \sum_{m \ge 1} \phi_{nm}^2 = \Big( \frac{H}{4\pi}\Big)^2 \|\phi_t\|_2^2.
$$
This concludes the proof.
\end{proof}
Now we are ready to estimate $\phi_t$.
\begin{proposition}\label{prop:phi_t}
For any $\delta > 0$, we can choose an $N$ large from the weight $W = NL - x$, and there exists some positive constant $C$ depending only on $\delta, A, L$, and $H$, such that
\begin{equation}\label{est:phi_t}
\|\phi_t\|_2 \le \frac{H}{4\pi}\|R\|_2 + \frac{\delta}{\sqrt{2}} \| \sqrt{\frac{W}{-u_{0yy}}} R\|_{x= L} + C \varepsilon^{\frac{1}{3}}  \Big( \|\frac{R}{\sqrt{-u_{0yy}}}\|_2 + \sqrt{\varepsilon} \| \sqrt{\frac{u_0}{-u_{0yy}}} \nabla \Delta \phi\|_2\Big).
\end{equation}
\end{proposition}
\begin{proof}
We rewrite the linearized Navier-Stokes equation \eqref{Linearized_NS} as
$$
\Delta \phi_t = - R_x -(u_s - u_0) \Delta \phi_x + (\Delta u_s - u_{0yy}) \phi_x - v_s \Delta \phi_y + \Delta v_s \phi_y + \varepsilon \Delta^2 \phi.
$$
We test the equation by $-g$ and estimate term by term.

For the temporal term, we simply integrate by parts and obtain
\begin{equation}\label{Dual_1}
(\Delta \phi_t, -g) = (\phi_t , - \Delta g) = \| \phi_t\|_2^2.
\end{equation}

For the Rayleigh terms, first by Sobolev inequality in $x$ and \eqref{est:g_x}, we have
\begin{align*}
|(R_x, g)| &\le |-(R,g_x)| + |(R,g)_{x=L}|\\
&\le \|R\|_2 \|g_x\|_2 + \sqrt{L}\|R\|_{x=L}\|g_x\|_2\\
&\le \frac{H}{4\pi} \|R\|_2 \|\phi_t\|_2 + \frac{H\sqrt{L}}{4\pi} \|R\|_{x=L} \| \phi_t\|_2.
\end{align*}
Then for any $\delta > 0$, we can choose $N$ large such that
$$
\frac{\sqrt{2}H}{4\pi \sqrt{-(N-1)u_{0yy}}} < \delta,
$$
which implies
\begin{equation}\label{Dual_2}
|(R_x, g)| \le \frac{H}{4\pi} \|R\|_2 \|\phi_t\|_2 + \frac{\delta}{\sqrt{2}} \| \sqrt{\frac{W}{-u_{0yy}}} R\|_{x= L} \| \phi_t\|_2.
\end{equation}

For the next Rayleigh term, we integrate by parts and write
\begin{equation}\label{Dual_3}
((u_s - u_0) \Delta \phi_x,g) = - ((u_s - u_0) \Delta \phi, g_x) + ((u_s- u_0) \Delta \phi ,g)_{x=L} - (u_{sx} \Delta \phi, g).
\end{equation}
For the first two terms, we use \eqref{ss_estimates_H}, \eqref{H1_control_by_R}, Hardy inequality in $y$, Sobolev inequality in $x$, and \eqref{est:g_basic} to estimate
\begin{equation}\label{Dual_3-1}
\begin{aligned}
|(\ref{Dual_3}.1+2)| \lesssim& A \varepsilon^{\frac{1}{3}} \| u_0 \Delta \phi\|_2 \| \frac{g_x}{u_0}\|_2 + A \varepsilon^{\frac{1}{3}} \| u_0 \Delta \phi\|_{x=L} \| \frac{g}{u_0}\|_{x=L}\\
\lesssim&_{A,H,L}~ \varepsilon^{\frac{1}{3}} \| \frac{R}{\sqrt{-u_{0yy}}} \|_2 \| g_{xy}\|_2 + \varepsilon^{\frac{1}{3}} \Big( \| \frac{R}{\sqrt{-u_{0yy}}} \|_{x=L} + \| \phi\|_{x=L} \Big) \| g_y\|_{x=L}\\
\lesssim&_{A,H,L}~ \varepsilon^{\frac{1}{3}} \| \frac{R}{\sqrt{-u_{0yy}}} \|_2 \| g_{xy}\|_2 + \varepsilon^{\frac{1}{3}} \Big( \| \frac{R}{\sqrt{-u_{0yy}}} \|_{x=L} + \| \phi_x\|_{2} \Big) (\| g_y\|_{2} + \|g_{xy}\|_2)\\
\lesssim&_{A,H,L}~ \varepsilon^{\frac{1}{3}} \Big( \| \frac{R}{\sqrt{-u_{0yy}}} \|_2 +  \| \frac{R}{\sqrt{-u_{0yy}}} \|_{x=L} \Big) \| \phi_t\|_2.
\end{aligned}
\end{equation}
For the third term, we use \eqref{ss_estimates_H}, \eqref{est:mix_norm_misc}, \eqref{est:mix_norm_Hardy} with $\phi = g$, and \eqref{est:g_basic} to estimate
\begin{equation}\label{Dual_3-2}
\begin{aligned}
|(\ref{Dual_3}.3)| \lesssim& \|u_{sx}\|_2 \| \sqrt{u_0} \Delta \phi\|_{L^\infty_x L^2_y} \| \frac{g}{\sqrt{u_0}}\|_{L^2_x L^\infty_y}\\
\lesssim&_{A,H,L} \varepsilon^{1-} \| \sqrt{u_0} \Delta \phi\|_{L^\infty_x L^2_y}  \|g_y\|_2\\
\lesssim&_{A,H,L} \varepsilon^{\frac{2}{3}-} \Big( \|\frac{R}{\sqrt{-u_{0yy}}}\|_2 + \sqrt{\varepsilon} \| \sqrt{\frac{u_0}{-u_{0yy}}} \nabla \Delta \phi\|_2\Big) \|\phi_t\|_2.
\end{aligned}
\end{equation}

Next, we split $\Delta u_s$ and estimate using \eqref{ss_estimates_H}, \eqref{est:mix_norm_misc}, \eqref{est:mix_norm_H}, \eqref{est:mix_norm_Hardy} with $\phi = g$, and \eqref{est:g_basic}
\begin{equation}\label{Dual_4}
\begin{aligned}
|((\Delta u_s - u_{0yy}) \phi_x, g)| \le& |((\Delta u_s - \Delta u_a) \phi_x, g) | + |((\Delta u_a - u_{0yy}) \phi_x, g) |\\
\lesssim&_{A,H,L} \| \sqrt{u_0} (\Delta u_s - \Delta u_a)\|_2 \| \phi_x\|_{L^\infty_x L^2_y} \| \frac{g}{\sqrt{u_0}}\|_{L^2_x L^\infty_y} + \varepsilon^{\frac{1}{3}}\|\phi_x\|_2\|g\|_2\\
\lesssim&_{A,H,L}~ \varepsilon^{\frac{1}{3}}  \Big( \|\frac{R}{\sqrt{-u_{0yy}}}\|_2 + \sqrt{\varepsilon} \| \sqrt{\frac{u_0}{-u_{0yy}}} \nabla \Delta \phi\|_2\Big) \|\phi_t\|_2.
\end{aligned}
\end{equation}

Similarly, we can split $\Delta v_s$ and estimate
\begin{equation}\label{Dual_5}
\begin{aligned}
|(\Delta v_s \phi_y, g)| \le& |((\Delta v_s - \Delta v_a) \phi_y, g) | + |(\Delta v_a  \phi_x, g) |\\
\lesssim&_{A,H,L} \| \sqrt{u_0} (\Delta v_s - \Delta v_a)\|_2 \| \phi_y\|_{L^\infty_x L^2_y} \| \frac{g}{\sqrt{u_0}}\|_{L^2_x L^\infty_y} + \varepsilon^{\frac{2}{3}}\|\phi_y\|_2\|g\|_2\\
\lesssim&_{A,H,L}~ \varepsilon^{\frac{2}{3}-}  \Big( \|\frac{R}{\sqrt{-u_{0yy}}}\|_2 + \sqrt{\varepsilon} \| \sqrt{\frac{u_0}{-u_{0yy}}} \nabla \Delta \phi\|_2\Big) \|\phi_t\|_2.
\end{aligned}
\end{equation}

For the last Rayleigh term, we can simply use \eqref{ss_estimates_H}, Hardy inequality in $y$, and \eqref{est:g_basic} to estimate
\begin{equation}\label{Dual_6}
\begin{aligned}
|( v_s \Delta\phi_y, g)| \lesssim_{A,H,L} \varepsilon^{1-} \| \sqrt{u_0} \Delta \phi_y\|_2\| \frac{g}{\sqrt{u_0}}\|_2 \lesssim_{A,H,L} \varepsilon^{\frac{1}{2}-} \Big( \sqrt{\varepsilon} \| \sqrt{\frac{u_0}{-u_{0yy}}}  \Delta \phi_y\|_2\Big)\phi_t\|_2.
\end{aligned}
\end{equation}

Finally for the dissipation term, we integrate by parts and use the $\varepsilon$-Navier boundary condition \eqref{Navier_bc2} to obtain
\begin{equation*}
\begin{aligned}
-\varepsilon(\Delta^2 \phi, g) &= - \varepsilon (\Delta \phi, \Delta g) + \varepsilon (\Delta \phi, g_y)_{y = H} - \varepsilon (\Delta \phi, g_y)_{y = 0}\\
&= - \varepsilon (\Delta \phi, \Delta g) - \frac{\varepsilon^{\frac{2}{3}}}{A} ( \phi_y, g_y)_{y = H} - \frac{\varepsilon^{\frac{2}{3}}}{A} ( \phi_y, g_y)_{y = 0}.
\end{aligned}
\end{equation*}
Therefore, by Sobolev inequality in $y$, \eqref{H1_control_by_R}, and \eqref{est:H_weighted_Hardy}, we can estimate
\begin{equation}\label{Dual_7}
\begin{aligned}
\varepsilon|(\Delta^2 \phi, g)| &\lesssim_{A,H} \varepsilon \| \Delta\phi\|_2\|\Delta g\|_2 + \varepsilon^{\frac{2}{3}} (\| \phi_y\|_2 + \| \phi_{yy}\|_2) (\|g_y\|_2 + \| g_{yy}\|_2)\\
&\lesssim_{A,H} \varepsilon^{\frac{1}{3}}  \Big( \|\frac{R}{\sqrt{-u_{0yy}}}\|_2 + \sqrt{\varepsilon} \| \sqrt{\frac{u_0}{-u_{0yy}}} \nabla \Delta \phi\|_2\Big) \|\phi_t\|_2.
\end{aligned}
\end{equation}
Collecting all the terms from \eqref{Dual_1}-\eqref{Dual_7} yields the desired estimate \eqref{est:phi_t}. This concludes the proof.
\end{proof}

\subsubsection{Proof of Theorem \ref{Thm_H_linear}}

Now we are ready to prove Theorem \ref{Thm_H_linear}.

 Collecting all the estimates from \eqref{Ray_nar_1}-\eqref{Dis_nar_1-2-3}, we can see that by taking $\varepsilon$ small and $HL/A^3$ small enough (assumption \eqref{H_assumption}), all the non-favorable terms can be absorbed by the crucial positive terms \eqref{Ray_nar_1-1}, (\ref{Dis_nar_1-1}.1), and (\ref{Dis_nar_1-1}.2). Combining with the temporal terms \eqref{Tem} and \eqref{Tem-1}, we can conclude that for any $\nu>0$ small, 
\begin{equation}\label{est:H_final1}
\begin{aligned}
&\frac{1}{2}\frac{d}{dt}\Big( \|\sqrt{\frac{u_0}{-u_{0yy}}} \Delta \phi \sqrt{W}\|_2^2 - \| \nabla \phi \sqrt{W}\|_2^2 + \| \phi\|_{x=L}^2 \Big)\\
&+ \Big(\frac{1}{2}- \nu \Big)  \|\frac{R}{\sqrt{-u_{0yy}}}\|_2^2 + \frac{1}{2}\|\sqrt{\frac{W}{-u_{0yy}}}R\|_{x=L}^2 + (1-\nu) \varepsilon  \| \sqrt{\frac{u_0}{-u_{0yy}}} \nabla \Delta \phi \sqrt{W}\|_2^2 \le |(\phi_t, \phi_x)|,
\end{aligned}
\end{equation}
provided $\varepsilon$ and $HL/A^3$ are small enough. Recall ourassumption \eqref{assumption_2} on Hardy-type inequality, and use it instead of \eqref{Hardy2_narrow} in the proof of Lemma \ref{lem_H1}, we can derive
$$
\|u_0 \nabla q\| \le C_2 H \|R\|_2.
$$
Therefore, we can estimate
\begin{equation}\label{est:phi_x_final}
\|\phi_x\|_2 = \|u_0 q_x\|_2 \le C_2 H \|R\|_2.
\end{equation}
Combining with the $\phi_t$ estimate \eqref{est:phi_t}, we can estimate the right-hand side of \eqref{est:H_final1}
\begin{equation}\label{est:H_final2}
\begin{aligned}
|(\phi_t, \phi_x)| \le& \frac{C_2H^2}{4\pi} \|R\|_2^2 + O(\delta)\|R\|_2 \| \sqrt{\frac{W}{-u_{0yy}}} R\|_{x= L}\\
& + O(\varepsilon^{\frac{1}{3}})C_\delta \|R\|_2 \Big( \|\frac{R}{\sqrt{-u_{0yy}}}\|_2 + \sqrt{\varepsilon} \| \sqrt{\frac{u_0}{-u_{0yy}}} \nabla \Delta \phi\|_2\Big).
\end{aligned}
\end{equation}
Our spectral assumption \eqref{assumption_3} implies that
$$
\frac{C_2H^2}{4\pi} \|R\|_2^2 <  \Big(\frac{1}{2}- \nu \Big)  \|\frac{R}{\sqrt{-u_{0yy}}}\|_2^2,
$$
provided $\nu$ is small enough. Therefore, by fixing $\delta$ in \eqref{est:H_final2} small first, we can conclude that for all $\varepsilon$ small enough, $|(\phi_t, \phi_x)|$ can be absorbed by the positive terms on the left-hand side of \eqref{est:H_final1}. Then, by the positivity estimate \eqref{Tem_positive}, we end up with
\begin{equation}\label{est:H_final3}
\begin{aligned}
&\frac{d}{dt}\Big( \|\sqrt{\frac{u_0}{-u_{0yy}}} \Delta \phi \sqrt{W}\|_2^2 - \| \nabla \phi \sqrt{W}\|_2^2 + \| \phi\|_{x=L}^2 \Big)+  \|\frac{R}{\sqrt{-u_{0yy}}}\|_2^2 +  \varepsilon  \| \sqrt{\frac{u_0}{-u_{0yy}}} \nabla \Delta \phi \sqrt{W}\|_2^2 \le 0
\end{aligned}
\end{equation}
On the other hand, by estimate \eqref{H1_control_by_energy}, and Sobolev inequality in $x$, we can see that 
$$
 \| \nabla \phi_0 \sqrt{W}\|_2^2 + \| \phi_0\|_{x=L}^2 \lesssim  \|\sqrt{\frac{u_0}{-u_{0yy}}} \Delta \phi_0 \sqrt{W}\|_2^2.
$$
Then the desired estimate \eqref{est:H_linear} follows from the positivity estimate \eqref{Tem_positive}. Finally, the exponential decay estimate \eqref{est:H_linear_enhance} follows from \eqref{est:H_final3}, \eqref{H1_control_by_energy}, and the weighted Hardy inequality
$$
\varepsilon^{\frac{1}{3}}\| \sqrt{\frac{u_0}{-u_{0yy}}}  \Delta \phi\|_2^2 \lesssim \varepsilon \| \sqrt{\frac{u_0}{-u_{0yy}}} \Delta \phi_y\|_2^2 + \|\frac{R}{\sqrt{-u_{0yy}}}\|_2^2
$$
from \eqref{est:H_weighted_Hardy}.


\subsection{Nonlinear stability estimates}\label{sec:H_nonlinear} In this subsection, we prove Theorem \ref{Thm_H_nonlinear}. Similar to the proof of Theorem \ref{Thm_L_nonlinear}, we define the space $\cX$ to be the closure of smooth function satisfying the boundary conditions \eqref{viscous_bc} and \eqref{Navier_bc2} under the norm
$$
\| \phi \|_{\cX}^2:= \sup_{t \in (0,\infty)} \| \sqrt{\frac{u_0}{-u_{0yy}}}  \Delta \phi\|_{L^2}^2 + \frac{1}{L} \int_0^\infty  \|\frac{R}{\sqrt{-u_{0yy}}}\|_{L^2}^2 \, dt + \varepsilon \int_0^\infty  \| \sqrt{\frac{u_0}{-u_{0yy}}} \nabla\Delta \phi\|_{L^2}^2 \, dt.
$$
We define the solution map $\cS: \psi \mapsto \phi$, where $\phi$ is the solution to
\begin{equation}\label{Nonlinear_new_H}
\left\{\begin{aligned}
\Delta \phi_t + u_s \Delta \phi_x - \Delta u_s \phi_x + v_s \Delta \phi_y - \Delta v_s \phi_y - \varepsilon \Delta^2 \phi &= -\psi_y \Delta \psi_x + \psi_x \Delta \psi_y, \quad \mbox{in}~\Omega \times (0,\infty),\\
\phi|_{t=0} &= \phi_0,
\end{aligned}\right.
\end{equation}
and satisfies the boundary conditions \eqref{viscous_bc} and \eqref{Navier_bc2}. We fix $\beta > 0$, and let 
$$
\cB := \{ \phi \in \cX : \| \phi \|_{\cX} \le 2 c_0 \varepsilon^{\frac{5}{6}+ \beta} \},
$$
where $c_0$ is the small constant in Theorem \ref{Thm_H_nonlinear}. Our goal is to show that $\cS$ is a contraction mapping in $\cB$ when $\varepsilon \ll 1$.

We test \eqref{Nonlinear_new_H} by $\frac{RW}{-u_{0yy}}$, where $W$ is given in \eqref{weight}, and estimate each term. The linear estimates have been done in the previous section, so it remains to estimate the nonlinear terms
\begin{equation}\label{Non_H}
-(\psi_y \Delta \psi_x, \frac{RW}{-u_{0yy}}) + (\psi_x \Delta \psi_y, \frac{RW}{-u_{0yy}}).
\end{equation}
For the first term, we can estimate using \eqref{est:nabla_phi_L_infty}, Hardy inequality in $y$, and \eqref{H1_control_by_energy},
\begin{equation}\label{Non_H_1}
\begin{aligned}
|(\ref{Non_H}.1)| &\le |(\psi_y \Delta \psi_x, \frac{u_0 \Delta \phi}{-u_{0yy}}W)| + |(\psi_y \Delta \psi_x, \phi W)|\\
&\lesssim_{H,L} \| \psi \|_{L^\infty} \| \sqrt{\frac{u_0}{-u_{0yy}}} \Delta \psi_x\|_2 \| \sqrt{\frac{u_0}{-u_{0yy}}} \Delta \phi\|_2  + \| \psi \|_{L^\infty} \| \sqrt{\frac{u_0}{-u_{0yy}}} \Delta \psi_x\|_2 \| \phi_y\|_2\\
&\lesssim_{H,L} \varepsilon^{-(\frac{5}{6}+)} \Big( \|\frac{R_\psi}{\sqrt{-u_{0yy}}}\|_{2}^2  + \varepsilon \| \sqrt{\frac{u_0}{-u_{0yy}}} \nabla\Delta \psi\|_{2}^2 \Big) \| \sqrt{\frac{u_0}{-u_{0yy}}} \Delta \phi\|_2,
\end{aligned}
\end{equation}
where $R_\psi := u_0 \Delta \psi - u_{0yy} \psi$. Similarly, we can also estimate
\begin{equation}\label{Non_H_2}
\begin{aligned}
|(\ref{Non_H}.2)|\lesssim_{H,L} \varepsilon^{-(\frac{5}{6}+)} \Big( \|\frac{R_\psi}{\sqrt{-u_{0yy}}}\|_{2}^2  + \varepsilon \| \sqrt{\frac{u_0}{-u_{0yy}}} \nabla\Delta \psi\|_{2}^2 \Big) \| \sqrt{\frac{u_0}{-u_{0yy}}} \Delta \phi\|_2.
\end{aligned}
\end{equation}
Combining with the linear estimates, we have
\begin{align*}
&\sup_{t \in (0,\infty)} \|\sqrt{\frac{u_0}{-u_{0yy}}} \Delta \phi\|_{2}^2 + \frac{1}{L} \int_0^\infty \|\frac{R}{\sqrt{-u_{0yy}}}\|_{2}^2 \, dt + \varepsilon \int_0^\infty  \| \sqrt{\frac{u_0}{-u_{0yy}}} \nabla \Delta \phi \|_{2}^2 \, dt\\
\lesssim&_{H,L}~ \varepsilon^{-(\frac{5}{6}+)} \sup_{t \in (0,\infty)} \| \sqrt{\frac{u_0}{-u_{0yy}}} \Delta \phi\|_2 \int_0^\infty \Big( \|\frac{R_\psi}{\sqrt{-u_{0yy}}}\|_{2}^2  + \varepsilon \| \sqrt{\frac{u_0}{-u_{0yy}}} \nabla\Delta \psi\|_{2}^2 \Big) \, dt + \| \sqrt{\frac{u_0}{-u_{0yy}}} \Delta \phi_0\|_2^2.
\end{align*}
Since $\psi \in \cB$, and by the initial data \eqref{H_initial}, we can conclude that $\phi \in \cB$ when $c_0$ is small enough. Hence the solution map $\cS$ maps $\cB$ into $\cB$. 

To show $\cS$ is a contraction mapping, we take $\psi_1, \psi_2 \in \cB$, and denote $\phi_1 = \cS(\psi_1), \phi_2 = \cS(\psi_2)$. We further denote $\widetilde{\phi} = \phi_1 - \phi_2$ and $\widetilde{\psi} = \psi_1 - \psi_2$. Repeating the linear and nonlinear estimates in this section for the equation for $\widetilde{\phi}$, we will end up with
$$
\| \widetilde \phi\|_{\cX} \lesssim c_0 \| \widetilde \psi\|_{\cX},
$$
which implies that $\cS$ is a contraction mapping when $c_0$ is small. By the contraction mapping theorem, there exists a unique solution $\phi$ satisfying the estimate \eqref{est:H_linear}. Similar as the linear case, the exponential decay estimate \eqref{est:H_linear_enhance} also follow. This concludes the proof of Theorem \ref{Thm_H_nonlinear}.


\section{Construction of the steady state solution}\label{sec:construction}


\subsection{The setup}

In this section, we present a unified construction of the steady-state solution that applies in both settings. The construction is uniform in $H, L$ and $A$. Therefore, the constants in estimates are allowed to depend on $H, L$ and $A$ without further specification.

Given the base flow $u_0$ satisfying either the assumptions in Theorem \ref{Thm_Existence_L} or those in Theorem \ref{Thm_H_Existence}, we define
$$
u_e^0(y) := u_0(y) + A u_0'(0) \varepsilon^{1/3}.
$$
Then $\{ u_e^0,0 \}$ satisfies the $\varepsilon$-Navier boundary condition \eqref{Navier_boundary_condition_new}. We will first construct an approximated steady solution, and study its stability. The approximated solution has the following ansatz:
\begin{equation}\label{approximated}
\begin{aligned}
u_a :=& u_e^0 + \varepsilon u_e^1 + \varepsilon u_p^1,\\
v_a :=& \varepsilon v_e^1 + \varepsilon^{4/3} v_p^1,\\
p_a :=& \varepsilon p_e^1 + \varepsilon^{4/3} p_p^1,
\end{aligned}
\end{equation}
where
\begin{align*}
u_p^1(x,y) &= u_p^-(x, \frac{y}{\varepsilon^{1/3}}) + u_p^+ (x, \frac{H-y}{\varepsilon^{1/3}}),\\
v_p^1(x,y) &= v_p^-(x, \frac{y}{\varepsilon^{1/3}}) + v_p^+ (x, \frac{H-y}{\varepsilon^{1/3}}),\\
p_p^1(x,y) &= p_p^-(x, \frac{y}{\varepsilon^{1/3}}) + p_p^+ (x, \frac{H-y}{\varepsilon^{1/3}}),
\end{align*}
are sum of Prandtl boundary layer correctors concentrating near the boundary $y = 0$ and $y = H$. We only focus on constructing $u_p^-, v_p^-$ near $y = 0$, as $u_p^+$ and $v_p^+$ are just the even and odd reflection of $u_p^-$ and $v_p^-$ with respect to $y = H/2$.

We denote the Prandtl variable
$$
Y := \frac{y}{\varepsilon^{1/3}}.
$$
Plugging the ansatz \eqref{approximated} into the left-hand side of \eqref{NS}, near the boundary $y = 0$, we have from the $u$ equation:
\begin{equation}\label{NS_u}
\begin{aligned}
NS[u_a]:=& (u_e^0 + \varepsilon u_e^1 + \varepsilon u_p^-)(\varepsilon u_{ex}^1 + \varepsilon u_{px}^-) + (\varepsilon v_e^1 + \varepsilon^{4/3} v_p^-)(u_{ey}^0 + \varepsilon u_{ey}^1 + \varepsilon^{2/3} u_{pY}^-)\\
&+ \varepsilon p_{ex}^1 + \varepsilon^{4/3} p_{px}^- - \varepsilon u_{eyy}^0 - \varepsilon^2 \Delta u_e^1 - \varepsilon^2 u_{pxx}^- - \varepsilon^{4/3} u_{pYY}^-;
\end{aligned}
\end{equation}
and from the $v$ equation:
\begin{equation}\label{NS_v}
\begin{aligned}
NS[v_a]:=& (u_e^0 + \varepsilon u_e^1 + \varepsilon u_p^-) (\varepsilon v_{ex}^1 + \varepsilon^{4/3} v_{px}^-) + (\varepsilon v_e^1 + \varepsilon^{4/3} v_p^-) (\varepsilon v_{ey}^1 + \varepsilon v_{pY}^-)\\
&+ \varepsilon p_{ey}^1 + \varepsilon p_{pY}^- - \varepsilon^2 \Delta v_e^1 - \varepsilon^{7/3} v_{pxx}^- - \varepsilon^{5/3} v_{pYY}^-.
\end{aligned}
\end{equation}
Collecting the terms of order $O(\varepsilon)$ that are away from the boundary from \eqref{NS_u} and \eqref{NS_v}, we know that $\{ u_e^1, v_e^1 \}$ satisfies the linearized Euler equation:
\begin{equation}\label{first_Euler_layer}
\left\{\begin{aligned}
u_e^0 u_{ex}^1 + u_{ey}^0 v_e^1 + p_{ex}^1 &= u_{eyy}^0,\\
u_e^0 v_{ex}^1 + p_{ey}^1 &= 0,\\
u_{ex}^1 + v_{ey}^1 &= 0.
\end{aligned}\right.
\end{equation}
In the next subsection, we will solve $v_e^1$ with the boundary condition $v_e^1 = 0$ on $y = 0, H$. Then we define 
\begin{equation}\label{first_Euler_u}
u_e^1(x,y) = - \int_0^x v_{ey}^1(x',y) \, dx'.
\end{equation}
This creates a mismatch for the boundary condition \eqref{Navier_boundary_condition_new}, which will be corrected by the Prandtl profile $\{u_p^-, v_p^-\}$.

Next we collect the leading order Prandtl terms. They are $O(\varepsilon^{4/3})$ order terms in \eqref{NS_u} and $O(\varepsilon)$ order terms in \eqref{NS_v}. In view of the boundary condition mismatch above, we know that $\{u_p^-, v_p^-\}$ satisfies the linearized Prandtl equation with a Robin boundary condition:
\begin{equation}\label{first_Prandtl_layer}
\left\{\begin{aligned}
\varepsilon^{-1/3} u_e^0 u_{px}^- + u_{ey}^0 v_p^- + p_{px}^- - u_{pYY}^- &= 0,\\
 p_{pY}^- &= 0,\\
u_{px}^- + v_{pY}^- &= 0,\\
\lim_{Y \to \infty} u_p^- = \lim_{Y \to \infty} v_p^- &= 0,\\
u_p^-(x,0) - A u_{pY}^-(x,0) &= u_e^1(x,0) - A \varepsilon^{1/3} u_{ey}^1(x,0).
\end{aligned}\right.
\end{equation}
Note that there is an $\varepsilon^{-1/3}$ factor appears in the first term. This is because near the boundary $y = 0$, $u_e^0 \sim u_{ey}^0(0) y + u_{ey}^0(0) A \varepsilon^{1/3}$. Therefore, $\varepsilon^{-1/3} u_e^0 \sim u_{ey}^0(0) (A+Y)$ is of order 1.

Once the approximated solution $\{u_a, v_a\}$ is constructed, we set
\begin{equation}\label{ss_construction}
\left\{
\begin{aligned}
u_s &= u_a + \varepsilon^{\frac{4}{3}-} u,\\
v_s &= v_a + \varepsilon^{\frac{4}{3}-} v,
\end{aligned}
\right.
\end{equation}
and solve for the remainder $\{u,v\}$. We denote by $\phi$ the stream function of the remainder $\{u,v\}$, then $\phi$ satisfies the equation
\begin{equation}\label{eqn:remainder}
u_a \Delta \phi_x - \Delta u_a \phi_x + v_a \Delta \phi_y - \Delta v_a \phi_y - \varepsilon \Delta^2 \phi = - \varepsilon^{-(\frac{4}{3}-)}(\partial_y NS[u_a] - \partial_x NS[v_a])  - \varepsilon^{\frac{4}{3}-} (\phi_y \Delta \phi_x - \phi_x \Delta \phi_y),
\end{equation}
with the viscous boundary condition \eqref{viscous_bc} and the $\varepsilon$-Navier-slip boundary condition \eqref{Navier_bc2}.

In the following subsections we construct solutions to \eqref{first_Euler_layer}, \eqref{first_Prandtl_layer}, and \eqref{eqn:remainder}. Note that a similar ansatz \eqref{approximated} was considered in \cite{IyerZhou} by Iyer and Zhou; they assumed that $\partial_y^j u_0(0) = \partial_y^j u_0(H) = 0$ for $2 \le j \le N_0$ with $N_0$ large. However, this condition is not satisfied by the Poiseuille flow. In our work, we only assume that $\| u_0'''/u_0 \|_{L^\infty} < \infty$, which leads to singularities in high-order $y$-derivative of $\{u_e^1, v_e^1\}$. We show that this weaker assumption nevertheless suffices for the stability analysis.

\subsection{Construction of the Euler profile $\{u_e^1, v_e^1\}$} Consider \eqref{first_Euler_layer} in vorticity formulation, i.e., $\partial_y$\eqref{first_Euler_layer}$_1 - \partial_x$\eqref{first_Euler_layer}$_2$, we have
\begin{equation}\label{first_Euler_layer_2}
\left\{\begin{aligned}
-u_e^0 \Delta v_e^1 + u_{eyy}^0 v_e^1 &= u_{eyyy}^0,\\
v_e^1|_{y=0,H} &= 0.
\end{aligned}\right.
\end{equation}
We will find a solution $v_e^1 = v_e^1(y)$ to \eqref{first_Euler_layer_2} that is independent of $x$, and consequently $u_e^1(x,y)$ is linear in $x$.

\begin{proposition}
There exists a solution $v_e^1(y)$ to the Euler equation \eqref{first_Euler_layer_2} satisfying
\begin{equation}\label{Euler_estimate1}
\begin{aligned}
\|\partial_y^j v_e^1\|_{L^\infty} &\lesssim 1, &&\mbox{for}~j = 0,1,2,\\
\|\partial_y^j v_e^1\|_{L^\infty} &\lesssim \varepsilon^{-(j-2)/3}, && \mbox{for}~j \ge 3.
\end{aligned}
\end{equation}
Consequently, $u_e^1(x,y):= -\int_0^x v_{ey}^1$ satisfies $u_{exx}^1 = 0$ and
\begin{equation}\label{Euler_estimate2}
\begin{aligned}
\|\partial_x^i \partial_y^j u_e^1\|_{L^\infty} &\lesssim 1, &&\mbox{for}~i,j = 0,1,\\
\|\partial_x^i \partial_y^j u_e^1\|_{L^\infty} &\lesssim_j \varepsilon^{-(j-1)/3}, && \mbox{for}~i =0,1, j \ge 2.
\end{aligned}
\end{equation}
\end{proposition}

\begin{proof}
Since we look for a solution $v_e^1$ that only depends on $y$, \eqref{first_Euler_layer_2} is reduced to an ODE. Since $u_e^0$ is even with respect to $y = H/2$, we know that $v_e^1$ is odd with respect to $y = H/2$, and consequently $v_e^1(H/2) = 0$. We define
$$
q_e := \frac{v_e^1}{u_e^0},
$$
then the equation \eqref{first_Euler_layer_2} can be written in divergence form:
\begin{equation}\label{first_Euler_layer_3}
- \Big( (u_e^0)^2 q_{ey} \Big)_y = u_{eyyy}^0.
\end{equation}
Multiplying the equation \eqref{first_Euler_layer_3} by $q_e$, integrating by parts, and using Lemma \ref{lem_hardy2_narrow}, we have
$$
\| u_e^0 q_{ey}\|^2 \le \|u_{eyyy}^0\| \|q_e\| \lesssim \|u_{eyyy}^0\| \| u_e^0 q_{ey}\|,
$$ 
which implies
\begin{equation}\label{est:qe}
\|q_e\| \le \| u_e^0 q_{ey}\| \lesssim 1.
\end{equation}
Therefore, the existence of $q_e$ follows from the Lax-Milgram theorem. We also have
$$
\| v_e^1 \|_{L^\infty} \lesssim \|v_{ey}^1\| = \| u_{ey}^0 q\| + \|u_e^0 q_y\| \lesssim \|u_e^0 q_y\| \lesssim 1.
$$
Since
$$
v_{eyy}^1 = u_{eyy}^0 \frac{v_e^1}{u_e^0} - \frac{u_{eyyy}^0}{u_e^0},
$$
by \eqref{est:qe} and the assumption that $\|u_{eyyy}^0/u_e^0\|_{L^\infty} < \infty$, we have
\begin{equation}\label{est:qey}
\| v_{ey}^1 \|_{L^\infty} \lesssim \|v_{eyy}^1\| + \|v_{ey}\| \lesssim 1.
\end{equation}
Next, it is easy to see that $|v_e^1/u_e^0| \lesssim 1$. Indeed, when $0< y < H/10$, we have
$$
\left| \frac{v_e^1}{u_e^0} \right| \lesssim \| v_{ey}^1 \|_{L^\infty} \frac{y}{\varepsilon^{1/3}+ y} \lesssim 1.
$$
Same argument applies when $9H/10< y < H$. When $H/10 \le y \le 9H/10$,
$$
\left| \frac{v_e^1}{u_e^0} \right| \lesssim \| v_{e}^1 \|_{L^\infty} \lesssim 1.
$$
Therefore,
\begin{equation}\label{est:qeyy}
\| v_{eyy}^1 \|_{L^\infty} \lesssim \left\| \frac{v_e^1}{u_e^0} \right\|_{L^\infty} + \left\| \frac{u_{eyyy}^0}{u_e^0} \right\|_{L^\infty} \lesssim 1.
\end{equation}
To estimate $v_{eyyy}^1$, we take $y$-derivative on the equation, we have
$$
-u_e^0  v_{eyyy}^1 -u_{ey}^0  v_{eyy}^1 + u_{eyy}^0 v_{ey}^1 + u_{eyyy}^0 v_e^1 = u_{eyyyy}^0.
$$
By \eqref{est:qe}, \eqref{est:qey}, and \eqref{est:qeyy}, we have
$$
|v_{eyyy}^1| \lesssim \frac{1}{u_e^0} \lesssim \varepsilon^{-1/3}.
$$
For estimating higher derivatives, we keep taking $y$-derivative and bootstrap. This concludes the proof.
\end{proof}

\subsection{Construction of the Prandtl profile $\{u_p^1, v_p^1\}$} The goal of this subsection is to construct $\{u_p^-, v_p^-\}$ so that
\begin{equation}\label{goal_prandtl}
\varepsilon^{-1/3} u_e^0 u_{px}^- + u_{ey}^0 v_p^- + p_{px}^- - u_{pYY}^- = O(\varepsilon^{\frac{1}{3}}).
\end{equation}
Instead of working on \eqref{first_Prandtl_layer}, we consider the approximated system:
\begin{equation}\label{eqn:approximated_Prandtl}
\left\{\begin{aligned}
 u_{ey}^0(0) (A+Y) u_{px}^- + u_{ey}^0(0) v_p^- + p_{px}^- - u_{pYY}^- &= 0,\\
 p_{pY}^- &= 0,\\
u_{px}^- + v_{pY}^- &= 0,\\
\lim_{Y \to \infty} u_p^- = \lim_{Y \to \infty} v_p^- &= 0,\\
u_p^-(x,0) - A u_{pY}^-(x,0) &= u_e^1(x,0) - A \varepsilon^{1/3} u_{ey}^1(x,0),
\end{aligned}\right.
\end{equation}
with a boundary condition $u_p^-(0,Y) = U_0(Y)$. We assume standard parabolic compatibility conditions on $U_0$ at the corner $(0,0)$, and sufficiently fast decay at infinity:
\begin{equation}\label{first_Prandtl_initialdata}
|e^Y d_Y^k U_0| \lesssim 1 \quad \mbox{for}~0 \le k \le N_0
\end{equation}
with some large $N_0$.  We denote
$$
g(x):= u_e^1(x,0) - A \varepsilon^{1/3} u_{ey}^1(x,0).
$$
From \eqref{first_Euler_u} and \eqref{Euler_estimate2}, we know that $g(0) = 0$, $|g(x)| + |g'(x)| \lesssim 1$, and $g''(x) = 0$. By taking $Y \to \infty$, one can see that $p_p^- \equiv 0$. Next, we homogenize the Robin boundary condition. We take $\chi(Y) \in C^\infty[0,\infty)$ such that
$$
\left\{
\begin{aligned}
& \chi(Y) = 1, \quad \mbox{when}~0<Y\le 1,\\
& \chi(Y) = 0, \quad \mbox{when}~Y \ge 2,\\
& \int_0^\infty \chi \, dY = 0.
\end{aligned}
\right.
$$
Let
\begin{align*}
u(x,Y)&:= u_p^- (x,Y) - \chi(Y)g(x),\\
v(x,Y)&:= \int_Y^\infty u_{px}^-(x,Y') \, dY' + g'(x) \int_0^Y \chi(Y') \, dY'.
\end{align*}
Then $\{u,v\}$ satisfies
\begin{equation}\label{first_Prandtl_layer_2}
\left\{\begin{aligned}
 u_{ey}^0(0) (A+Y) u_x + u_{ey}^0(0) v - u_{YY} &= F,\\
 v =  \int_Y^\infty u_x(x,Y')&\, dY',\\
\lim_{Y \to \infty} u  &= 0,\\
u(x,0) - A u_{Y}(x,0) &= 0,\\
u(0,Y) &= U_0(Y),
\end{aligned}\right.
\end{equation}
where
$$
F(x,Y):= -u_{ey}^0(0) (A+Y) \chi(Y) g'(x) + u_{ey}^0(0) g'(x) \int_0^Y \chi(Y')\, dY' + g(x) \chi''(Y).
$$
\begin{proposition}
Under the assumptions above, there exists a unique solution $\{u,v\}$ to \eqref{first_Prandtl_layer_2} that satisfies
\begin{equation}\label{est:Prandtl}
|(1+Y)^m \partial_x^i \partial_Y^j \{u,v\}| \lesssim_{m,i,j} 1 \quad \mbox{for large}~m,i,j.
\end{equation}
\end{proposition}

\begin{proof}
We will only prove the apriori estimates. The existence of solution for such linear Prandtl system is classical, see, e.g., \cite{GuoIyer21}.

First, we multiply the equation \eqref{first_Prandtl_layer_2} by $u_x$ and integrate in $Y$, we have
\begin{equation}\label{energy:basic_Prandtl}
\int_0^\infty  u_{ey}^0(0) (A+Y) u_x^2 \, dY - \int_0^\infty  u_{ey}^0(0) v v_Y \, dY - \int_0^\infty  u_{YY} u_x \, dY = \int_0^\infty  F u_x \,dY.
\end{equation}
Integrating by parts, we have
$$
(\ref{energy:basic_Prandtl},2) =  - \frac{1}{2} \int_0^\infty  u_{ey}^0(0) (v)^2_Y \, dY = \frac{1}{2} u_{ey}^0(0)v^2(x,0),
$$
which is positive. Using the Robin boundary condition, we also have
\begin{align*}
(\ref{energy:basic_Prandtl},3) = \int_0^\infty  u_Y u_{Yx} \, dY + u_Y(x,0) u_x(x,0)= \frac{1}{2} \partial_x \int_0^\infty  u_Y^2 \, dY + \frac{1}{2A} \partial_x |u(x,0)|^2.
\end{align*}
Therefore, we end up with a basic energy estimate
\begin{equation}\label{est:basic_Prandtl}
\begin{aligned}
&\int_0^L \int_0^\infty  u_{ey}^0(0) (A+Y) u_x^2 \, dY dx + \int_0^L |v(x,0)|^2\, dx + \sup_{x\in(0,L)} \int_0^\infty  u_Y^2 \, dY + \sup_{x\in(0,L)} |u(x,0)|^2\\
& \lesssim \|F\|_2^2 + \|U_0\|_{H^1}^2.
\end{aligned}
\end{equation}

Here we recall that, due to $g(x)$ is linear in $x$ and the estimate \eqref{Euler_estimate2}, we have
$$
\|F\|_{H^m} \lesssim_m 1 \quad \mbox{for all}~m \geq 0.
$$

Next, applying $\partial_Y$ to the equation \eqref{first_Prandtl_layer_2}, we obtain
\begin{equation}\label{eqn:Prandtl_Y}
u_{ey}^0(0) (A+Y) u_{xY} - u_{YYY} = F_Y.
\end{equation}
Since $u_x \to 0$ as $Y \to \infty$, we have that
\begin{equation}\label{Prandtl:ux_boundary}
u_x(0,Y) = - \int_Y^\infty u_{xY}(0,Y') \, dY' = - \int_Y^\infty \frac{1}{u_{ey}^0(0) (A+Y)} \Big( F_Y(0,Y') + U_{0YYY}(Y') \Big) \, dY',
\end{equation}
which will serve as the boundary condition for $u_x$. Then, applying $\partial_x$ to the equation \eqref{first_Prandtl_layer_2}, we have
\begin{equation}\label{eqn:Prandtl_x}
u_{ey}^0(0) (A+Y) u_{xx} + u_{ey}^0(0) v_x - u_{YYx} = F_x.
\end{equation}
We test the equation \eqref{eqn:Prandtl_x} by $u_{xx}$ and follow the exact same strategy as in deriving \eqref{est:basic_Prandtl} with the boundary condition \eqref{Prandtl:ux_boundary}, we end up with
\begin{equation}\label{est:basic_x_Prandtl}
\begin{aligned}
&\int_0^L \int_0^\infty  u_{ey}^0(0) (A+Y) u_{xx}^2 \, dY dx + \int_0^L |v_x(x,0)|^2\, dx + \sup_{x\in(0,L)} \int_0^\infty  u_{xY}^2 \, dY + \sup_{x\in(0,L)} |u_x(x,0)|^2\\
& \lesssim \|\nabla F\|_2^2 + \|U_0\|_{H^3}^2.
\end{aligned}
\end{equation}
Next, we multiply the equation \eqref{eqn:Prandtl_Y} by $u_{xY}$ and integrate in $Y$, we have
\begin{equation}\label{energy:basic_Y_Prandtl}
\int_0^\infty  u_{ey}^0(0) (A+Y) u_{xY}^2 \, dY - \int_0^\infty  u_{YYY} u_{xY} \, dY = \int_0^\infty  F_Y u_{xY} \,dY.
\end{equation}
The first term is positive, and the right-hand side can be estimated using H\"older's inequality. To estimate (\ref{energy:basic_Y_Prandtl}.2), we integrate by parts and use the equation \eqref{first_Prandtl_layer_2} and the Robin boundary condition to obtain
\begin{align*}
(\ref{energy:basic_Y_Prandtl}.2) &= - \int_0^\infty  u_{YYY} u_{xY} \, dY = \int_0^\infty u_{YY} u_{xYY} \, dY + u_{YY}(x,0) u_{xY}(x,0)\\
&= \frac{1}{2} \partial_x \int_0^\infty  u_{YY}^2 \, dY + \Big( A u_{ey}^0(0) u_x(x,0) + u_{ey}^0(0) v(x,0) - F(x,0)   \Big) u_{xY}(x,0)\\
&= \frac{1}{2} \partial_x \int_0^\infty  u_{YY}^2 \, dY + u_{ey}^0(0) |u_x(x,0)|^2 + u_{ey}^0(0) v(x,0) u_x(x,0)  - F(x,0)  u_x(x,0),
\end{align*}
where the last two terms can be controlled by the positive terms in estimates \eqref{est:basic_Prandtl} and \eqref{est:basic_x_Prandtl}.

To obtain the higher-order derivative estimates, we inductively differentiate the equation in $Y$ and $x$, and repeat the procedure above. To deal with the boundary terms at $Y=0$, we can always use the equation to reduce the terms (as we did above) so that there is at most one $Y$ derivative of $u$. Then we can use the Robin boundary condition to estimate them.

To obtain the decay estimates, we multiply $(1+Y)^m$ to the equations and repeat the same procedure above. This concludes the proof.
\end{proof}

Now we have the solution $\{u,v\}$ to the approximated Prandtl system \eqref{eqn:approximated_Prandtl}, our goal is to construct $\{u_p^-,v_p^-\}$ so that \eqref{goal_prandtl} is satisfied. We take another smooth positive cutoff function $\tilde{\chi}(Y) \in C_0^\infty[0,\infty)$ such that
$$
\left\{
\begin{aligned}
& \tilde\chi(Y) = 1, \quad \mbox{when}~0<Y\le 1,\\
& \tilde\chi(Y) = 0, \quad \mbox{when}~Y \ge 2.\\
\end{aligned}
\right.
$$
For $N>0$ large, we define
$$
\left\{
\begin{aligned}
u_p^- &:= \tilde{\chi}\Big(\frac{Y}{N}\Big) u - \frac{1}{N} \tilde{\chi}' \Big(\frac{Y}{N}\Big) \int_0^x v(x',Y) \, dx',\\
v_p^- &:= \tilde{\chi}\Big(\frac{Y}{N}\Big) v.
\end{aligned}
\right.
$$
Now we estimate the left-hand side of \eqref{goal_prandtl}. We write
\begin{align*}
\varepsilon^{-1/3} u_e^0 u_{px}^- + u_{ey}^0 v_p^- - u_{pYY}^- &= [\varepsilon^{-1/3} u_e^0 - u_{ey}^0(0)(A+Y)]u_{px}^- + [ u_{ey}^0 - u_{ey}^0(0)] v_p^-\\
&+ u_{ey}^0(0)(A+Y) u_{px}^- + u_{ey}^0(0) v_p^- - u_{pYY}^-.
\end{align*}
By Taylor expansion near $y = 0$, we know that
\begin{align*}
\varepsilon^{-1/3} u_e^0 &= \varepsilon^{-1/3} \Big( u_{ey}^0(0) A \varepsilon^{1/3} +  u_{ey}^0(0) y + O(y^2) \Big)\\
&= u_{ey}^0(0) (A+Y) + O(\varepsilon^{1/3})Y^2.
\end{align*}
Therefore, by \eqref{est:Prandtl}, we have
$$
[\varepsilon^{-1/3} u_e^0 - u_{ey}^0(0)(A+Y)]u_{px}^- = O(\varepsilon^{1/3})Y^2 u_{px}^- = O(\varepsilon^{1/3}).
$$
Similarly, we can estimate
$$
[ u_{ey}^0 - u_{ey}^0(0)] v_p^- = O(\varepsilon^{1/3})Y v_p^- = O(\varepsilon^{1/3}).
$$
Finally, by direct computations from the definition of $\{u_p^-,v_p^-\}$, we have
\begin{align*}
&u_{ey}^0(0)(A+Y) u_{px}^- + u_{ey}^0(0) v_p^- - u_{pYY}^-\\
=& - \frac{u_{ey}^0(0)(A+Y)}{N} \tilde{\chi}' v - \frac{3}{N} \tilde{\chi}' u_Y + \frac{3}{N^2} \tilde{\chi}''u - \frac{1}{N^3} \tilde{\chi}''' \int_0^x v \, dx = O(\varepsilon^{1/3})
\end{align*}
by \eqref{est:basic_Prandtl} and choosing $N = \frac{\varepsilon^{-1/3}}{100}$. Therefore, \eqref{goal_prandtl} is satisfied.\\

\subsection{Construction of the remainder}\label{sec:remainder}
In this subsection, we construct the remainder $\{u,v\}$, whose stream function satisfies the equation \eqref{eqn:remainder}. We denote
$$
E := \partial_y NS[u_a] - \partial_x NS[v_a].
$$
By the construction of $\{u_a, v_a\}$ above, one can see that $|E| \lesssim \varepsilon^{\frac{4}{3}}$. We will work with the equation
\begin{equation}\label{eqn:remainder_new}
u_a \Delta \phi_x - \Delta u_a \phi_x + v_a \Delta \phi_y - \Delta v_a \phi_y - \varepsilon \Delta^2 \phi = - \varepsilon^{-(\frac{4}{3}-)}E  - \varepsilon^{\frac{4}{3}-} (\psi_y \Delta \psi_x - \psi_x \Delta \psi_y)
\end{equation}
with the boundary conditions \eqref{viscous_bc} and \eqref{Navier_bc2}, and use a fixed point argument.

\subsubsection{Small $L$ case}
First we construct the solution of \eqref{eqn:remainder} under the assumption of Theorem \ref{Thm_Existence_L}. We define the space $\cX_1$ to be the closure of smooth function satisfying the boundary conditions \eqref{viscous_bc} and \eqref{Navier_bc2} under the norm
$$
\| \phi \|_{\cX_1}^2:=  \frac{1}{L} \| u_0 \Delta \phi \|_{L^2}^2 + \varepsilon  \| \sqrt{u_0} \nabla \Delta \phi \|_{L^2}^2.
$$
We define the solution map $\cS: \psi \mapsto \phi$, where $\phi$ is the solution to \eqref{eqn:remainder_new} with the boundary conditions \eqref{viscous_bc} and \eqref{Navier_bc2}. Let 
$$
\cB_1 := \{ \phi \in \cX_1 : \| \phi \|_{\cX_1} \le 1 \}.
$$
We will prove that $\cS$ is a contraction mapping on $\cB_1$. We test the equation \eqref{eqn:remainder_new} by $u_0 \Delta \phi W$, with $W = 2 - \frac{x}{L}$. The estimates of the left-hand side of \eqref{eqn:remainder_new} follows exactly the same as in Section \ref{sec:L_linear}, as the regularity of $\{u_a,v_a\}$ is better than $\{u_s, v_s\}$. The linear part gives positive terms including
\begin{equation}\label{pos_1}
\frac{1}{L} \| u_0 \Delta \phi \|_{L^2}^2 + \varepsilon  \| \sqrt{u_0} \nabla \Delta \phi \|_{L^2}^2.
\end{equation}
It remains to estimate the right-hand side of \eqref{eqn:remainder_new}. First,
\begin{equation}\label{cons_1}
|\varepsilon^{\frac{4}{3}-}(E, u_0 \Delta \phi W)| \lesssim \varepsilon^{0+} \| u_0 \Delta \phi \|_{2}^2.
\end{equation}
For the nonlinear terms, by \eqref{est:mix_norm}, we can estimate
\begin{equation}\label{cons_2}
\begin{aligned}
&\varepsilon^{\frac{4}{3}-} |\Big((\psi_y \Delta \psi_x - \psi_x \Delta \psi_y), u_0 \Delta \phi W \Big)|\\
\lesssim & \varepsilon^{\frac{4}{3}-}\Big( \| \psi_y\|_{L^2_x L^\infty_y} \| \sqrt{u_0}\Delta \psi_x\|_2 \| \sqrt{u_0}\Delta \phi\|_{L^\infty_x L^2_y} + \| \psi_x\|_{L^2_x L^\infty_y} \| \sqrt{u_0}\Delta \psi_y\|_2 \| \sqrt{u_0}\Delta \phi\|_{L^\infty_x L^2_y}\Big)\\
\lesssim &\varepsilon^{\frac{1}{3}-} \| \psi \|_{\cX_1}^2 \| \phi\|_{\cX_1}.
\end{aligned}
\end{equation}
Therefore, for $\psi \in \cB_1$, one can see that \eqref{cons_1} and \eqref{cons_2} can be absorbed by the positive contribution \eqref{pos_1} provided $\varepsilon$ is small enough. This implies that $\phi = \cS(\psi) \in \cB_1$. Similarly, one can show that $\cS : \cB_1 \to \cB_1$ is a contraction mapping. Therefore, the existence and uniqueness of $\phi$ as the solution of \eqref{eqn:remainder} follow.

\subsubsection{General $L$ case}
Now we construct the solution of \eqref{eqn:remainder} under the assumption of Theorem \ref{Thm_H_Existence}. We instead define the space $\cX_2$ to be the closure of smooth function satisfying the boundary conditions \eqref{viscous_bc} and \eqref{Navier_bc2} under the norm
$$
\| \phi \|_{\cX_2}^2:=    \|\frac{R}{\sqrt{-u_{0yy}}}\|_{2}^2 + \varepsilon   \| \sqrt{\frac{u_0}{-u_{0yy}}} \nabla \Delta \phi \|_{2}^2.
$$
Let 
$$
\cB_2 := \{ \phi \in \cX_2 : \| \phi \|_{\cX_2} \le 1 \}.
$$
We will prove that the solution map $\cS$ is a contraction mapping on $\cB_1$. To this end, we test the equation \eqref{eqn:remainder_new} by $\frac{R}{-u_{0yy}} W$, with $W = 2L - x$. The estimates of the left-hand side of \eqref{eqn:remainder_new} follows  as in Section \ref{sec:H_linear}. Since no estimate of the temporal term is required here, we can use the simpler choice $N=2$ for the weight $W$. The linear part gives positive terms including
\begin{equation}\label{pos_2}
 \|\frac{R}{\sqrt{-u_{0yy}}}\|_{2}^2 + \varepsilon  \| \sqrt{\frac{u_0}{-u_{0yy}}} \nabla \Delta \phi \|_{2}^2.
\end{equation}
It remains to estimate the right-hand side of \eqref{eqn:remainder_new}. First,
\begin{equation}\label{cons_3}
|\varepsilon^{\frac{4}{3}-}(E, \frac{R}{-u_{0yy}} W)| \lesssim \varepsilon^{0+}  \|\frac{R}{\sqrt{-u_{0yy}}}\|_{2}^2.
\end{equation}
For the nonlinear terms, by  \eqref{est:mix_norm_H}, \eqref{est:mix_norm_Hardy}, and \eqref{est:mix_norm_misc}, we can estimate
\begin{equation}\label{cons_4}
\begin{aligned}
&\varepsilon^{\frac{4}{3}-} |\Big(\psi_y \Delta \psi_x, \frac{R}{-u_{0yy}} W \Big)|
 \le \varepsilon^{\frac{4}{3}-} |(\psi_y \Delta \psi_x, \frac{u_0 \Delta \phi}{-u_{0yy}}W)| + \varepsilon^{\frac{4}{3}-} |(\psi_y \Delta \psi_x, \phi W)|\\ 
\lesssim&  \varepsilon^{\frac{4}{3}-}\Big( \| \psi_y\|_{L^2_x L^\infty_y} \| \sqrt{\frac{u_0}{-u_{0yy}}}\Delta \psi_x\|_2 \| \sqrt{\frac{u_0}{-u_{0yy}}}\Delta \phi\|_{L^\infty_x L^2_y} +  \| \psi_y\|_{L^\infty_x L^2_y}\sqrt{\frac{u_0}{-u_{0yy}}}\Delta \psi_x\|_2 \| \frac{\phi}{\sqrt{u_0}}\|_{L^2_x L^\infty_y} \Big)\\
\lesssim &\varepsilon^{\frac{1}{3}-} \| \psi \|_{\cX_1}^2 \| \phi\|_{\cX_1}.
\end{aligned}
\end{equation}
Similarly, we can estimate
\begin{equation}\label{cons_5}
\varepsilon^{\frac{4}{3}-} |\Big(\psi_x \Delta \psi_y, \frac{R}{-u_{0yy}} W \Big)| \lesssim \varepsilon^{\frac{1}{3}-} \| \psi \|_{\cX_1}^2 \| \phi\|_{\cX_1}.
\end{equation}
Therefore, for $\psi \in \cB_2$, one can see that \eqref{cons_3}-\eqref{cons_5} can be absorbed by the positive contribution \eqref{pos_2} provided $\varepsilon$ is small enough. This implies that $\phi = \cS(\psi) \in \cB_2$. Similarly, one can show that $\cS : \cB_2 \to \cB_2$ is a contraction mapping. Therefore,  the existence and uniqueness of $\phi$ as the solution of \eqref{eqn:remainder} follow for the general $L$ case.

\subsection{Regularity of the steady state}

Finally, we verify that the constructed steady state has the desired regularity estimates \eqref{ss_estimates} or \eqref{ss_estimates_H}. Recall that our steady state is given by \eqref{ss_construction}. \eqref{Euler_estimate1}, \eqref{Euler_estimate2}, and \eqref{est:Prandtl} give the required pointwise control of the approximated component $\{u_a, v_a\}$. Thus the only remaining step is to bound the remainder $\{u,v\}$ constructed in Section \ref{sec:remainder}.

\subsubsection{Proof of \eqref{ss_estimates}}

First, we prove \eqref{ss_estimates} under the assumptions in Theorem \ref{Thm_Existence_L}. It suffices to prove
\begin{align}
&|\nabla \phi| \lesssim \varepsilon^{-\frac{1}{6}}, \qquad |\nabla \phi_x| \lesssim \varepsilon^{-(\frac{1}{3}+)}, \label{est:rem_pointwise}\\
&\| \sqrt{u_0} \nabla \Delta \phi \|_{L^\infty_x L^2_y} + \| \sqrt{u_0} \nabla \Delta \phi_x \|_{2} \lesssim \varepsilon^{-\frac{1}{2}}, \label{est:rem_weight}
\end{align}
where $\phi$ is the stream function for the remainder $\{u,v\}$.
We test the equation \eqref{eqn:remainder} by $-(u_0 \Delta \phi_x W)_x$, where $W = 2 - x/L$. Following the estimates in Section \ref{sec:L_nonlinear}, we will end up getting the uniform estimates
$$
\| u_0 \Delta \phi_x\|_2 + \sqrt{\varepsilon} \| \sqrt{u_0} \nabla\Delta \phi_x\|_2 \lesssim 1.
$$
This implies estimate \eqref{est:rem_weight} with Sobolev inequality in $x$. To estimate $\nabla \phi$, we use Sobolev inequality in $x$ and \eqref{est:mix_norm} to estimate
$$
|\nabla \phi| \lesssim \| \nabla \phi_x\|_{L^2_x L^\infty_y} \lesssim \varepsilon^{-\frac{1}{6}}.
$$
To estimate $\nabla \phi_x$, we follow the proofs of \eqref{est:nabla_phi_L_infty} and \eqref{est:Delta_phi_mix} with $\phi$ replaced by $\phi_x$ to obtain
\begin{align*}
|\nabla \phi_x| \lesssim_{q,\theta} \|\Delta \phi_x\|_{L^2_xL^2_y}^\theta \|\Delta \phi_x\|_{L^2_xL^q_y}^{1-\theta} \lesssim_{q,\theta} \|\Delta \phi_x\|_{L^2_xL^2_y}^\theta (\| \Delta \phi_x\|_2 + \| \sqrt{u_0} \Delta\phi_{xy}\|_2)^{1-\theta},
\end{align*}
where $q > 2$ and $\theta > 0$. Then estimate \eqref{est:rem_pointwise} follows by taking $\theta = 0+$ and \eqref{est:L_weighted_Hardy}.

\subsubsection{Proof of \eqref{ss_estimates_H}} Lastly, we prove \eqref{ss_estimates_H} under the assumptions in Theorem \ref{Thm_H_Existence}. It suffices to prove
\begin{align}
|\nabla \phi| \lesssim \varepsilon^{-(\frac{1}{3}+)}, \qquad \| \phi_{xy}\|_2 \lesssim \varepsilon^{-\frac{1}{3}},\qquad \| \sqrt{u_0} \nabla \Delta \phi \|_{2} \lesssim \varepsilon^{-\frac{1}{2}}, \label{est:rem_weight2}
\end{align}
where $\phi$ is the stream function for the remainder $\{u,v\}$. We test the equation \eqref{eqn:remainder} by $\frac{R}{-u_{0yy}}W$, where $W = 2L - x$. Following the estimates in Sections \ref{sec:H_linear} and \ref{sec:H_nonlinear}, we will end up getting the uniform estimates
$$
 \|\frac{R}{\sqrt{-u_{0yy}}}\|_{2} + \sqrt\varepsilon  \| \sqrt{\frac{u_0}{-u_{0yy}}} \nabla \Delta \phi \|_{2} \lesssim 1.
$$
Therefore, estimate \eqref{est:rem_weight2} follows from \eqref{est:nabla_phi_L_infty} and \eqref{est:H_weighted_Hardy}.


\appendix
\section{}\label{Sec:appendix}


In this Appendix, we provide a proof of \eqref{assumption_2} for the Poiseuille flow $\mu(y) = y(1-y)$ with $C_2 = 2$.

\begin{lemma}
Let $f$ be any $H^1$ function satisfying $f(1/2) = 0$, then
\begin{equation}\label{Hardy_goal}
\int_0^{1/2} f^2 \, dy \le 4 \int_0^{1/2} \mu^2 f_y^2 \, dy,
\end{equation}
where $\mu (y) = y(1-y)$. The constant $4$ is sharp.
\end{lemma}

\begin{proof}
First, by change of variable $x = 1-2y$, one can see that \eqref{Hardy_goal} is equivalent to
$$
\int_0^1 f^2 \,dx \le \int_0^1 (1-x^2)^2 f_x^2 \, dx,
$$
where $f(x)$ satisfies $f(0) = 0$. We further change $x = \cos \theta$, and denote $g(\theta) = f(x)$. Then \eqref{Hardy_goal} is equivalent to
\begin{equation}\label{Hardy_goal_g}
\int_0^{\frac{\pi}{2}} g^2 \cos \theta \,d\theta \le \int_0^{\frac{\pi}{2}} g_\theta^2 \cos^3 \theta \, d\theta,
\end{equation}
where $g(\theta)$ satisfies $g(0) = 0$. By the fundamental theorem of calculus and H\"older's inequality, we can estimate
\begin{align*}
g^2(\theta) &= \Big(\int_0^\theta g_\theta (s) \, ds \Big)^2\\
&\le \int_0^\theta |g_\theta (s)|^2 \cos^2 s \, ds \int_0^\theta \frac{1}{\cos^2 s} \, ds\\
&= \int_0^\theta |g_\theta (s)|^2 \cos^2 s \, ds \tan \theta.
\end{align*}
Therefore,
\begin{align*}
\int_0^{\frac{\pi}{2}} g^2 \cos \theta \,d\theta & \le \int_0^{\frac{\pi}{2}} \int_0^\theta |g_\theta (s)|^2 \cos^2 s \, ds \sin \theta \, d\theta\\
& = \int_0^{\frac{\pi}{2}}    |g_\theta (s)|^2 \cos^2 s \int_s^{\frac{\pi}{2}} \sin \theta \, d\theta ds\\
& = \int_0^{\frac{\pi}{2}}    |g_\theta (s)|^2 \cos^3 s\, ds,
\end{align*}
which is exactly \eqref{Hardy_goal_g}. To prove the sharpness, for any small $\delta > 0$, we define
$$
g_\delta(\theta): = \left\{
\begin{aligned}
&\tan \theta, && \theta \in (0, \frac{\pi}{2} - \delta),\\
&\tan (\frac{\pi}{2} - \delta),  && \theta \in [\frac{\pi}{2} - \delta, \frac{\pi}{2}).
\end{aligned}
\right.
$$
Since $\tan (\frac{\pi}{2} - \delta) \sim \frac{1}{\delta}$ when $\delta$ is small, the left-hand side of \eqref{Hardy_goal_g} becomes
\begin{align*}
\int_0^{\frac{\pi}{2}} g_\delta^2 \cos \theta \,d\theta &= \int_0^{\frac{\pi}{2} - \delta} \frac{\sin^2 \theta}{\cos \theta} \, d\theta + \tan^2 (\frac{\pi}{2} - \delta) \int_{\frac{\pi}{2} - \delta}^{\frac{\pi}{2}} \cos \theta \, d\theta\\
&= \int_0^{\frac{\pi}{2} - \delta} \frac{1}{\cos \theta} \, d\theta -  \int_0^{\frac{\pi}{2} - \delta} \cos \theta \, d\theta + \tan^2 (\frac{\pi}{2} - \delta) O(\delta^2)\\
&= \int_0^{\frac{\pi}{2} - \delta} \frac{1}{\cos \theta} \, d\theta + O(1).
\end{align*}
Moreover,
\begin{align*}
\int_0^{\frac{\pi}{2} - \delta} \frac{1}{\cos \theta} \, d\theta = \int_0^{\frac{\pi}{2} - \delta} \frac{\cos \theta}{1 - \sin^2 \theta} \, d\theta= \int_0^{\sin (\frac{\pi}{2} - \delta)} \frac{1}{1-x^2} \, dx = O(|\log \delta|).
\end{align*}
Note that for the right-hand side of \eqref{Hardy_goal_g}, we also have
$$
 \int_0^{\frac{\pi}{2}} (g_\delta')^2 \cos^3 \theta \, d\theta = \int_0^{\frac{\pi}{2} - \delta} \frac{1}{\cos \theta} \, d\theta = O(|\log \delta|).
$$
Therefore, we can conclude that for any $\varepsilon > 0$, we can find a small $\delta > 0$ such that
$$
\int_0^{\frac{\pi}{2}} g_\delta^2 \cos \theta \,d\theta > (1-\varepsilon)  \int_0^{\frac{\pi}{2}} (g_\delta')^2 \cos^3 \theta \, d\theta.
$$
This confirms the sharpness of the inequality \eqref{Hardy_goal}.
\end{proof}

\bibliographystyle{amsplain}

\begin{bibdiv}
\begin{biblist}



\bib{Bedrossian-Germain-Masmoudi17}{article}{
      author={Bedrossian, Jacob},
      author={Germain, Pierre},
      author={Masmoudi, Nader},
       title={On the stability threshold for the 3{D} {C}ouette flow in
  {S}obolev regularity},
        date={2017},
        ISSN={0003-486X,1939-8980},
     journal={Ann. of Math. (2)},
      volume={185},
      number={2},
       pages={541\ndash 608},
         url={https://doi.org/10.4007/annals.2017.185.2.4},
      review={\MR{3612004}},
}

\bib{Bedrossian-Germain-Masmoudi19}{article}{
      author={Bedrossian, Jacob},
      author={Germain, Pierre},
      author={Masmoudi, Nader},
       title={Stability of the {C}ouette flow at high {R}eynolds numbers in two
  dimensions and three dimensions},
        date={2019},
        ISSN={0273-0979,1088-9485},
     journal={Bull. Amer. Math. Soc. (N.S.)},
      volume={56},
      number={3},
       pages={373\ndash 414},
         url={https://doi.org/10.1090/bull/1649},
      review={\MR{3974608}},
}

\bib{Bedrossian-Germain-Masmoudi20}{article}{
      author={Bedrossian, Jacob},
      author={Germain, Pierre},
      author={Masmoudi, Nader},
       title={Dynamics near the subcritical transition of the 3{D} {C}ouette
  flow {I}: {B}elow threshold case},
        date={2020},
        ISSN={0065-9266,1947-6221},
     journal={Mem. Amer. Math. Soc.},
      volume={266},
      number={1294},
       pages={v+158},
         url={https://doi.org/10.1090/memo/1294},
      review={\MR{4126259}},
}

\bib{Bedrossian-Germain-Masmoudi22}{article}{
      author={Bedrossian, Jacob},
      author={Germain, Pierre},
      author={Masmoudi, Nader},
       title={Dynamics near the subcritical transition of the 3{D} {C}ouette
  flow {II}: {A}bove threshold case},
        date={2022},
        ISSN={0065-9266,1947-6221},
     journal={Mem. Amer. Math. Soc.},
      volume={279},
      number={1377},
       pages={v+135},
         url={https://doi.org/10.1090/memo/1377},
      review={\MR{4458538}},
}



\bib{Bedrossian-He-Iyer-Li-Wang25}{article}{
      author={Bedrossian, Jacob},
      author={He, Siming},
      author={Iyer, Sameer},
      author={Li, Linfeng},
      author={Wang, Fei},
       title={Stability threshold of close-to-{Couette} shear flows with
  no-slip boundary conditions in 2d},
        date={2025},
         url={https://arxiv.org/abs/2510.16378},
        note={Preprint, {arXiv}:2510.16378 [math.{AP}]},
}

\bib{Bedrossian-He-Iyer-Wang25}{article}{
      author={Bedrossian, Jacob},
      author={He, Siming},
      author={Iyer, Sameer},
      author={Wang, Fei},
       title={Stability threshold of nearly-{C}ouette shear flows with {N}avier
  boundary conditions in 2{D}},
        date={2025},
        ISSN={0010-3616,1432-0916},
     journal={Comm. Math. Phys.},
      volume={406},
      number={2},
       pages={Paper No. 28, 42},
         url={https://doi.org/10.1007/s00220-024-05175-4},
      review={\MR{4848788}},
}

\bib{Bedrossian-Vicol-Wang18}{article}{
      author={Bedrossian, Jacob},
      author={Vicol, Vlad},
      author={Wang, Fei},
       title={The {S}obolev stability threshold for 2{D} shear flows near
  {C}ouette},
        date={2018},
        ISSN={0938-8974,1432-1467},
     journal={J. Nonlinear Sci.},
      volume={28},
      number={6},
       pages={2051\ndash 2075},
         url={https://doi.org/10.1007/s00332-016-9330-9},
      review={\MR{3867637}},
}

\bib{Beekie-Chen-Jia26}{article}{
      author={Beekie, Rajendra},
      author={Chen, Shan},
      author={Jia, Hao},
       title={Uniform {V}orticity {D}epletion and {I}nviscid {D}amping for
  {P}eriodic {S}hear {F}lows in the {H}igh {R}eynolds {N}umber {R}egime},
        date={2026},
        ISSN={0003-9527,1432-0673},
     journal={Arch. Ration. Mech. Anal.},
      volume={250},
      number={1},
       pages={Paper No. 7},
         url={https://doi.org/10.1007/s00205-025-02162-4},
      review={\MR{5006587}},
}

\bib{Bian-Grenier-Masmoudi-Zhao25}{article}{
      author={Bian, Dongfen},
      author={Grenier, Emmanuel},
      author={Masmoudi, Nader},
      author={Zhao, Weiren},
       title={Boundary driven instabilities of {C}ouette flows},
        date={2025},
        ISSN={0010-3616,1432-0916},
     journal={Comm. Math. Phys.},
      volume={406},
      number={9},
       pages={Paper No. 221, 23},
         url={https://doi.org/10.1007/s00220-025-05401-7},
      review={\MR{4940211}},
}

\bib{Chen-Ding-Lin-Zhang23}{article}{
      author={Chen, Qi},
      author={Ding, Shijin},
      author={Lin, Zhilin},
      author={Zhang, Zhifei},
       title={Nonlinear stability for 3-{D} plane {Poiseuille} flow in a finite
  channel},
        date={2023},
         url={https://arxiv.org/abs/2310.11694},
        note={Preprint, {arXiv}:2310.11694 [math.{AP}]},
}

\bib{Chen-Jia-Wei-Zhang25}{article}{
      author={Chen, Qi},
      author={Jia, Hao},
      author={Wei, Dongyi},
      author={Zhang, Zhifei},
       title={Asymptotic stability of the {Kolmogorov} flow at high {Reynolds}
  numbers},
        date={2025},
         url={https://arxiv.org/abs/2510.13181},
        note={Preprint, {arXiv}:2510.13181 [math.{AP}]},
}

\bib{Chen-Li-Shen-Zhang25}{article}{
      author={Chen, Qi},
      author={Li, Hao},
      author={Shen, Shunlin},
      author={Zhang, Zhifei},
       title={Asymptotic stability of the symmetric flow via inviscid damping
  and enhanced dissipation},
        date={2025},
         url={https://arxiv.org/abs/2510.18361},
        note={Preprint, {arXiv}:2510.18361 [math.{AP}]},
}

\bib{Chen-Li-Wei-Zhang20}{article}{
      author={Chen, Qi},
      author={Li, Te},
      author={Wei, Dongyi},
      author={Zhang, Zhifei},
       title={Transition threshold for the 2-{D} {C}ouette flow in a finite
  channel},
        date={2020},
        ISSN={0003-9527,1432-0673},
     journal={Arch. Ration. Mech. Anal.},
      volume={238},
      number={1},
       pages={125\ndash 183},
         url={https://doi.org/10.1007/s00205-020-01538-y},
      review={\MR{4121130}},
}

\bib{Chen-Wei-Zhang23a}{article}{
      author={Chen, Qi},
      author={Wei, Dongyi},
      author={Zhang, Zhifei},
       title={Linear inviscid damping and enhanced dissipation for monotone
  shear flows},
        date={2023},
        ISSN={0010-3616,1432-0916},
     journal={Comm. Math. Phys.},
      volume={400},
      number={1},
       pages={215\ndash 276},
         url={https://doi.org/10.1007/s00220-022-04597-2},
      review={\MR{4581474}},
}

\bib{Chen-Wei-Zhang23b}{article}{
      author={Chen, Qi},
      author={Wei, Dongyi},
      author={Zhang, Zhifei},
       title={Linear stability of pipe {Poiseuille} flow at high {Reynolds}
  number regime},
        date={2023},
        ISSN={0010-3640},
     journal={Comm. Pure Appl. Math.},
      volume={76},
      number={9},
       pages={1868\ndash 1964},
}

\bib{CotiZelati-Elgindi-Widmayer20}{article}{
      author={Coti~Zelati, Michele},
      author={Elgindi, Tarek~M.},
      author={Widmayer, Klaus},
       title={Enhanced dissipation in the {N}avier-{S}tokes equations near the
  {P}oiseuille flow},
        date={2020},
        ISSN={0010-3616,1432-0916},
     journal={Comm. Math. Phys.},
      volume={378},
      number={2},
       pages={987\ndash 1010},
         url={https://doi.org/10.1007/s00220-020-03814-0},
      review={\MR{4134940}},
}



\bib{Zotto23}{article}{
      author={Del~Zotto, Augusto},
       title={Enhanced dissipation and transition threshold for the
  {P}oiseuille flow in a periodic strip},
        date={2023},
        ISSN={0036-1410,1095-7154},
     journal={SIAM J. Math. Anal.},
      volume={55},
      number={5},
       pages={4410\ndash 4424},
         url={https://doi.org/10.1137/21M1444011},
      review={\MR{4641646}},
}

\bib{Deng-Masmoudi23}{article}{
      author={Deng, Yu},
      author={Masmoudi, Nader},
       title={Long-time instability of the {C}ouette flow in low {G}evrey
  spaces},
        date={2023},
        ISSN={0010-3640,1097-0312},
     journal={Comm. Pure Appl. Math.},
      volume={76},
      number={10},
       pages={2804\ndash 2887},
         url={https://doi.org/10.1002/cpa.22092},
      review={\MR{4630602}},
}

\bib{Ding-Lin25}{article}{
      author={Ding, Shijin},
      author={Lin, Zhilin},
       title={Stability for the 2-{D} plane {P}oiseuille flow in a finite
  channel},
        date={2025},
        ISSN={1674-7283,1869-1862},
     journal={Sci. China Math.},
      volume={68},
      number={10},
       pages={2333\ndash 2346},
  url={https://doi-org.proxy.lib.ohio-state.edu/10.1007/s11425-024-2382-4},
      review={\MR{4962641}},
}



\bib{Grenier-Guo-Toan16}{article}{
      author={Grenier, Emmanuel},
      author={Guo, Yan},
      author={Nguyen, Toan~T.},
       title={Spectral instability of general symmetric shear flows in a
  two-dimensional channel},
        date={2016},
        ISSN={0001-8708,1090-2082},
     journal={Adv. Math.},
      volume={292},
       pages={52\ndash 110},
         url={https://doi.org/10.1016/j.aim.2016.01.007},
      review={\MR{3464020}},
}


\bib{Grenier-Toan19}{article}{
      author={Grenier, Emmanuel},
      author={Nguyen, Toan~T.},
       title={{$L^\infty$} instability of {P}randtl layers},
        date={2019},
        ISSN={2524-5317,2199-2576},
     journal={Ann. PDE},
      volume={5},
      number={2},
       pages={Paper No. 18, 36},
  url={https://doi-org.proxy.lib.ohio-state.edu/10.1007/s40818-019-0074-3},
      review={\MR{4038143}},
}


\bib{GuoIyer21}{article}{
      author={Guo, Yan},
      author={Iyer, Sameer},
       title={Regularity and expansion for steady {P}randtl equations},
        date={2021},
        ISSN={0010-3616,1432-0916},
     journal={Comm. Math. Phys.},
      volume={382},
      number={3},
       pages={1403\ndash 1447},
         url={https://doi.org/10.1007/s00220-021-03964-9},
      review={\MR{4232771}},
}

\bib{GuoIyer23}{article}{
      author={Guo, Yan},
      author={Iyer, Sameer},
       title={Validity of steady {P}randtl layer expansions},
        date={2023},
        ISSN={0010-3640,1097-0312},
     journal={Comm. Pure Appl. Math.},
      volume={76},
      number={11},
       pages={3150\ndash 3232},
      review={\MR{4642817}},
}

\bib{IyerMasmoudi21a}{article}{
      author={Iyer, Sameer},
      author={Masmoudi, Nader},
       title={Global-in-$x$ stability of steady prandtl expansions for 2{D}
  {N}avier-{S}tokes flows},
        date={2021},
        note={Preprint, arXiv:2008.12347 [math.{AP}]},
}

\bib{IyerZhou}{article}{
      author={Iyer, Sameer},
      author={Zhou, Chunhui},
       title={Stationary inviscid limit to shear flows},
        date={2019},
        ISSN={0022-0396,1090-2732},
     journal={J. Differential Equations},
      volume={267},
      number={12},
       pages={7135\ndash 7153},
         url={https://doi.org/10.1016/j.jde.2019.07.017},
      review={\MR{4011041}},
}

\bib{Li-Wei-Zhang20}{article}{
      author={Li, Te},
      author={Wei, Dongyi},
      author={Zhang, Zhifei},
       title={Pseudospectral bound and transition threshold for the 3{D}
  {K}olmogorov flow},
        date={2020},
        ISSN={0010-3640,1097-0312},
     journal={Comm. Pure Appl. Math.},
      volume={73},
      number={3},
       pages={465\ndash 557},
         url={https://doi.org/10.1002/cpa.21863},
      review={\MR{4057900}},
}

\bib{Masmoudi-Zhao22}{article}{
      author={Masmoudi, Nader},
      author={Zhao, Weiren},
       title={Stability threshold of two-dimensional {C}ouette flow in
  {S}obolev spaces},
        date={2022},
        ISSN={0294-1449,1873-1430},
     journal={Ann. Inst. H. Poincar\'e{} C Anal. Non Lin\'eaire},
      volume={39},
      number={2},
       pages={245\ndash 325},
         url={https://doi.org/10.4171/aihpc/8},
      review={\MR{4412070}},
}

\bib{Trefethen}{article}{
      author={Trefethen, Lloyd~N.},
      author={Trefethen, Anne~E.},
      author={Reddy, Satish~C.},
      author={Driscoll, Tobin~A.},
       title={Hydrodynamic stability without eigenvalues},
        date={1993},
        ISSN={0036-8075,1095-9203},
     journal={Science},
      volume={261},
      number={5121},
       pages={578\ndash 584},
         url={https://doi.org/10.1126/science.261.5121.578},
      review={\MR{1229495}},
}

\bib{Wei-Zhang-Zhao19}{article}{
      author={Wei, Dongyi},
      author={Zhang, Zhifei},
      author={Zhao, Weiren},
       title={Linear inviscid damping and vorticity depletion for shear flows},
        date={2019},
        ISSN={2524-5317,2199-2576},
     journal={Ann. PDE},
      volume={5},
      number={1},
       pages={Paper No. 3, 101},
         url={https://doi.org/10.1007/s40818-019-0060-9},
      review={\MR{3919496}},
}

\bib{Wei-Zhang-Zhao20}{article}{
      author={Wei, Dongyi},
      author={Zhang, Zhifei},
      author={Zhao, Weiren},
       title={Linear inviscid damping and enhanced dissipation for the
  {K}olmogorov flow},
        date={2020},
        ISSN={0001-8708,1090-2082},
     journal={Adv. Math.},
      volume={362},
       pages={106963, 103},
         url={https://doi.org/10.1016/j.aim.2019.106963},
      review={\MR{4050586}},
}

\end{biblist}
\end{bibdiv}

\end{document}